\documentclass[12pt,fullpage]{article}

\usepackage{amssymb}
\usepackage{amsmath}
\usepackage{cases}
\usepackage{amsfonts}
\usepackage{amsmath,amssymb}
\usepackage{multicol}
\usepackage{color}
\usepackage{graphicx}
\usepackage{graphics}
\def\R{\mathbb{R}}
\def\N{\mathbb{N}}

\def\epsilon{\varepsilon}
\renewcommand\tilde{\widetilde}
\newcommand{\be}{\begin{equation}}
\newcommand{\ee}{\end{equation}}
\newcommand{\baa}{\begin{array}}
\newcommand{\eaa}{\end{array}}
\newcommand{\ba}{\begin{eqnarray}}
\newcommand{\ea}{\end{eqnarray}}

\newtheorem{theorem}{Theorem}[section]
\newtheorem{lemma}[theorem]{Lemma}
\newtheorem{corollary}[theorem]{Corollary}

\newtheorem{definition}[theorem]{Definition}
\newtheorem{remark}[theorem]{Remark}
\newtheorem{conjecture}[theorem]{Conjecture}

\numberwithin{equation}{section}
\newenvironment{proof}[1][Proof]{\noindent\textbf{#1.} }{\hfill $\Box$}
\allowdisplaybreaks

\begin{document}
\date{}
\title{\bf{On the mean speed of bistable transition fronts in unbounded domains}\thanks{This work has been carried out in the framework of Archim\`ede Labex of Aix-Marseille University. The project leading to this publication has received funding from Excellence Initiative of Aix-Marseille University~-~A*MIDEX, a French ``Investissements d'Avenir" programme, from the European Research Council under the European Union's Seventh Framework Programme (FP/2007-2013) ERC Grant Agreement n.~321186~- ReaDi~- Reaction-Diffusion Equations, Propagation and Modelling and from the ANR NONLOCAL project (ANR-14-CE25-0013). H. Guo was supported by the China Scholarship Council for 3 years of study at Aix-Marseille University. W.J. Sheng was supported by NSF of China (11401134), by postdoctoral scientific research developmental fund of Heilongjiang Province (LBH-Q17061) and the China Scholarship Council for a one-year visit at Aix-Marseille University.}}
\author{Hongjun Guo \textsuperscript{a}, Fran\c{c}ois Hamel \textsuperscript{a} and Wei-Jie Sheng \textsuperscript{b}\\
\\
\footnotesize{\textsuperscript{a} Aix Marseille Univ, CNRS, Centrale Marseille, I2M, Marseille, France}\\
\footnotesize{\textsuperscript{b} Department of Mathematics, Harbin Institute of Technology, Harbin, China}}
\maketitle

\begin{abstract}\noindent{}This paper is concerned with the existence and further properties of propagation speeds of transition fronts for bistable reaction-diffusion equations in exterior domains and in some domains with multiple cylindrical branches. In exterior domains we show that all transition fronts with complete propagation propagate with the same global mean speed, which turns out to be equal to the uniquely defined planar speed. In domains with multiple cylindrical branches, we show that the solutions emanating from some branches and propagating completely are transition fronts propagating with the unique planar speed. We also give some geometrical and scaling conditions on the domain, either exterior or with multiple cylindrical branches, which guarantee that any transition front has a global mean speed.
\noindent{}
\end{abstract}

\tableofcontents


\section{Introduction and main results}

This paper is concerned with propagation phenomena for semilinear parabolic reaction-diffusion equations of the type
\be\label{eq1.1}\left\{\begin{array}{rcll}
u_t & = & \Delta u+f(u), & t\in\R,\ x\in\Omega,\vspace{3pt}\\
u_{\nu} & = & 0, & t\in\R,\ x\in\partial\Omega,\end{array}\right.
\ee
set in smooth unbounded open connected sets (domains) $\Omega$ of $\R^N$ with $N\ge 2$ (more precise assumptions on~$\Omega$ will be made later). Here, $u_t=\frac{\partial u}{\partial t}$,  $\nu=\nu(x)$ is the outward unit normal on the boundary~$\partial \Omega$ and $u_{\nu}=\frac{\partial u}{\partial \nu}$. Neumann no-flux boundary conditions are thus imposed on $\partial \Omega$. Such equations are ubiquitous in many models in biology, ecology and physics. The quantity $u(t,x)$ is assumed to be bounded, and then to range in the interval $[0,1]$ without loss of generality. It may for instance stand for the density of a species at time $t$ and location $x$.

The reaction term $f$ is assumed to be of class $C^{1,\beta}([0,1],\R)$ (with $\beta>0$) with two stable zeroes~$0$ and~$1$, that is,
\begin{align}\label{F1}
f(0)=f(1)=0,\ \ f'(0)<0\ \hbox{ and }\ f'(1)<0.
\end{align}
Having in mind that we are interested in propagating solutions $u$ of~\eqref{eq1.1} connecting the two steady states $0$ and $1$ in some unbounded domains $\Omega$, we assume throughout the paper that the states $0$ and $1$ can at least be connected by a planar traveling front, that is, equation~\eqref{eq1.1}, if set in $\Omega=\R$, admits a traveling front $\phi(x-c_ft)$ solving
\be\label{phi-eq}
\phi''+c_f\phi'+f(\phi)=0\ \text{in $\R$},\ \ 0<\phi<1\hbox{ in }\R,\ \ \phi(-\infty)=1\text{ and }\phi(+\infty)=0,
\ee
with $c_f\neq0$ (the front is not stationary, it is truly propagating). Without loss of generality, even if it means changing $\phi(x-c_ft)$ into $1-\phi(-x-c_ft)$ and $f(s)$ into $-f(1-s)$, one can then assume that
\be\label{F2}
c_f>0.
\ee
By multiplying~\eqref{phi-eq} by $\phi'$ and integrating over $(-\infty,\xi)$ with an arbitrary $\xi\in(-\infty,+\infty]$, condition~\eqref{F2} also implies that $\int_\eta^1 f(s)ds>0$ for every $0\le\eta<1$. The front $\phi(x-c_ft)$ describes the invasion of the state $0$ by the state $1$, with constant propagation speed $c_f>0$ and constant profile~$\phi$. Actually, with condition~\eqref{F1}, the speed $c_f$ of~\eqref{phi-eq} is unique and the profile~$\phi$ is unique up to shifts, see~\cite{AW,FM}. {\it The assumptions~\eqref{F1}-\eqref{F2} are made thourghout the paper and are therefore not repeated in the statements of the results}.

For a function $f$ satisfying~\eqref{F1}, condition~\eqref{phi-eq} is fulfilled if $f$ is bistable, that is, there exists $\theta\in(0,1)$ such that
\be\label{F3}
f<0\hbox{ on }(0,\theta)\ \hbox{ and }\ f>0\hbox{ on }(\theta,1),
\ee
see~\cite{AW,FM}. In particular, an important case~of a function $f$ satisfying~\eqref{F1}-\eqref{F2} is the cubic nonlinea\-rity $f(u)=u(1-u)(u-\theta)$ with $\theta\in (0,1/2)$. Other sufficient conditions for~\eqref{phi-eq} to hold for multistable functions $f$ having more than one oscillation in the interval $(0,1)$ are given in~\cite{FM}. Lastly, for mathematical purposes, we extend the function $f$ in $\R\backslash [0,1]$ as follows: $f(u)=f'(0)u>0$ for $u\in (-\infty,0)$ and $f(u)=f'(1)(u-1)<0$ for $u\in (1,+\infty)$.

The paper deals with propagation phenomena for~\eqref{eq1.1} in various types of domains $\Omega$. Since the steady state $1$ is in some sense more stable than the steady state $0$ due to~\eqref{F2}, the region where $u$ is close to $1$ is expected to invade the region where $u$ is close to $0$. Propagating solutions connecting the stable steady states $0$ and $1$ play an important role in the dynamics of~\eqref{eq1.1} and the characterization of their propagation speed is a fundamental issue. This paper deals with this issue and in particular with the understanding of the role of the geometry of the underlying domain $\Omega$. In the one-dimensional line $\Omega=\R$, standard traveling fronts $\phi(x-ct)$ connecting~$0$ and~$1$ and propagating with constant speed $c=c_f$ are assumed to exist by~\eqref{phi-eq}. However, for general domains $\Omega\subset\R^N$, standard planar traveling fronts of the type $\phi(x\cdot e-ct)$ for some unit vector $e$ and some speed $c\in\R$ do not exist anymore. The shape of the level sets of the solutions $u$ of~\eqref{eq1.1} depend on $\Omega$ and may be much more complex in general. To describe propagating solutions connecting the two stable states $0$ and $1$ in arbitrary unbounded domains~$\Omega$, we therefore have to consider the recently introduced and more general notions of transition fronts and global mean speed.


\subsection{Notions of transition fronts and global mean speed}\label{sec1.1}

Before giving the general definition of transition fronts connecting the steady states $0$ and $1$ in general domains $\Omega$, let us first recall some well-known results about the reaction-diffusion equation~\eqref{eq1.1} in the whole space $\R^N$. By assumption~\eqref{phi-eq}, equation~\eqref{eq1.1} set in $\R$ admits standard traveling fronts $u(t,x)=\phi(x-c_f t)$. In higher dimensions $\Omega=\R^N$ with $N\ge 2$, planar fronts $\phi(x\cdot e-c_ft)$ moving with constant speed $c_f$ still exist, where $e\in\mathbb{S}^{N-1}$ is any unit vector of $\R^N$. Even in the homogeneous space $\Omega=\R^N$, under assumptions~\eqref{F1}-\eqref{F3}, many other types of solutions of~\eqref{eq1.1} exist, such as axisymmetric conical-shaped, resp. pyramidal, fronts $\psi(x-cte)$, where $e\in\mathbb{S}^{N-1}$, $c\ge c_f$ and the level sets of the function $\psi:\R^N\to(0,1)$ are invariant by rotation around the vector $e$, resp. have a pyramidal shape, see~\cite{HM,HMR1,HMR2,HS,NT1,T1,T2,T3}. These solutions are all invariant with respect to time in a moving frame. Let us also mention that standard traveling fronts $\psi(x-cte)$, with other shapes, still exist when~$f$ is balanced, that is, $\int_0^1\!f(s)ds=0$, see~\cite{CGHNR,DKW,G1}.

In general domains $\Omega$ without any invariance by translation, standard traveling fronts do not exist anymore in general. Instead, there may still exist front-like solutions which behave asymptotically as standard traveling fronts in some sense. For instance, in exterior domains and in some cylinder-like domains, there may exist front-like solutions $u(t,x)$ satisfying
$$u(t,x)-\phi(x\cdot e-c_f t)\to0\ \text{ as $t\rightarrow -\infty$ uniformly in $\overline{\Omega}$},$$
where $e$ is the direction of propagation for very negative times, see~\cite{BBC,BHM,CG,P}. We will come back to these solutions later.

Considering the various types of known existing fronts and the generality of the underlying domains $\Omega$, unifying notions of generalized traveling fronts, namely the transition fronts,  have been introduced in~\cite{BH1,BH2} (see also~\cite{S} in the one-dimensional setting). In order to recall the notion of transition fronts and that of their global mean speed, let us introduce a few notations. The unbounded open connected set $\Omega\subset \R^N$ is assumed to have a globally $C^{2,\beta}$ boundary with $\beta>0$ (this is what we call a smooth domain throughout the paper), that is, there exist $\rho>0$ and $C>0$ such that, for every $y\in\partial \Omega$, there are a rotation $R_y$ of $\R^N$ and a $C^{2,\beta}$ map $\psi_y:\bar{B}=\big\{x'\in\R^{N-1}: |x'|\le2\rho\big\}\rightarrow \R$ such that $\psi_y(0)=0$, $\|\psi_y\|_{C^{2,\beta}(\bar{B})}\le C$ and
$$\Omega\cap B(y,\rho)=\left[y+R_y\big(\{x=(x',x_N)\in\R^N: x'\in\bar{B}, x_N>\psi_y(x')\}\big)\right]\cap B(y,\rho),$$
where
$$B(y,\rho)=\big\{x\in\R^N: |x-y|<\rho\big\}$$
and $|\ \ |$ denotes the Euclidean norm. Let $d_{\Omega}$ be the geodesic distance in $\overline{\Omega}$. For any two subsets~$A$ and~$B$ of~$\overline{\Omega}$, we set
\begin{equation*}
d_{\Omega}(A,B)=\inf\big\{d_{\Omega}(x,y): (x,y)\in A\times B\big\},
\end{equation*}
and $d_{\Omega}(x,A)=d_{\Omega}(\{x\},A)$ for $x\in\R^N$.

Consider now two families $(\Omega_t^-)_{t\in \mathbb{R}}$ and $(\Omega_t^+)_{t\in \mathbb{R}}$ of open non-empty subsets of $\Omega$ such that
\begin{eqnarray}\label{eq1.3}
\forall t\in \mathbb{R},\ \ \left\{\begin{array}{l}
\Omega_t^-\cap \Omega_t^+=\emptyset,\vspace{3pt}\\
\partial \Omega_t^-\cap \Omega=\partial \Omega_t^+\cap \Omega=:\Gamma_t,\vspace{3pt}\\
\Omega_t^-\cup \Gamma_t \cup \Omega_t^+=\Omega,\vspace{3pt}\\
\sup\big\{d_{\Omega}(x,\Gamma_t): x\in \Omega_t^+\big\}=\sup\big\{d_{\Omega}(x,\Gamma_t): x\in \Omega_t^-\big\}=+\infty\end{array}\right.
\end{eqnarray}
and
\begin{eqnarray}\label{eq1.4}
\left\{\begin{array}{l}
\inf\Big\{\sup\big\{d_{\Omega}(y,\Gamma_t): y\in \Omega_t^+, d_{\Omega}(y,x)\leq r\big\}: t\in \mathbb{R},\ x\in \Gamma_t\Big\}\rightarrow +\infty\vspace{3pt}\\
\inf\Big\{\sup\big\{d_{\Omega}(y,\Gamma_t): y\in \Omega_t^-, d_{\Omega}(y,x)\leq r\big\}: t\in \mathbb{R},\ x\in \Gamma_t\Big\}\rightarrow +\infty\end{array}\right.
\text{ as}\ \ r\rightarrow +\infty.
\end{eqnarray}
Notice that the condition~\eqref{eq1.3} implies in particular that the interface $\Gamma_t$ is not empty for every~$t\in \mathbb{R}$. As far as~\eqref{eq1.4} is concerned, it says that for any $M>0$, there is $r_M>0$ such that, for every $t\in \mathbb{R}$ and~$x\in \Gamma_t$, there are $y^{\pm}\in \mathbb{R}^N$ such that
\begin{eqnarray}\label{eq1.5}
y^{\pm}\in \Omega^{\pm}_t,\ \ d_{\Omega}(x,y^{\pm})\leq r_M\ \ \text{and}\ \ d_{\Omega}(y^{\pm},\Gamma_t)\geq M.
\end{eqnarray}
In other words, condition~\eqref{eq1.4} means that any point on $\Gamma_t$ is not too far from the centers of two large balls (in the sense of the geodesic distance in $\Omega$) included in $\Omega^-_t$ and $\Omega^+_t$, this property being in some sense uniform with respect to $t$ and to the point on $\Gamma_t$ (without loss of generality, one can also assume that the map $M\mapsto r_M$ is non-decreasing). Moreover, in order to avoid interfaces with infinitely many twists, the sets $\Gamma_t$ are assumed to be included in finitely many graphs: there is an integer $n\geq 1$ such that, for each $t\in \mathbb{R}$, there are $n$ open subsets $\omega_{i,t}\subset \mathbb{R}^{N-1}$(for $1\leq i\leq n$), $n$ continuous maps $\psi_{i,t}: \omega_{i,t}\rightarrow \mathbb{R}$ and $n$ rotations $R_{i,t}$ of $\mathbb{R}^N$, with
\begin{equation}\label{eq1.6}
\Gamma_t \subset \bigcup_{1\leq i\leq n} R_{i,t}\left(\big\{x=(x',x_N)\in \mathbb{R}^N: x'\in \omega_{i,t},\ x_N=\psi_{i,t}(x')\big\}\right).
\end{equation}

\begin{definition}\label{TF}{\rm{\cite{BH1,BH2}}}
For problem~\eqref{eq1.1}, a transition front connecting $0$ and $1$ is a classical solution $u:\mathbb{R}\times\overline{\Omega} \rightarrow (0,1)$ for which there exist some sets $(\Omega_t^{\pm})_{t\in \mathbb{R}}$ and $(\Gamma_t)_{t\in \mathbb{R}}$ satisfying~\eqref{eq1.3}-\eqref{eq1.6}, and, for every $\varepsilon>0$, there exists $M_{\varepsilon}>0$ such that
\begin{eqnarray}\label{eq1.7}
\left\{\baa{l}
\forall t\in \mathbb{R},\ \ \forall x\in \overline{\Omega_t^+}, \ \ \left(d_{\Omega}(x,\Gamma_t)\geq M_{\varepsilon}\right)\Rightarrow \left(u(t,x)\geq 1-\varepsilon\right)\!,\vspace{3pt}\\
\forall t\in \mathbb{R},\ \ \forall x\in \overline{\Omega_t^-}, \ \ \left(d_{\Omega}(x,\Gamma_t)\geq M_{\varepsilon}\right)\Rightarrow \left(u(t,x)\leq \varepsilon\right)\!.\eaa\right.
\end{eqnarray}
Furthermore, $u$ is said to have a global mean speed $\gamma$ $(\geq 0)$ if
$$\frac{d_{\Omega}(\Gamma_t,\Gamma_s)}{|t-s|}\rightarrow \gamma \ \ \text{as}\ \ |t-s|\rightarrow +\infty.$$
\end{definition}

This definition has been shown in~\cite{BH1,BH2,H} to cover and unify all classical cases. Condition~\eqref{eq1.7} somehow means that the transition between the limiting steady states $0$ and $1$ takes place in some uniformly-bounded-in-time neighborhoods of $\Gamma_t$. Without this uniformity condition, other solutions may exist: for instance, for the equation~\eqref{eq1.1} in $\R$ under assumptions~\eqref{F1}-\eqref{F3}, there are solutions $u(t,x)$ converging to $0$ and $1$ as $x\to\pm\infty$ for each time~$t\in\R$, but which do not satisfy~\eqref{eq1.7} and thus are not transition fronts, see~\cite{MN}. Notice that, for a given transition front connecting $0$ and $1$, the families $(\Omega^{\pm}_t)_{t\in\R}$ and $(\Gamma_t)_{t\in\R}$ satisfying~\eqref{eq1.7} are not unique, since any uniformly-bounded-in-time thickening or thinning of $\Omega^-_t$ or $\Omega^+_t$ satisfies~\eqref{eq1.7} too. However, for a given transition front satisfying~\eqref{eq1.7} with some families $(\Omega^{\pm}_t)_{t\in\R}$, $(\Gamma_t)_{t\in\R}$ and $(\tilde{\Omega}^{\pm}_t)_{t\in\R}$, $(\tilde{\Gamma}_t)_{t\in\R}$, the Hausdorff distance between the interfaces $\Gamma_t$ and $\tilde{\Gamma}_t$ is uniformly bounded in time, see~\cite{BH2}. Notice also that, for a given transition front connecting $0$ and $1$, the global mean speed $\gamma$, if any, does not depend on the choice of the families $(\Omega^{\pm}_t)_{t\in\R}$ and $(\Gamma_t)_{t\in\R}$. However, for equations of the type~\eqref{eq1.1} in $\R$ with Fisher-KPP reactions $f$, there are transition fronts without global mean speed, see~\cite{HN,HR2}.

For problem~\eqref{eq1.1} under assumptions~\eqref{F1}-\eqref{F2}, it was proved in~\cite{H} that, in the one-dimensional case $\Omega=\R$, the only transition fronts connecting $0$ and $1$ are the right-moving and left-moving traveling fronts $\phi(\pm x-c_ft)$ (up to shifts), moving with constant speed $c_f$. In the space $\R^N$, any almost-planar transition front (in the sense that, for every $t\in\R$, $\Gamma_t$ is a hyperplane) connecting~$0$ and~$1$ is truly planar. Furthermore, still in $\R^N$, any transition front connecting $0$ and $1$ has a global mean speed $\gamma$, which is equal to $c_f$ and is therefore independent of the shape of the level sets of $u$, see~\cite{H}. We point out that the aforementioned axisymmetric conical-shaped or pyramidal fronts of the type $\psi(x-cte)$, existing for any $e\in\mathbb{S}^{N-1}$ and any $c\ge c_f$, still have a global mean speed equal to $c_f$, whatever $c$ may be in $[c_f,+\infty)$, since the distance between the level sets of these fronts at times $t$ and $s$ is always asymptotically equivalent to $c_f|t-s|$ as~$|t-s|\to+\infty$. Lastly, even in the homogeneous space $\R^N$, non-standard transition fronts which are not invariant in any moving frame were also constructed in~\cite{H} under assumptions~\eqref{F1}-\eqref{F3}. More generally speaking, there is now a large literature devoted to transition fronts for bistable reactions in homogeneous or heterogeneous settings~\cite{BHM,DHZ,E,G2,S1,SG,Z3}, as well as for other types of homogeneous or space/time dependent reactions in dimension~1~\cite{Du,HR1,HR2,MNRR,MRS,N,NR3,NRRZ,NR4,S4,SS,TZZ,Z1} and in higher dimensions~\cite{AHLTZ,BGW,NR1,NR2,S2,S3,Z2,Z4}.

Our goal in the present paper is to study some qualitative properties of transition fronts and their propagation speeds for problem~\eqref{eq1.1} under the assumptions~\eqref{F1}-\eqref{F2} in heterogeneous domains~$\Omega$. We will focus on two classes of domains: the exterior domains (that is, the complements of compact sets) and the domains with multiple cylindrical branches (more precise definitions will be given in the next subsections). We will especially give some sufficient conditions on $\Omega$ for the transition fronts to have a global mean speed equal to the planar speed $c_f$ in the homogeneous case $\Omega=\R^N$. We will also comment on the sharpness of these conditions.


\subsection{Exterior domains}

Let us first consider the case where $\Omega$ is an exterior domain, that is, $\Omega=\R^N\setminus K$, where~$K$ is a compact set with smooth boundary. The interaction between the obstacle $K$ and a planar traveling front, say $\phi(x_1-c_ft)$ propagating in the direction $x_1$ without loss of generality, was thoroughly studied in~\cite{BHM}. A solution $u(t,x)$ of~\eqref{eq1.1} converging to $\phi(x_1-c_ft)$ as $t\to-\infty$ uniformly in $x\in\overline{\Omega}$ was constructed in~\cite{BHM}, for $C^{1,1}([0,1])$ functions $f$ satisfying~\eqref{F1}-\eqref{F2}. It was also proved that, if the obstacle $K$ is star-shaped or directionally convex with respect to some hyperplane\footnote{The obstacle $K$ is called star-shaped if either $K=\emptyset$ or there is $x$ in the interior $\mathrm{Int}(K)$ of $K$ such that $x+t(y-x)\in\mathrm{Int}(K)$ for all $y\in\partial K$ and $t\in[0,1)$. In the latter case, we say that $K$ is star-shaped with respect to the point $x$. The obstacle $K$ is called directionally convex with respect to a hyperplane $H=\{x\in\R^N: x\cdot e=a\}$, with $e\in\mathbb{S}^{N-1}$ and $a\in\R$, if for every line $\Sigma$ parallel to $e$, the set $K\cap\Sigma$ is either a single line segment or empty and if $K\cap H$ is equal to the orthogonal projection of $K$ onto $H$.} (see Figure~$1$), then the solution $u$ passes the obstacle, in the sense that~$u(t,x)$ converges to $\phi(x_1-c_ft)$ as $t\to+\infty$ uniformly in $x\in\overline{\Omega}$. In particular, in these cases, the propagation is said to be {\it complete}, namely
\be\label{complete}
u(t,x)\to1\ \hbox{ as }t\to+\infty\ \hbox{ locally uniformly in }x\in\overline{\Omega}.
\ee
That solution $u$ is also proved to be an almost-planar transition front connecting $0$ and $1$, in the sense that Definition~\ref{TF} is satisfied and one can choose $\Gamma_t=\big\{x\in\Omega: x_1=c_ft\big\}$ in~\eqref{eq1.3} to be the intersection of $\Omega$ with a hyperplane. Notice that the interfaces $\Gamma_t$ give a rough idea of the location of the level sets of $u(t,\cdot)$ but they are not equal to these level sets in general. For instance, here the solution $u$ is not planar as soon as $K$ is not empty, even if the interfaces $\Gamma_t$ can be chosen to be included in parallel hyperplanes.
\begin{figure}[ht]\label{fig1}\centering
\includegraphics[scale=0.45]{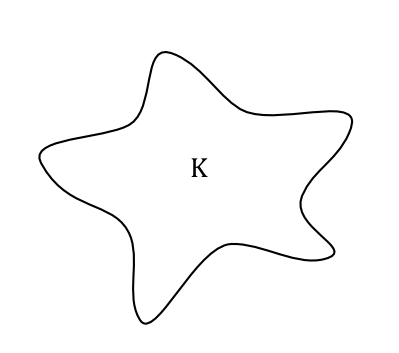}\hskip 1cm \includegraphics[scale=0.6]{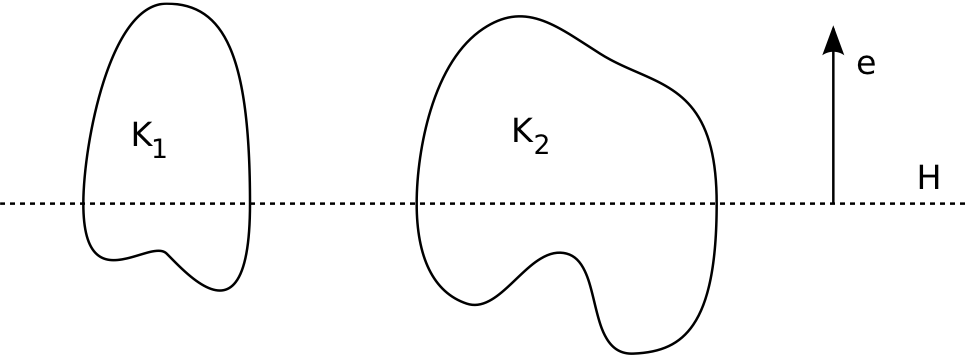}
\caption{Left: a star-shaped obstacle; right: a directionally convex obstacle $K=K_1\cup K_2$.}
\end{figure}

Transition fronts connecting $0$ and $1$ with other shapes are also expected to exist. For instance, we conjecture that, under some geometrical conditions on $K$, equation~\eqref{eq1.1} admits transition fronts converging to any given axisymmetric conical-shaped or pyramidal traveling front $\psi(x-cte)$ as $t\to\pm\infty$ uniformly in $x\in\overline{\Omega}$, where $e\in\mathbb{S}^{N-1}$ is any given unit vector. This question will be the purpose of a future work.

\subsubsection*{Transition fronts with complete propagation}

In the present paper, we focus on the existence and characterization of the global mean speed of {\it any} transition front connecting $0$ and $1$, whatever the shape of its level sets may be. As recalled in Section~\ref{sec1.1}, any transition front connecting $0$ and $1$ for~\eqref{eq1.1} in $\R^N$ has a global mean speed $\gamma$, which is equal to $c_f$, whatever the shape, planar or not, of level sets of the front may be, see~\cite{H}. Our first main result states that the same conclusion holds in exterior domains, provided that the propagation is complete.

\begin{theorem}\label{th1}
Let $\Omega=\R^N\setminus K$, where $K$ is a compact set with smooth boundary. Then any transition front of~\eqref{eq1.1} connecting $0$ and $1$ and satisfying the complete propagation condition~\eqref{complete} has a global mean speed~$\gamma$, and $\gamma=c_f$.
\end{theorem}

The aforementioned transition fronts constructed in~\cite{BHM} and converging to $\phi(x_1-c_ft)$ as~$t\to\pm\infty$ uniformly in $x\in\overline{\Omega}$ are particular examples fulfilling the hypotheses and the conclusion of Theorem~\ref{th1}.

There are also known counterexamples without the condition~\eqref{complete}. Namely, it was proved in~\cite{BHM} that, for {\it some} particular strongly non-convex obstacles $K$, the solutions $u(t,x)$ emanating from the planar front $\phi(x_1-c_ft)$ as $t\to-\infty$ do not pass the obstacle completely, in the sense that~\eqref{complete} is not fulfilled (in other words, the disturbance caused by the obstacle $K$ remains forever). However, one still has $u(t,x)-\phi(x_1-c_ft)\to0$ as $|x|\to+\infty$ uniformly in $t\in\R$, and the solutions $u$ are proved to still converge to the planar front $\phi(x_1-c_ft)$ as $t\to+\infty$ far behind the obstacle, in the sense that $\sup_{x\in\overline{\Omega},\,x_1\ge\xi_t}|u(t,x)-\phi(x_1-c_ft)|\to0$ as $t\to+\infty$ if $\xi_t\to+\infty$ as~$t\to+\infty$. Therefore, these solutions $u$ can still be viewed as transition fronts connecting~$0$ and~$1$ with, say,
$$\left\{\baa{lll}
\Omega_t^-=\big\{x\in\Omega: x_1>c_ft\big\}, & \Omega_t^+=\big\{x\in\Omega: x_1<c_ft\big\}, &\hbox{for }t\le0,\vspace{3pt}\\
\Omega_t^-=\big\{x\in\Omega: x_1>c_ft\big\}\ \cup \big(B(0,R)\cap\Omega\big), & \Omega_t^+=\big\{x\in\Omega: x_1<c_ft\big\}\setminus B(0,R), & \hbox{for }t>0,\eaa\right.$$
where $R>0$ is, say, any real number such that $K\subset B(0,R)$. These fronts connecting $0$ and $1$ do not satisfy~\eqref{complete}, and they do not have any global mean speed since $d_\Omega(\Gamma_t,\Gamma_s)\sim c_f|t-s|$ as $|t-s|\to+\infty$ with $t,s\le0$, but $d_\Omega(\Gamma_t,\Gamma_s)=0$ for all $t,s>0$. Notice nevertheless that, from~\cite{BHM}, these solutions could still be viewed as almost-planar fronts with $\Gamma_t=\big\{x\in\Omega: x_1=c_ft\big\}$ and global mean speed~$c_f$, but, instead of connecting $0$ and $1$ in Definition~\ref{TF}, they connect $0$ and a {\it non-constant} classical solution $p:\overline{\Omega}\to(0,1)$ of the stationary problem
\be\label{eqp}
\Delta p+f(p)=0\hbox{ in }\Omega,\ \ \ p_{\nu}=0\hbox{ on }\partial\Omega,\ \hbox{ and }\ p(x)\to1\hbox{ as }|x|\to+\infty.
\ee

\begin{remark}\label{rem13}{\rm
It actually turns out that, for {\it any} exterior domain (and not only those considered in the previous paragraph), the hypothesis~\eqref{complete} in Theorem~\ref{th1} is not only sufficient but also necessary for the conclusion to hold. Indeed, the arguments used in Section~\ref{sec2.3} below imply that, for any transition front $u$ of~\eqref{eq1.1} connecting $0$ and $1$, there is a solution $p:\overline{\Omega}\to(0,1]$ of~\eqref{eqp} such that
\be\label{liminfp}
\liminf_{t\to+\infty}u(t,x)\ge p(x)>0\ \hbox{ locally uniformly in }x\in\overline{\Omega},
\ee
in the sense that $\liminf_{t\to+\infty}\min_{\mathcal{C}}\big(u(t,\cdot)-p\big)\ge0$ for any compact set $\mathcal{C}\subset\overline{\Omega}$. Now, if~$u$ has global mean speed equal to $c_f$, then in particular $d_{\Omega}(\Gamma_t,\Gamma_0)\to+\infty$ as $t\to+\infty$. Hence~$\inf\big\{|x|: x\in\Gamma_t\big\}\to+\infty$ as $t\to+\infty$ and Definition~\ref{TF} implies that either~$u(t,x)\to0$ or~$u(t,x)\to1$ as $t\to+\infty$, locally uniformly in $x\in\overline{\Omega}$. Since the former case is ruled out by~\eqref{liminfp}, only the latter case holds, meaning that~\eqref{complete} holds and the propagation is complete. In other words, the converse of Theorem~\ref{th1} is true, that is, in an exterior domain, any transition front connecting~$0$ and~$1$ has global mean speed equal to $c_f$ if and only if it satisfies~\eqref{complete}.}
\end{remark}

\begin{remark}{\rm Other notions of distances between sets could be considered instead of the one used in the paper. Some would lead to similar results. But the Hausdorff distance $d_H$ would lead to different results in general. For instance, for~\eqref{eq1.1} in $\R^N$ with $N\ge2$ under assumptions~\eqref{F1}-\eqref{F3}, there are transition fronts (examples of such fronts are the axisymmetric conical-shaped or pyramidal fronts) connecting $0$ and $1$ for which $d_H(\Gamma_t,\Gamma_s)\sim c|t-s|$ as $|t-s|\to+\infty$ with $c>c_f$ and for which the global mean speed thus depends on the front. We refer to~\cite[Remark~2.8]{H} for further details.}
\end{remark}

\subsubsection*{The case of almost-planar fronts}

If a transition front $u$ connecting $0$ and $1$ for~\eqref{eq1.1} is further assumed to be almost-planar, in the sense that for every $t\in\R$ one can choose $\Gamma_t=H_t\cap\Omega$ for some hyperplane $H_t$, then one can easily show that the propagation must be complete in the sense of~\eqref{complete} and, by Theorem~\ref{th1}, the hyperplanes $H_t$ move with global mean speed $c_f$ (see more details in Section \ref{sec2.3}). In particular, for such a front $u$, the hyperplanes $H_t$ turn out to be parallel to each other by definition of the distance, since the global mean speed $c_f$ is nonzero. As a consequence, the following result holds.

\begin{corollary}\label{cor1}
Let $\Omega=\R^N\setminus K$, where $K$ is a compact set with smooth boundary. Then any almost-planar transition front of~\eqref{eq1.1} connecting $0$ and $1$ has a global mean speed equal to $c_f$.
\end{corollary}

We point out that the situation considered here differs from the case of periodic media, for which the speeds of pulsating fronts, which are particular almost-planar transition fronts, may depend on the direction of propagation, see~\cite{D,DR,Du2,G2,He}.

\subsubsection*{Sufficient geometrical and scaling conditions on the obstacle}

It follows from~\cite[Theorems~6.1 and~6.4]{BHM} that, if the compact obstacle $K$ is either star-shaped or directionally convex with respect to some hyperplane, then any solution $p:\overline{\Omega}\to[0,1]$ of~\eqref{eqp} is identically equal to~$1$. Together with property~\eqref{liminfp} satisfied by any transition front $u$ connecting $0$ and $1$, Theorem~\ref{th1} immediately yields the following corollary.

\begin{corollary}\label{cor2}
Let $\Omega=\R^N\setminus K$, where $K$ is a compact set with smooth boundary. If $K$ is either star-shaped or directionally convex with respect to some hyperplane, then any transition front of~\eqref{eq1.1} connecting $0$ and $1$ has a global mean speed equal to $c_f$.
\end{corollary}

Lastly, we will show in Section~\ref{sec2.3} that, under assumptions~\eqref{F1}-\eqref{F2}, for any smooth exterior domain $\Omega=\R^N\setminus K$, the dilated exterior domains $R\,\Omega=\R^N\setminus(R\,K)$ and their shifts $R\,\Omega+x_0=\R^N\setminus(R\,K)\,+\,x_0$ are such that, for all $R>0$ large enough and all $x_0\in\R^N$, any solution $p:\overline{R\,\Omega+x_0}\to[0,1]$ of~\eqref{eqp} with $R\,\Omega+x_0$ instead of $\Omega$ is identically equal to $1$. This Liouville-type result is actually new with respect to the related ones in~\cite{BHM,B} since the obstacles $R\,K+x_0$ may not be star-shaped or directionally convex, or close to any of these two cases. Together with Theorem~\ref{th1} and~\eqref{liminfp}, the next result will then follow.

\begin{corollary}\label{cor3}
Let $\Omega=\R^N\setminus K$, where $K$ is a compact set with smooth boundary. Then there is $R_0>0$ such that, for any $R\ge R_0$ and for any $x_0\in\R^N$, any transition front connecting $0$ and $1$ for~\eqref{eq1.1} in the domain~$R\,\Omega+x_0$ has a global mean speed equal to $c_f$.
\end{corollary}


\subsection{Domains with multiple cylindrical branches}

In the second part of the paper, we consider the class of domains with multiple cylindrical branches. By way of an example, the simplest case of such a domain is a straight cylinder, namely when $\Omega$ can be written up to rotation as
\be\label{straight}
\Omega=\big\{(x_1,x'): x_1\in\R,\ x'\in\omega\big\},
\ee
where $\omega\subset\R^{N-1}$ is a smooth bounded non-empty open connected subset of $\R^{N-1}$. In such a domain, the planar front $\phi(x_1-c_f t)$ is a solution of~\eqref{eq1.1} moving with constant speed $c_f$.

Another particular case of a domain with multiple cylindrical branches is that of a bilaterally straight cylinder, that is, $\Omega$ can be written up to rotation as
\be\label{bilateral}
\Omega=\big\{(x_1,x'): x_1\in\R,\ x'\in\omega(x_1)\big\},
\ee
where $(\omega(x_1))_{x_1\in\R}$ is a smooth family of smooth bounded non-empty open connected subset of~$\R^{N-1}$ such that $\omega(x_1)$ is independent of $x_1$ for $x_1\le-L$ and for $x_1\ge L$, for some $L>0$. In a bilaterally straight cylinder $\Omega$ of the type~\eqref{bilateral} which is not a straight cylinder, the planar front~$\phi(x_1-c_f t)$ is not a solution of~\eqref{eq1.1} anymore. However, there still exist some front-like solutions emanating from the planar front $\phi(x_1-c_f t)$ coming from the ``left" part of the domain. More precisely, there exists a unique solution $u:\R\times\overline{\Omega}\to(0,1)$ of~\eqref{eq1.1} such that
\be\label{eq1.9}
u(t,x)-\phi(x_1-c_f t)\rightarrow 0\,\text{ as } t\rightarrow -\infty \text{ uniformly in $\overline{\Omega}$},
\ee
see~\cite{BBC,P} (see also~\cite{CG} for related results in some more specific geometries). We call such a solution a front-like solution. In fact, a similar result holds if $\Omega$ is bent, that is, the ``left" and ``right" parts of~$\Omega$ may not be parallel to the same direction, see~\cite{BBC}. Moreover, it was shown in~\cite{BBC} that, under some geometrical conditions on $\Omega$ (for instance, if the map $x_1\mapsto\omega(x_1)$ is non-increasing), the propagation is complete, in the sense that the solution $u$ of~\eqref{eq1.1} satisfying~\eqref{eq1.9} satisfies~\eqref{complete} too. However, under some other geometrical conditions (for instance, if $x_1\mapsto\omega(x_1)$ is non-decreasing and if $\omega(0)$ is contained in a small ball and $\omega(1)$ contains a large ball), blocking phenomena may occur, that is, the solution $u$ of~\eqref{eq1.1} satisfying~\eqref{eq1.9} is such that
\be\label{blocking}
u(t,x)\rightarrow u_{\infty}(x)\ \text{ as }t\rightarrow +\infty\text{ uniformly in }\overline{\Omega},\ \text{ with } u_{\infty}(x)\rightarrow 0\ \text{ as }x_1\rightarrow +\infty,
\ee
see~\cite{BBC,CG,RRBK} (see also~\cite{DR} for similar blocking phenomena in some periodic domains).

In the second part of this paper, we deal with transition fronts in domains with multiple cylindrical branches generalizing the domains of the type~\eqref{bilateral}. The results are concerned with the propagation speeds of transition fronts. They are of a different nature than the aforementioned known results  and are actually new even in the case of bilaterally straight cylinders, in the sense that we especially also make more precise the behavior of the solutions as $t\to+\infty$ when the propagation is complete, and we provide other geometrical and scaling conditions for the existence of a global mean speed for {\it all} transition fronts.

Let us now define what we mean by a domain with multiple branches. A cylindrical branch in a direction $e\in\mathbb{S}^{N-1}$ with cross section $\omega\,$\footnote{The section $\omega$ is a smooth non-empty bounded domain for the $(N-1)$-dimensional Euclidean topology of the hyperplane of $\R^N$ orthogonal to $e$ and containing $0$.} and shift $x_0\in\R^N$ is the open unbounded domain of $\R^N$ defined by
\be\label{defbranch}
\mathcal{H}_{e,\omega,x_0}=\big\{x\in\R^N: x-(x\cdot e)e\in\omega,\ x\cdot e>0\big\}+x_0.
\ee
A smooth unbounded domain of $\R^N$ is called a domain with multiple cylindrical branches if there exist a real number $L>0$, an integer $m\ge2$, and $m$ cylindrical branches $\mathcal{H}_i:=\mathcal{H}_{e_i,\omega_i,x_i}$ (with $i=1,\cdots,m$), such that
\be\label{branches}\left\{\baa{l}
\displaystyle\mathcal{H}_i\cap\mathcal{H}_j=\emptyset\hbox{ for every }i\neq j\in\big\{1,\cdots,m\big\}\ \hbox{ and }\ \Omega\setminus B(0,L)=\mathop{\bigcup}_{i=1}^{m}\mathcal{H}_i\setminus B(0,L),\vspace{3pt}\\
\mathcal{H}_i\setminus B(0,L)\hbox{ is connected for every }i\in\{1,\cdots,m\}.\eaa\right.
\ee
Some examples of domains with 5 cylindrical branches are depicted in Figure~2.
\begin{figure}[ht]\centering
\includegraphics[scale=0.55]{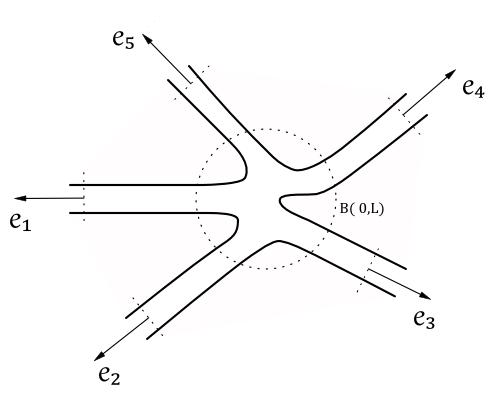}\includegraphics[scale=0.45]{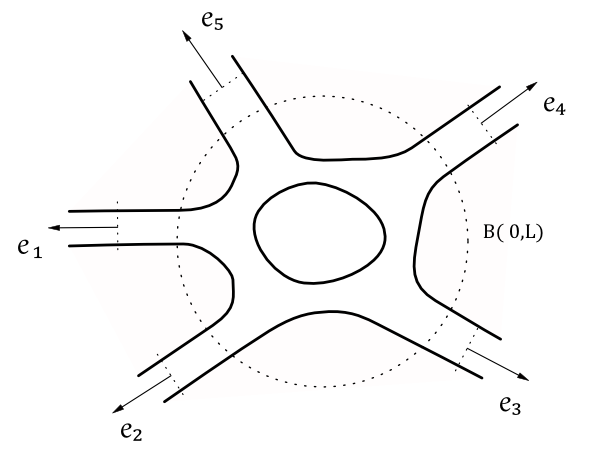}
\caption{Examples of domains with 5 cylindrical branches.}
\end{figure}

\subsubsection*{Solutions emanating from planar fronts with complete propagation}

We first show that the front-like solutions emanating from planar fronts in some branches and propagating completely are transition fronts and converge to planar fronts in the other branches.

\begin{theorem}\label{th10}
Let $\Omega$ be a smooth domain with $m\,(\ge2)$ cylindrical branches and let $I$ and $J$ be two non-empty sets of $\{1,\cdots,m\}$ such that $I\cap J=\emptyset$ and $I\cup J=\{1,\cdots,m\}$. Let $u:\R\times\overline{\Omega}\to(0,1)$ be a time-increasing solution of~\eqref{eq1.1} such that
\be\label{frontlike}
\left\{\baa{rcll}
u(t,x)\!-\!\phi(-x\cdot e_i\!-\!c_f t\!+\!\sigma_i) & \!\!\!\to\!\!\! & 0\!\! & \text{uniformly in $\overline{\mathcal{H}_i}\cap\overline{\Omega}\,$ for every $i\in I$},\vspace{3pt}\\
u(t,x) & \!\!\to\!\! & 0\!\! & \displaystyle\text{uniformly in $\overline{\Omega\setminus \mathop{\bigcup}_{i\in I}\mathcal{H}_i}$},\eaa\right.\text{as $t\rightarrow -\infty$}
\ee
for some real numbers $(\sigma_i)_{i\in I}$. If $u$ propagates completely in the sense of~\eqref{complete}, then it is a transition front connecting $0$ and $1$ with global mean speed $c_f$ and $(\Gamma_t)_{t\in \R}$, $(\Omega^{\pm}_t)_{t\in \R}$ defined by
\be\label{eq+1.11}
\Gamma_t\!=\!\mathop{\bigcup}_{i\in I}\!\big\{x\!\in\!\mathcal{H}_i\cap\Omega: x\cdot e_i\!=\!c_f|t|\!+\!A\big\}\ (t\!\le\!0),\ \Gamma_t\!=\!\mathop{\bigcup}_{j\in J}\!\big\{x\!\in\!\mathcal{H}_j\cap\Omega: x\cdot e_j\!=\!c_f t\!+\!A\big\}\ (t\!>\!0),
\ee
and
\begin{eqnarray}\label{eq+1.12}
\left\{\baa{llll}
\displaystyle\Omega_t^+=\mathop{\bigcup}_{i\in I}\big\{x\in \mathcal{H}_i\cap\Omega: x\cdot e_i>c_f|t|+A\big\}, & \Omega_t^-=\Omega\setminus\overline{\Omega_t^+}, &\text{ for $t\le0$,}\vspace{3pt}\\
\displaystyle\Omega_t^-=\mathop{\bigcup}_{j\in J}\big\{x\in \mathcal{H}_j\cap\Omega: x\cdot e_j>c_f t+A\big\}, & \Omega_t^+=\Omega\setminus\overline{\Omega_t^-}, &\text{ for $t>0$,}\eaa\right.
\end{eqnarray}
for some $A>0$. Moreover, there exist some real numbers $(\tau_j)_{j\in J}$ such that
\be\label{largetime}\left\{\baa{rcll}
u(t,x)-\phi(x\cdot e_j\!-\!c_f t\!+\!\tau_j) & \!\!\!\to\!\!\! & 0\!\! & \text{uniformly in $\overline{\mathcal{H}_j}\cap\overline{\Omega}\,$ for every $j\in J$},\vspace{3pt}\\
u(t,x) & \!\!\!\to\!\!\! & 1\!\! & \displaystyle\text{uniformly in $\overline{\Omega\setminus \mathop{\bigcup}_{j\in J}\mathcal{H}_j}$},\eaa\right.\text{ as $t\rightarrow+\infty$.}
\ee
\end{theorem}

In a bilaterally straight cylinder or in a domain with two cylindrical branches, the existence and uniqueness of time-increasing solutions $u$ satisfying~\eqref{frontlike} follows from the arguments used in~\cite{BBC,BHM,P}. In a domain with multiple cylindrical branches, the existence and uniqueness of time-increasing solutions $u$ satisfying~\eqref{frontlike} could be shown with similar arguments. The interest of Theorem~\ref{th10} is to show that, provided the propagation is complete in the sense of~\eqref{complete}, these solutions $u$ are transition fronts connecting $0$ and $1$ and more precisely that they converge at large time to the planar fronts (with some finite shifts) in the other branches. These facts are actually new even in the case of bilaterally straight cylinders, where $c_f$ was only known to be the asymptotic spreading speed.

Notice also that the condition~\eqref{complete} of complete propagation is necessary for the conclusion to hold. In bilaterally straight cylinders of the type~\eqref{bilateral}, some examples of blocking phenomena have been exhibited in~\cite{BBC,CG,RRBK}, namely there exist time-increasing solutions $u:\R\times\overline{\Omega}\to(0,1)$ satisfying~\eqref{eq1.9} and~\eqref{blocking}, for which~\eqref{complete} is therefore not fulfilled. These solutions could still be viewed as transition fronts connecting $0$ and $1$ with, say, $\Omega^{\pm}_t=\big\{x\in\Omega: \pm(x_1-c_ft)<0\big\}$ and $\Gamma_t=\big\{x\in\Omega: x_1=c_ft\big\}$ for $t\le0$ and $\Omega^{\pm}_t=\big\{x\in\Omega: \pm x_1<0\big\}$ and $\Gamma_t=\big\{x\in\Omega: x_1=0\big\}$ for $t\ge0$. In particular, these solutions do not have a global mean speed.

Our next result is concerned with the existence and uniqueness of the global mean speed of {\it any} transition front, assuming the complete propagation of the front-like solutions emanating from each of the branches. In the sequel, for each $i\in\big\{1,\cdots,m\big\}$, we call $u_i:\R\times\overline{\Omega}\to(0,1)$ the time-increasing solution of~\eqref{eq1.1} coming from the branch $\mathcal{H}_i$, that is,
\be\label{frontlikei}
u_i(t,x)\!-\!\phi(-x\cdot e_i\!-\!c_f t)\!\mathop{\longrightarrow}_{t\to-\infty}\!0\text{ uniformly in }\overline{\mathcal{H}_i}\cap\overline{\Omega},\ u_i(t,x)\!\mathop{\longrightarrow}_{t\to-\infty}\!0 \text{ uniformly in $\overline{\Omega\!\setminus\!\mathcal{H}_i}.$}
\ee

\begin{theorem}\label{th6}
Let $\Omega$ be a smooth domain with $m\,(\ge2)$ cylindrical branches. Assume that, for every $i\in\big\{1,\cdots,m\big\}$, the time-increasing solution $u_i$ of~\eqref{frontlikei} propagates completely in the sense of~\eqref{complete}. Then any transition front of~\eqref{eq1.1} connecting $0$ and $1$ propagates completely in the sense of~\eqref{complete} and it has a global mean speed equal to $c_f$.
\end{theorem}

The condition of complete propagation of all these front-like solutions $u_i$ is actually necessary for the conclusion to hold in general. Indeed, as already mentioned~\cite{BBC}, in some bilaterally straight cylinders of the type~\eqref{bilateral}, there are solutions $u$ satisfying~\eqref{eq1.9} and~\eqref{blocking}, which then do not satisfy~\eqref{complete} and yet are transition fronts connecting $0$ and $1$ without global mean speed. In the general case of a domain~$\Omega$ with $m\,(\ge 2)$ cylindrical branches, consider now, if any, a solution~$u_i$ satisfying~\eqref{frontlikei} but not~\eqref{complete}. If this solution $u_i$ were still a transition front connecting $0$ and $1$, then $\Gamma_t$ could be chosen as $\Gamma_t=\big\{x\in\mathcal{H}_i\cap\Omega: x\cdot e_i=c_f|t|\big\}$ for very negative $t$ while $\liminf_{t\to+\infty}\inf\big\{|x|: x\in\Gamma_t\big\}<+\infty$ (since otherwise $u_i(t,x)\to1$ as $t\to+\infty$ locally uniformly in $x\in\overline{\Omega}$ by~\eqref{eq1.7} and the time-monotonicity of $u_i$). Therefore, this solution $u_i$ does not have any global mean speed. As a matter of fact, we conjecture that any solution $u_i$ satisfying~\eqref{frontlikei} but not~\eqref{complete} is still a transition front connecting $0$ and $1$. From the previous arguments, this conjecture is equivalent to the following one.

\begin{conjecture}
The converse of Theorem~$\ref{th6}$ holds as well, that is, the complete propagation of all the front-like solutions $u_i$ is equivalent to the property that all transition fronts connecting~$0$ and~$1$ have a global mean speed equal to $c_f$.
\end{conjecture}

\begin{remark}\label{3branches}{\rm
In Theorem~$\ref{th1}$, we stated that, in an exterior domain $\Omega$, any transition front connecting~$0$ and~$1$ with complete propagation has a global mean speed, equal to $c_f$. We could wonder whether a similar property could hold in a domain with multiple cylindrical branches. We conjecture that such a property is actually false in general, more precisely that there may exist some transition fronts with complete propagation and without global mean speed. To be more explicit, consider the case of a domain with $3$ cylindrical branches. Assume that the solution~$u_1$ emanating from the branch~$\mathcal{H}_1$ (it is given by~\eqref{frontlikei} with $i=1$) satisfies~\eqref{complete}. Assume also that the solution $u_2$ emanating from the branch $\mathcal{H}_2$ does not satisfy~\eqref{complete}, and converges as~$t\to+\infty$ to a stable stationary solution $v:\overline{\Omega}\to(0,1)$ such that $v(x)\to1$ as~$x\cdot e_2\to+\infty$ in $\overline{\mathcal{H}_2}$ and~$v(x)\to0$ as $|x|\to+\infty$ in $\overline{\Omega\setminus\mathcal{H}_2}$ (hence, for any $a\in(0,1)$, the set where $v(x)>a$ is included in the union of $\overline{\mathcal{H}_2}$ with a bounded set), see Figure~3. Such a domain~$\Omega$ could be cooked up from the results of~\cite{BBC}. We then believe that there exists a time-increasing solution~$u:\R\times\overline{\Omega}\to(0,1)$ of~\eqref{eq1.1} such that
$$\left\{\baa{l}
\displaystyle \!u(t,x)\!-\!\phi(-x\cdot e_1\!-\!c_ft)\!\mathop{\longrightarrow}_{t\to-\infty}\!0\hbox{ (resp. }\!u(t,x)\!\mathop{\longrightarrow}_{t\to-\infty}\!v(x)\hbox{) uniformly in }\overline{\mathcal{H}_1}\!\cap\!\overline{\Omega}\hbox{ (resp. in }\overline{\Omega\!\setminus\!\mathcal{H}_1}\hbox{)},\vspace{3pt}\\
\displaystyle \!u(t,x)\!-\!\phi(x\cdot e_3\!-\!c_ft+\tau)\!\mathop{\longrightarrow}_{t\to+\infty}\!0\hbox{ (resp. }\!u(t,x)\!\mathop{\longrightarrow}_{t\to+\infty}\!1\hbox{) uniformly in }\overline{\mathcal{H}_3}\!\cap\!\overline{\Omega}\hbox{ (resp. in }\overline{\Omega\!\setminus\!\mathcal{H}_3}\hbox{)},\eaa\right.$$
for some $\tau\in\R$. This solution $u$ would then be a transition front connecting $0$ and $1$ with, say,
$$\left\{\baa{lll}
\Omega^+_t\!=\!\big\{x\in\mathcal{H}_1\cap\Omega: x\cdot e_1\!>\!c_f|t|\!+\!L\big\}\cup\big\{x\in\mathcal{H}_2\cap\Omega: x\cdot e_2\!>\!L\big\}, & \!\Omega^-_t\!=\!\Omega\!\setminus\!\overline{\Omega^+_t}, & \!\hbox{for }t\le0,\vspace{3pt}\\
\Omega^-_t\!=\!\big\{x\in\mathcal{H}_3\cap\Omega: x\cdot e_3\!>\!c_ft\!+\!L\big\}, & \!\Omega^+_t\!=\!\Omega\!\setminus\!\overline{\Omega^-_t}, & \!\hbox{for }t>0,\eaa\right.$$
where $L>0$ is a large real number, for instance as in~\eqref{branches}. Therefore, $u$ does not have any global mean speed since $d_{\Omega}(\Gamma_t,\Gamma_s)\!\sim\!c_f|t\!-\!s|$ as $|t\!-\!s|\to+\infty$ with $t,s>0$, whereas $d_{\Omega}(\Gamma_t,\Gamma_s)=0$ for all $t,s\le0$. However, this transition front $u$ satisfies~\eqref{complete}: the propagation is complete.}
\end{remark}
\begin{figure}[ht]\centering
\includegraphics[scale=0.45]{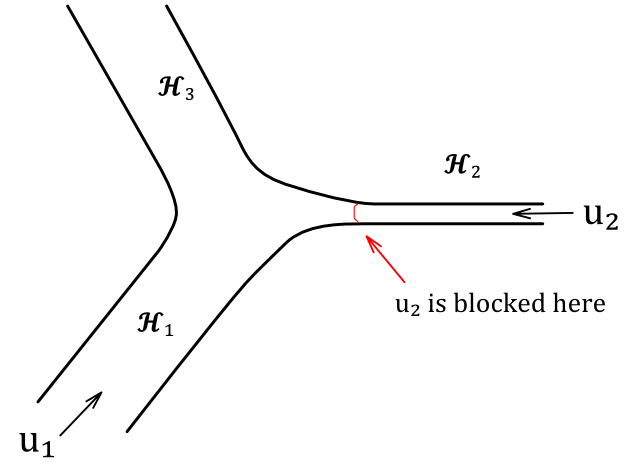}
\caption{An example for Remark \ref{3branches} (the red curve represents for instance a level set of $u_2$).}
\end{figure}

\subsubsection*{Sufficient geometrical and scaling conditions}\label{secgeom}

Finally, we give some suitable conditions on $\Omega$ so that the front-like solutions $u_i$ satisfying~\eqref{frontlikei} propagate completely in the sense of~\eqref{complete} (remember from~\cite{BBC} that blocking phenomena may occur in general and that the propagation is therefore not complete in general even for domains with two cylindrical branches). Then, under these conditions, it follows from Theorem~\ref{th6} that any transition front connecting $0$ and $1$ actually propagates with a global mean speed, equal to~$c_f$. To do so, we need some additional notations. Under the assumption~\eqref{branches}, we consider for every $i\neq j\in\{1,\cdots,m\}$ a continuous path $P_{i,j}: \R\rightarrow \Omega$ connecting the two branches $\mathcal{H}_{i}$ and $\mathcal{H}_j$ in the sense that, for any real number $A$, there is $\sigma>0$ such that
$$P_{i,j}(s)\in\big\{x\!\in\!\mathcal{H}_i\cap\Omega: x\cdot e_i\!\ge\!A\big\}\text{ for }s\!\le\!-\sigma\ \hbox{ and }\ P_{i,j}(s)\in\big\{x\!\in\!\mathcal{H}_j\cap\Omega: x\cdot e_j\!\ge\!A\big\}\text{ for }s\!\ge\!\sigma.$$
One can assume without loss of generality that $P_{i,j}(\R)=P_{j,i}(\R)$ for every $i\neq j$, and that, for every~$i$, any two paths $P_{i,j}$ and $P_{i,k}$ with $j\neq k\in\{1,\cdots,m\}\setminus\{i\}$ (if $m\ge3$) share the same parts in $\mathcal{H}_i$ far away from the origin, that is, there is $s_i\in\R$ such that $P_{i,j}(s)=P_{i,k}(s)$ for every~$s\le s_i$ and $j\neq k\in\{1,\cdots,m\}\setminus\{i\}$.

\begin{corollary}\label{th5}
There is a real number $R>0$ such that, for any smooth domain $\Omega$ with $m\,(\ge2)$ cylindrical branches, if the paths $P_{i,j}$ given above can be chosen with
$$B(P_{i,j}(s),R)\subset \Omega\ \text{ for all $s\in\R\ $ and $\ 0\in\!\!\bigcup_{1\le i\neq j\le m}\!\!P_{i,j}(\R)$},$$
and if $\Omega$ is star-shaped with respect to $0$, then any solution $u$ given in Theorem~$\ref{th10}$ propagates completely and is a transition front connecting $0$ and $1$ with global mean speed $c_f$. Furthermore, any transition connecting $0$ and $1$ propagates completely and has global mean speed equal to $c_f$.
\end{corollary}

In the particular case of bilaterally straight cylinders $\Omega$ of the type~\eqref{bilateral}, similar conditions to those of Corollary~\ref{th5} had been given in~\cite[Theorem~1.12]{BBC} to warrant the complete propagation of the front-like solutions $u$ satisfying~\eqref{eq1.9}.

Lastly, similarly to Corollary~\ref{cor3}, the following result holds in dilated domains with multiple cylindrical branches, without any star-shapedness assumption, as a consequence of Theorem~\ref{th6}.

\begin{corollary}\label{cor4}
Let $\Omega$ be a smooth domain with $m\,(\ge2)$ cylindrical branches. Then there is~$R_0>0$ such that, for any $R\ge R_0$ and for any $x_0\in\R^N$, any solution $u$ given in Theorem~$\ref{th10}$ in the domain~$R\,\Omega+x_0$ propagates completely and is a transition front connecting $0$ and $1$ with global mean speed $c_f$. Furthermore, any transition connecting $0$ and $1$ for~\eqref{eq1.1} in the domain~$R\,\Omega+x_0$ propagates completely and has global mean speed equal to $c_f$.
\end{corollary}

\subsubsection*{A conjecture on the classification of all transition fronts}

As already recalled from~\cite{H}, on the one-dimensional line $\Omega=\R$, the only transition fronts connec\-ting $0$ and $1$ are the traveling fronts $\phi(\pm x-c_ft)$ (up to shifts). Similarly, it can easily be shown (see Lemma~\ref{lemma3.2-} below) that, in a straight cylinder of the type~\eqref{straight}, the only transition fronts connecting $0$ and~$1$ for~\eqref{eq1.1} are the planar traveling fronts $\phi(\pm x_1-c_ft)$ (up to shifts), that is, the unique solutions emanating from the planar fronts from each of the two branches of the domain. Based on these observations and on Theorem~\ref{th6}, we conjecture that a similar description of all transition fronts connecting $0$ and $1$ holds in any domain with multiple cylindrical branches, provided that the conditions of Theorem~\ref{th6} are satisfied.

\begin{conjecture}\label{conj2}
Let $\Omega$ be a smooth domain with $m\,(\ge2)$ cylindrical branches in the sense of~\eqref{branches}. If all conditions of Theorem~$\ref{th6}$ are satisfied, then any transition front of~\eqref{eq1.1} connecting $0$ and $1$ is of the type~\eqref{frontlike}-\eqref{eq+1.12}, that is, it emanates from the planar fronts coming from some proper subset of branches as $t\to-\infty$ and it converges the planar fronts in the other branches as $t\to+\infty$.
\end{conjecture}

Notice that this conjecture is not expected to hold in general without the assumptions of Theorem~\ref{th6}, having in mind the possible counter-example mentioned in Remark~\ref{3branches}. However, from the comments given before Corollary~\ref{th5}, Conjecture~\ref{conj2} is expected to hold in the locally star-shaped and ``large" domains considered in Corollaries~\ref{th5} and~\ref{cor4}.

Lastly, under the assumptions of Theorem~\ref{th6}, we conjecture that the solutions given in Theorem~\ref{th10} are stable with respect to small perturbations and that they attract all solutions of the Cauchy problem which are initially close to $0$ or $1$ in each branch asymptotically. These questions will be the purpose of future work.

We complete this introduction by mentioning a recent work of Jimbo and Morita~\cite{JM} on the existence of solutions emanating from some branches and of possible blocking phenomena for equations set on the finite union of infinitely thin branches with a common vertex and Kirchhoff law for the first-order spatial derivatives at that vertex.\hfill\break

\noindent{\bf{Outline of the paper.}} This article is organized as follows. Section 2 is devoted to the proof of Theorem~\ref{th1} and its corollaries on the global mean speed of transition fronts with complete propagation in exterior domains. In Section~3, we prove Theorem~\ref{th10} on the front-like solutions emanating from some branches in domains with multiple cylindrical branches. In Section~4, we prove Theorem~\ref{th6} and its corollaries on the global mean speed of any transition front in domains with multiple branches.


\section{Transition fronts in exterior domains}\label{sec2}

This section is devoted to the proof of the existence and uniqueness properties of the global mean speed of transition fronts in exterior domains. Throughout this section, $K$ is a smooth non-empty compact subset of $\R^N$ and $\Omega=\R^N\setminus K$. Assume without loss of generality that~$0\in K$, and let $L>0$ be such that
\be\label{defL}
K\subset B(0,L).
\ee
We recall that we assume~\eqref{F1}-\eqref{F2} throughout the paper. Let
\be\label{deftheta12}
\theta_1=\min\big\{s\in(0,1): f(s)=0\big\}\ \hbox{ and }\ \theta_2=\max\big\{s\in(0,1): f(s)=0\big\}
\ee
and note that $0<\theta_1\le\theta_2<1$. We fix a real number $0<\delta<1/4$ such that
\be\label{d}
0\!<\!\delta\!<\!\min\Big(\frac{\theta_1}{4},\frac{1\!-\!\theta_2}{4},\frac{|f'(0)|}{2},\frac{|f'(1)|}{2}\Big),\ f'\!\le\!\frac{f'(0)}{2} \text{ in }[0,4\delta],\text{ and }f'\!\le\!\frac{f'(1)}{2} \text{ in $[1\!-\!4\delta,1]$}.
\ee
Notice that $0<4\delta<\theta_1\le\theta_2<1-4\delta<1$.

Section~\ref{sec2.1} is devoted to the proof of some key-lemmas that are used in the proof of the main results, namely Theorem~\ref{th1} in Section~\ref{sec2.2} and Corollaries~\ref{cor1} and~\ref{cor3} in Section~\ref{sec2.3}.


\subsection{Key-lemmas}\label{sec2.1}

For any $x_0\in\R^N$ and $R>0$, let $v_{x_0,R}(t,x)$ denote the solution of the Cauchy problem
\be\label{vrx0}
\left\{\baa{lll}
(v_{x_0,R})_t-\Delta v_{x_0,R}=f(v_{x_0,R}), & t>0,\ x\in\overline{\Omega},\vspace{3pt}\\
(v_{x_0,R})_\nu=0, & t>0,\ x\in\partial\Omega,
\eaa
\right.
\ee
with initial condition
$$v_{x_0,R}(0,x)= 1-\delta\hbox{ for }x\in B(x_0,R)\cap\overline{\Omega},\ \hbox{ and }\ v_{x_0,R}(0,x)=0\hbox{ for }x\in\overline{\Omega}\setminus B(x_0,R).$$
We also denote $\tilde{v}_{x_0,R}$ the solution of the Cauchy problem~\eqref{vrx0} with initial condition
$$\tilde{v}_{x_0,R}(0,x)= 1-2\delta\hbox{ for }x\in B(x_0,R)\cap\overline{\Omega},\ \hbox{ and }\ \tilde{v}_{x_0,R}(0,x)=0\hbox{ for }x\in\overline{\Omega}\setminus B(x_0,R).$$

One immediately gets the following lemma from \cite[Lemma 5.2]{BHM}.

\begin{lemma}{\rm{\cite[Lemma 5.2]{BHM}}}\label{lemma2.1}
There exist $8$ positive real numbers $R_1$, $R_2$, $R_3$, $T$, $\tilde{R}_1$, $\tilde{R}_2$, $\tilde{R}_3$ and~$\tilde{T}$ such that $R_3>R_2>R_1>0$, $\tilde{R}_3>\tilde{R}_2>\tilde{R}_1>0$, $R_2-R_1>c_f T/4$, $\tilde{R}_2-\tilde{R}_1>c_f\tilde{T}/4$ and, if~$B(x_0,R_3)\subset \Omega$, resp. if $B(x_0,\tilde{R}_3)\subset\Omega$, then
$$v_{x_0,R_1}(T,\cdot)\ge 1\!-\!\delta \text{ in }\overline{B(x_0,R_2)}\,(\subset\Omega),\ \ \hbox{resp. }\tilde{v}_{x_0,\tilde{R}_1}(\tilde{T},\cdot)\ge 1\!-\!2\delta \text{ in }\overline{B(x_0,\tilde{R}_2)}\,(\subset\Omega).$$
\end{lemma}

The first key-lemma in the proof of Theorem~\ref{th1}, namely Lemma~\ref{lemma2.2} below, provides some lower bounds in expanding balls for the solutions emanating from compactly supported initial conditions close to $1$ on large balls. 

\begin{lemma}\label{lemma2.2}
For any $\varepsilon\in(0,c_f)$, there exist some real numbers $L_{\varepsilon}>L$ and $R_{\varepsilon}>0$ such that, for any point $x_0\in\Omega$ with $|x_0|\ge R_{\varepsilon}+L_{\varepsilon}$, there holds
\be\label{eq+2.3}
v_{x_0,R_\epsilon}(t,x)\ge 1-2\delta \text{ for all $0\le t\le T_{\varepsilon}=\frac{|x_0|-R_{\varepsilon}-L_{\varepsilon}}{c_f-\varepsilon}$ and $x\in\overline{B(x_0,(c_f-\varepsilon)t)}\,(\subset\overline{\Omega})$}.
\ee
Furthermore, under the notations of Lemma~$\ref{lemma2.1}$, there is $X_\varepsilon>0$ such that, if $|x_0|\ge X_\varepsilon$ and if there are $\tau\ge0$ and a solution $v$ of the Cauchy problem~\eqref{vrx0} with $1\ge v(0,\cdot)\ge v_{x_0,R_\epsilon}(0,\cdot)$ in $\Omega$ and $v(t,\cdot)\ge1-2\delta$ in $\overline{B(0,L+\tilde{R}_{3}-\tilde{R}_{2})}\cap\overline{\Omega}\,$ for all $t\ge T_{\varepsilon}+\tau$, then there is $T'_{\varepsilon}\ge \tau$ such that
\be\label{eq+2.4}
v(t,x)\ge 1-4\delta \text{ for all $t\ge T_{\varepsilon}+T'_{\varepsilon}$ and } x\in\overline{B(x_0,(c_f-\varepsilon)(t-T'_{\varepsilon}))}\cap\overline{\Omega}.
\ee
\end{lemma}

\begin{proof} It is based on the maximum principle and on the construction of suitable sub-solutions close to some radially symmetric expanding fronts propagating with speeds less than but close to $c_f$. We first fix some parameters used throughout the proof and also in other proofs, and we then show~\eqref{eq+2.3} and~\eqref{eq+2.4}.

{\it Step 1: choice of some parameters.} Take any $\varepsilon\in (0,c_f)$. From the properties of the function~$\phi$, there is $C>0$ such that
\be\label{eq+2.2}
\phi\ge 1-\delta\text{ in $(-\infty,-C]$}\ \text{ and }\  \phi\le \delta \text{ in $[C,+\infty)$}.
\ee
Since $\phi'$ is continuous and negative in $\R$, there are some real numbers $k>0$ and then $\omega>0$ and~$\delta_\epsilon>0$  satisfying
\be\label{eq+2.4-}
\phi'\le-k\hbox{ in }[-C,C],\ \ k\,\omega\ge 2\delta+ 2\max_{[0,1]}|f'|,\ \hbox{ and }\ \delta_{\varepsilon}=\min\Big(\frac{\varepsilon\,k}{2\max_{[0,1]} |f'|},\frac{\delta}{2}\Big)>0.
\ee
There is also $C_{\varepsilon}>C>0$ such that
\be\label{ghsa1}
\phi\ge 1-\delta_{\varepsilon} \text{ in }(-\infty,-C_{\varepsilon}]\ \text{ and }\ \phi\le \delta_{\varepsilon} \text{ in $[C_{\varepsilon},+\infty)$}.
\ee

As in \cite[Lemma 4.1]{H}, it is easy to see that there is a $C^2$ function $h_{\varepsilon}: [0,+\infty)\rightarrow (0,+\infty)$ satisfying
\be\label{defheps1}
\left\{\baa{l}
0\le h'_{\varepsilon}\le 1 \text{ on $[0,+\infty)$},\vspace{3pt}\\
h'_{\varepsilon}=0 \text{ on a neighborhood of $0$},\ \ h_{\varepsilon}(r)=r \text{ on $[H_{\varepsilon},+\infty)$ for some $H_{\varepsilon}>0$},\vspace{3pt}\\
\displaystyle\frac{(N-1)\,h'_{\varepsilon}(r)}{r} +h''_{\varepsilon}(r)\le \frac{\varepsilon}{2}\text{ for all $r\in[0,+\infty)$}.\eaa\right.
\ee
Notice in particular that $r\le h_{\varepsilon}(r)\le r+h_{\varepsilon}(0)$ for all $r\ge 0$. Finally, define
\be\label{defRLeps}
R_{\varepsilon}=\max\big(H_{\varepsilon},h_{\varepsilon}(0)+\omega+C_{\varepsilon}+C\big)>0\ \hbox{ and }\ L_{\varepsilon}=L-C+C_{\varepsilon}>L>0.
\ee

{\it Step 2: proof of~\eqref{eq+2.3}.} Let $x_0$ be any point in $\R^N$ such that $|x_0|\ge R_\varepsilon+L_\varepsilon$ (which yields~$x_0\in\Omega$ by~\eqref{defL}) and denote
$$T_{\varepsilon}=\frac{|x_0|-R_{\varepsilon}-L_{\varepsilon}}{c_f-\varepsilon}$$
(notice that~\eqref{eq+2.3} is actually immediate if $T_\epsilon=0$, namely $|x_0|=R_\epsilon+L_\epsilon$). For all $(t,x)\in [0,+\infty)\times\R^N$, we set
$$\underline{v}(t,x)=\max\big(\phi(\underline{\zeta}(t,x))-\delta e^{-\delta t}-\delta_{\varepsilon},0\big),\ \hbox{ with }\ \underline{\zeta}(t,x)=h_{\varepsilon}(|x|)-(c_f-\varepsilon)t-\omega e^{-\delta t}+\omega-R_{\varepsilon}+C.$$
Let us check that $\underline{v}(t,x-x_0)$ is a sub-solution of the problem satisfied by $v_{x_0,R_\epsilon}(t,x)$ for $0\le t\le T_{\varepsilon}$ and $x\in\overline{\Omega}$.

Let us first check the initial and boundary conditions. At time $0$, since~$\phi\le1$ and $1-\delta-\delta_\varepsilon>0$, it is obvious that $\underline{v}(0,x-x_0)\le 1-\delta-\delta_{\varepsilon}\le v_{x_0,R_\epsilon}(0,x)$ for all~$x\in B(x_0,R_{\varepsilon})\,(\subset\overline{\Omega})$. Furthermore, if $x\in\overline{\Omega}\setminus B(x_0,R_{\varepsilon})$, then $|x-x_0|\ge R_{\varepsilon}\ge H_{\varepsilon}$ and $\underline{\zeta}(0,x-x_0)=h_{\varepsilon}(|x-x_0|)-R_{\varepsilon}+C\ge C$, hence $0\le\underline{v}(0,x-x_0)\le \max(\delta-\delta-\delta_{\varepsilon},0)=0\le v_{x_0,R_\epsilon}(0,x)$. Thus, $\underline{v}(0,x-x_0)\le v_{x_0,R_\epsilon}(0,x)$ for all $x\in\overline{\Omega}$. On the other hand, for $0\le t\le T_{\varepsilon}$ and~$x\in \overline{B(0,L)}\cap\overline{\Omega}$, one has that
$$\underline{\zeta}(t,x-x_0)\ge |x-x_0|-(|x_0|-R_{\varepsilon}-L_{\varepsilon})-R_{\varepsilon}+C\ge L_{\varepsilon}-L+C=C_\epsilon,$$
hence $0\le\underline{v}(t,x-x_0)\le\max(\delta_{\varepsilon}-\delta e^{-\delta t}-\delta_{\varepsilon},0)=0$. Thus, $\underline{v}(t,x-x_0)=0$ for $0\le t\le T_{\varepsilon}$ and~$x\in \overline{B(0,L)}\cap\overline{\Omega}$, hence $\underline{v}_\nu(t,x-x_0)=0$ for all $x\in\partial\Omega$ since $K=\R^N\setminus\Omega\subset B(0,L)$.

Let us now check that
$$\mathcal{L} \underline{v}(t,x-x_0):=\underline{v}_t(t,x-x_0)-\Delta \underline{v}(t,x-x_0)-f(\underline{v}(t,x-x_0))\le 0$$
for all $0<t\le T_{\varepsilon}$ and $x\in \Omega$ such that $\underline{v}(t,x-x_0)>0$. For any such $(t,x)$, there holds
$$\begin{array}{rcl}
\mathcal{L} \underline{v}(t,x\!-\!x_0)& \!\!=\!\! & \displaystyle\frac{\varepsilon}{2}\phi'(\underline{\zeta}(t,x\!-\!x_0))\!+\!f(\phi(\underline{\zeta}(t,x\!-\!x_0)))\!-\!f(\underline{v}(t,x\!-\!x_0))\!+\!\omega \delta e^{-\delta t} \phi'(\underline{\zeta}(t,x\!-\!x_0))\vspace{3pt}\\
& \!\!\!\! & \displaystyle+\delta^2 e^{-\delta t}+\Big(\frac{\varepsilon}{2}-\frac{(N-1)h'_{\varepsilon}(|x-x_0|)}{|x-x_0|}-h''_{\varepsilon}(|x-x_0|)\Big)\phi'(\underline{\zeta}(t,x-x_0))\vspace{3pt}\\
& \!\!\!\! & +(1-(h'_{\varepsilon}(|x-x_0|))^2)\,\phi''(\underline{\zeta}(t,x-x_0))
\end{array}$$
By the same analysis in the proof of \cite[Lemma 4.1, page 317]{H}, one can infer that
$$\begin{array}{rcl}
f(\phi(\underline{\zeta}(t,x-x_0)))-f(\phi(\underline{\zeta}(t,x-x_0))-\delta e^{-\delta t})+\omega \delta e^{-\delta t} \phi'(\underline{\zeta}(t,x-x_0))+\delta^2 e^{-\delta t} & \!\!\!\! & \vspace{3pt}\\
\displaystyle+\Big(\frac{\varepsilon}{2}\!-\!\frac{(N\!-\!1)h'_{\varepsilon}(|x\!-\!x_0|)}{|x\!-\!x_0|}\!-\!h''_{\varepsilon}(|x\!-\!x_0|)\Big)\phi'(\underline{\zeta}(t,x\!-\!x_0))\!+\!(1\!-\!(h'_{\varepsilon}(|x\!-\!x_0|))^2)\phi''(\underline{\zeta}(t,x\!-\!x_0)) & \!\!\!\le\!\!\! & 0.
\end{array}$$
Therefore, one has that
$$\mathcal{L} \underline{v}(t,x-x_0)\le \frac{\varepsilon}{2}\phi'(\underline{\zeta}(t,x-x_0))+f(\phi(\underline{\zeta}(t,x-x_0))-\delta e^{-\delta t})-f(\underline{v}(t,x-x_0)).$$
If $\underline{\zeta}(t,x-x_0)\le -C$, then $1-2\delta\le \phi(\underline{\zeta}(t,x-x_0))-\delta e^{-\delta t}<1$ and $1-4\delta\le \underline{v}(t,x-x_0)<1$. By~\eqref{d}, it follows that $f(\phi(\underline{\zeta}(t,x-x_0))-\delta e^{-\delta t})-f(\underline{v}(t,x-x_0))\le(f'(1)/2)\delta_{\varepsilon}<0$ and
$$\mathcal{L} \underline{v}(t,x-x_0)\le \frac{\varepsilon}{2}\phi'(\underline{\zeta}(t,x-x_0))+\frac{f'(1)}{2}\delta_{\varepsilon} \le 0,$$
since $\phi'(\underline{\zeta}(t,x-x_0))<0$. Similarly, one can get that $\mathcal{L} \underline{v}(t,x-x_0)\le 0$ if $\underline{\zeta}(t,x-x_0)\ge C$. Finally, if $-C\le \underline{\zeta}(t,x-x_0)\le C$, then $f(\phi(\underline{\zeta}(t,x-x_0))-\delta e^{-\delta t})-f(\underline{v}(t,x-x_0))\le \delta_{\varepsilon}\max_{[0,1]}|f'|$, and $\phi'(\underline{\zeta}(t,x-x_0))\le-k$. Thus, it follows from~\eqref{eq+2.4-} that $\mathcal{L} \underline{v}(t,x-x_0)\le-\varepsilon k/2+\delta_{\varepsilon}\max_{[0,1]}|f'|\le 0$.

By the comparison principle, one concludes that
\be\label{eq+2.5}
v_{x_0,R_\epsilon}(t,x)\ge \underline{v}(t,x-x_0)\ge \phi(\underline{\zeta}(t,x-x_0))-\delta e^{-\delta t}-\delta_{\varepsilon}\ \text{ for all $0\le t\le T_{\varepsilon}$ and $x\in\overline{\Omega}$}.
\ee
For $0\le t\le T_{\varepsilon}$ and $x\in\overline{B(x_0,(c_f-\varepsilon)t)}$, one has
$$\underline{\zeta}(t,x-x_0)\le (c_f-\varepsilon)t+h_{\varepsilon}(0)-(c_f- \varepsilon )t+\omega-R_{\varepsilon}+C \le -C_{\varepsilon}.$$
Then, it follows from~\eqref{eq+2.5} and $\delta_{\varepsilon}\le \delta/2$ that $v_{x_0,R_\epsilon}(t,x)\ge 1-2\delta_{\varepsilon}-\delta\ge 1-2\delta$ for all~$0\le t\le T_{\varepsilon}$ and $x\in\overline{B(x_0,(c_f-\varepsilon)t)}$. This completes the proof of~\eqref{eq+2.3}.

{\it Step 3: proof of~\eqref{eq+2.4}.}
Consider a $C^2$ function $\tilde{h}_{\varepsilon}:\R\rightarrow [0,1]$ such that for some $\xi_{\varepsilon}>0$,
$$\tilde{h}_{\varepsilon}=0\text{ in $(-\infty,-C-\xi_{\varepsilon}]$},\ \tilde{h}_{\varepsilon}=1\text{ in $[-C,+\infty)$, and } 0\le \tilde{h}'_{\varepsilon}\le 1 \text{ in $\R$}.$$
Even if it means increasing $\xi_{\varepsilon}$, one can assume without loss of generality that
\be\label{eq+2.6-}
\omega\delta^2\|\tilde{h}'_{\varepsilon}\|_{L^\infty(\R)}+2\|\tilde{h}'_{\varepsilon}\|_{L^\infty(\R)}\|\phi'\|_{L^\infty(\R)}+\delta\|\tilde{h}''_{\varepsilon}\|_{L^\infty(\R)}\le \frac{|f'(1)|}{4}\delta_{\varepsilon}.
\ee

With the same notations as above, it follows from~\eqref{eq+2.3} that, if $|x_0|\ge L_\epsilon+R_\epsilon$, then
\be\label{eq+2.6}
v_{x_0,R_\epsilon}(T_{\varepsilon},x)\ge 1-2\delta \text{ for all $x\in\overline{B(x_0,|x_0|-R_{\varepsilon}-L_{\varepsilon})}$}.
\ee
Let $\tilde{R}_1$, $\tilde{R}_2$, $\tilde{R}_3$ and $\tilde{T}$ be the positive numbers given in Lemma~\ref{lemma2.1}. Let $X_\epsilon>0$ be such that
\be\label{eq+2.7}
X_\epsilon>L_{\varepsilon}+L+R_{\varepsilon}+\tilde{R}_3,\ \ \frac{N-1}{X_\epsilon+L_{\varepsilon}+R_{\varepsilon}}\le \frac{\varepsilon}{2},\ \hbox{ and }\ \frac{(N-1)\,\|\phi'\|_{L^\infty(\R)}}{X_\epsilon+L_{\varepsilon}+R_{\varepsilon}}\le \frac{|f'(1)|}{4}\delta_{\varepsilon}.
\ee
Let $D_{\varepsilon}=L_{\varepsilon}+R_{\varepsilon}+\xi_{\varepsilon}+\omega+2C$. Since $\sup_{|x_0|\ge X_\epsilon,\,x\in\overline{B(x_0,|x_0|+D_{\varepsilon})}\cap\overline{\Omega}} d_{\Omega}(x,x_0)<+\infty$, there exists an integer $n_{\varepsilon}\ge 1$ such that, for any $|x_0|\ge X_\epsilon$ and any point $x\in\overline{B(x_0,|x_0|+D_{\varepsilon})}$ such that~$|x|\ge L+\tilde{R}_{3}-\tilde{R}_2$, there are $n_{\varepsilon}$ points $x^1,\cdots,x^{n_{\varepsilon}}$ in $\R^N$ satisfying
\begin{eqnarray*}
\left\{\baa{lll}
B(x^1,\tilde{R}_1)\subset B(x_0,|x_0|-L_{\varepsilon}-R_{\varepsilon}),\vspace{3pt}\\
B(x^i,\tilde{R}_{3})\subset \Omega  &\text{ for $1\le i\le n_{\varepsilon}$},\vspace{3pt}\\
B(x^{i+1},\tilde{R}_1)\subset B(x^i,\tilde{R}_2)  &\text{ for $1\le i\le n_{\varepsilon}-1$},\vspace{3pt}\\
x\in B(x^{n_{\varepsilon}},\tilde{R}_2).
\eaa
\right.
\end{eqnarray*}
Lemma \ref{lemma2.1} and~\eqref{eq+2.6} then yield $v_{x_0,R_\epsilon}(T_\epsilon,\cdot)\ge\tilde{v}_{x_1,\tilde{R}_1}(0,\cdot)$ in $\overline{\Omega}$ and $v_{x_0,R_\epsilon}(T_{\varepsilon}+\widetilde T,y)\ge 1-2\delta$ for all $y\in\overline{B(x^1,\tilde{R}_2)}$. Since $B(x^2,\tilde{R}_1)\subset B(x^1,\tilde{R}_2)$ and $B(x^2,\tilde{R}_3)\subset\Omega$, one can apply Lemma~\ref{lemma2.1} again and get that $v_{x_0,R_\epsilon}(T_{\varepsilon}+\widetilde T,y)\ge 1-2\delta$ for all $y\in\overline{B(x^2,\tilde{R}_2)}$. By induction, one gets that $v_{x_0,R_\epsilon}(T_{\varepsilon}+n_{\varepsilon}\widetilde T,y)\ge 1-2\delta$ for all $y\in\overline{B(x^{n_{\varepsilon}},\tilde{R}_2)}$. Thus,
$$v_{x_0,R_\epsilon}(T_{\varepsilon}+n_{\varepsilon}T,x)\ge 1-2\delta\ \hbox{ for any }|x_0|\ge R_\epsilon\hbox{ and }x\in\overline{B(x_0,|x_0|+D_{\varepsilon})}\hbox{ with }|x|\ge L+\tilde{R}_{3}-\tilde{R}_2.$$

Assume now, until the end of the proof of Lemma~\ref{lemma2.2}, that $|x_0|\ge X_\epsilon$ and that there exist~$\tau\ge0$ and a solution $v$ of the Cauchy problem~\eqref{vrx0} with $1\ge v(0,\cdot)\ge v_{x_0,R_\epsilon}(0,\cdot)$ in $\Omega$ and $v(t,x)\ge1-2\delta$ for all $t\ge T_{\varepsilon}+\tau$ and $x\in \overline{B(0,L+\tilde{R}_{3}-\tilde{R}_{2})\cap \Omega}$. Even if it means increasing $n_\epsilon$,
one can assume without loss of generality that $n_\epsilon \widetilde{T}\ge\tau$. Call now $T'_\epsilon=n_\epsilon \widetilde{T}$. From the results of the previous paragraph and the assumptions made in the present paragraph, one infers that
\be\label{eq+2.8}
v(T_{\varepsilon}+T'_{\varepsilon},x)\ge 1-2\delta\ \text{ for all $x\in\overline{B(x_0,|x_0|+D_{\varepsilon})\cap \Omega}$}.
\ee

For all $(t,x)\in[0,+\infty)\times\R^N$, define
$$\Phi(t,x)\!=\!\tilde{h}_{\varepsilon}(\underline{\xi}(t,x))\phi(\underline{\xi}(t,x))\!+\!(1\!-\!\tilde{h}_{\varepsilon}(\underline{\xi}(t,x)))(1\!-\!\delta)\hbox{ and }\underline{w}(t,x)\!=\!\max\!\big(\Phi(t,x)\!-\!\delta_{\varepsilon}\!-\!2\delta e^{-\delta (t\!-\!T_{\varepsilon}\!-\!T'_{\varepsilon})},0\big)$$
with
$$\underline{\xi}(t,x)=|x-x_0|-(c_f-\varepsilon)(t-T_{\varepsilon}-T'_{\varepsilon})-\omega e^{-\delta (t-T_{\varepsilon}-T'_{\varepsilon})}-|x_0|-L_{\varepsilon}-R_{\varepsilon}-\xi_{\varepsilon}-C,$$
and let us check that $\underline{w}(t,x)$ is a sub-solution of the problem satisfied by $v(t,x)$ for $t\ge T_{\varepsilon}+T'_{\varepsilon}$ and $x\in\overline{\Omega}$.

Let us first check the initial (at time $T_{\varepsilon}+T'_{\varepsilon}$) and boundary conditions. At time~$T_{\varepsilon}\!+\!T'_{\varepsilon}$,~\eqref{eq+2.8} implies that $\underline{w}(T_{\varepsilon}+T'_{\varepsilon},x)\le 1-2\delta\le v(T_{\varepsilon}+T'_{\varepsilon},x)$ for all $x\in B(x_0,|x_0|+D_{\varepsilon})\cap\overline{\Omega}$. Furthermore, if $x\in\overline{\Omega}\setminus B(x_0,|x_0|+D_{\varepsilon})$, then~$|x-x_0|\ge |x_0|+D_{\varepsilon}=|x_0|+L_\epsilon+R_\epsilon+\xi_\epsilon+\omega+2C$ and~$\underline\xi(T_{\varepsilon}+T'_{\varepsilon},x)\ge C$. Hence~$\tilde{h}_{\varepsilon}(\underline\xi(T_{\varepsilon}+T'_{\varepsilon},x))=1$ and $\underline{w}(T_{\varepsilon}+T'_{\varepsilon},x)\le \max(\delta-\delta_{\varepsilon}-2\delta,0)=0\le v(T_{\varepsilon}+T'_{\varepsilon},x)$ for all $x\in\overline{\Omega}\setminus B(x_0,|x_0|+D_{\varepsilon})$. Finally, $\underline{w}(T_{\varepsilon}+T'_{\varepsilon},x)\le v(T_{\varepsilon}+T'_{\varepsilon},x)$ for all $x\in\overline{\Omega}$. On the other hand, for any $t\ge T_{\varepsilon}+T'_{\varepsilon}$ and $x\in \overline{B(0,L)}\cap\overline{\Omega}$, it follows from the definition of $L_{\varepsilon}$ and $R_{\varepsilon}$ that $\underline{\xi}(t,x)\le L-L_{\varepsilon}-R_{\varepsilon}-\xi_{\varepsilon}-C\le -C-\xi_{\varepsilon}$, hence $\tilde{h}_{\varepsilon}(\underline{\xi}(t,x))=0$ and~$\underline{w}(t,x)=\max(1-\delta-\delta_{\varepsilon}-2\delta e^{-\delta(t-T_{\varepsilon}-T'_{\varepsilon})},0)=1-\delta-\delta_{\varepsilon}-2\delta e^{-\delta(t-T_{\varepsilon}-T'_{\varepsilon})}$. Therefore,~$\underline{w}_\nu(t,x)=0$ for all $t\ge T_\epsilon+T'_\epsilon$ and $x\in\partial \Omega$, since $\partial\Omega\subset B(0,L)$.

Let us now check that
$$\mathcal{L}\underline{w}(t,x)=\underline{w}_t(t,x)-\Delta \underline{w}(t,x)-f(\underline{w}(t,x))\le 0$$
for any point $(t,x)\in(T_{\varepsilon}+T'_{\varepsilon},+\infty)\times\Omega$ such that $\underline{w}(t,x)>0$. If $\underline{\xi}(t,x)<-C-\xi_{\varepsilon}$, one has that $\tilde{h}_{\varepsilon}(\underline{\xi}(t,x))=0$, $\Phi(t,x)=1-\delta$ and $\underline{w}(t,x)=1-\delta-\delta_{\varepsilon}-2\delta e^{-\delta(t-T_{\varepsilon}-T'_{\varepsilon})}\ge 1-4\delta>0$ (and these formulas hold in a neighborhood of the point $(t,x)$). By~\eqref{d}, it follows that
$$-f(\underline{w}(t,x))\le \frac{f'(1)}{2}(\delta+\delta_{\varepsilon}+2\delta e^{-\delta (t-T_{\varepsilon}-T'_{\varepsilon})})<f'(1)\,\delta e^{-\delta (t-T_{\varepsilon}-T'_{\varepsilon})}$$
and $\mathcal{L}\underline{w}(t,x)\le2\delta^2 e^{-\delta(t-T_{\varepsilon}-T'_{\varepsilon})}+f'(1)\delta e^{-\delta (t-T_{\varepsilon}-T'_{\varepsilon})}\le 0$, since $\delta<-f'(1)/2$. If $-C-\xi_{\varepsilon}\le \underline{\xi}(t,x)\le -C$, one has $0\le \tilde{h}_{\varepsilon}(\underline{\xi}(t,x))\le 1$ and~$1>\Phi(t,x)\ge 1-\delta$. A straightforward computation leads to
$$\baa{rcl}
\mathcal{L}\underline{w}(t,x) & = & \tilde{h}_{\varepsilon}(\underline{\xi}(t,x))f(\phi(\underline{\xi}(t,x)))-f(\underline{w}(t,x))+(\varepsilon+\omega\delta e^{-\delta (t-T_{\varepsilon}-T'_{\varepsilon})}) \tilde{h}_{\varepsilon}(\underline{\xi}(t,x))\phi'(\underline{\xi}(t,x))\vspace{3pt}\\
& & +\,(c_f-\varepsilon-\omega\delta e^{-\delta (t-T_{\varepsilon}-T'_{\varepsilon})}) \tilde{h}'_{\varepsilon}(\underline{\xi}(t,x))(1-\delta-\phi(\underline{\xi}(t,x)))-2\tilde{h}'_{\varepsilon}(\underline{\xi}(t,x))\phi'(\underline{\xi}(t,x))\vspace{3pt}\\
& & \displaystyle+\,\tilde{h}''_{\varepsilon}(\underline{\xi}(t,x))(1-\delta-\phi(\underline{\xi}(t,x)))+\tilde{h}'_{\varepsilon}(\underline{\xi}(t,x))(1-\delta-\phi(\underline{\xi}(t,x)))\frac{N-1}{|x-x_0|}\vspace{3pt}\\
& & \displaystyle-\,\tilde{h}_{\varepsilon}(\underline{\xi}(t,x))\phi'(\underline{\xi}(t,x))\frac{N-1}{|x-x_0|}+2\delta^2 e^{-\delta (t-T_{\varepsilon}-T'_{\varepsilon})}\vspace{3pt}\\
& \le & \tilde{h}_{\varepsilon}(\underline{\xi}(t,x))f(\phi(\underline{\xi}(t,x)))-f(\underline{w}(t,x))+ \omega \delta^2 \tilde{h}'_{\varepsilon}(\underline{\xi}(t,x))-2\tilde{h}'_{\varepsilon}(\underline{\xi}(t,x))\phi'(\underline{\xi}(t,x))\vspace{3pt}\\
& & \displaystyle+\,\delta|\tilde{h}''_{\varepsilon}(\underline{\xi}(t,x))|-\tilde{h}_{\varepsilon}(\underline{\xi}(t,x))\phi'(\underline{\xi}(t,x))\frac{N-1}{|x-x_0|}+2\delta^2 e^{-\delta (t-T_{\varepsilon}-T'_{\varepsilon})},\eaa$$
since $\varepsilon\le c_f$, $\tilde{h}_\epsilon\ge0$, $\tilde{h}'_{\varepsilon}\ge 0$, and $\phi'<0$. Since $\underline{\xi}(t,x)\le -C$ implies $\phi(\underline{\xi}(t,x))\ge 1-\delta$ and $\underline{w}(t,x)\ge 1-4\delta$, one gets that $f(\phi(\underline{\xi}(t,x)))\ge 0$ and
$$\begin{array}{l}
\tilde{h}_{\varepsilon}(\underline{\xi}(t,x))f(\phi(\underline{\xi}(t,x)))\!-\!f(\underline{w}(t,x))\vspace{3pt}\\
\qquad\qquad\qquad\qquad\le f(\phi(\underline{\xi}(t,x)))-f(\underline{w}(t,x))\vspace{3pt}\\
\qquad\qquad\qquad\qquad\le\displaystyle\frac{f'(1)}{2}(1\!-\!\tilde{h}_{\varepsilon}(\underline{\xi}(t,x)))(1\!-\!\delta\!-\!\phi(\underline{\xi}(t,x)))+\frac{f'(1)}{2}(\delta_{\varepsilon}\!+\!2\delta e^{-\delta (t-T_{\varepsilon}-T'_{\varepsilon})})\vspace{3pt}\\
\qquad\qquad\qquad\qquad\displaystyle\le\frac{f'(1)}{2}(\delta_{\varepsilon}+2\delta e^{-\delta (t-T_{\varepsilon}-T'_{\varepsilon})}).\end{array}$$
Since $\underline{\xi}(t,x)\ge -C-\xi_{\varepsilon}$, one has $|x-x_0|\ge |x_0|+L_{\varepsilon}+R_{\varepsilon}$ and it then follows from~\eqref{d},~\eqref{eq+2.6-} and~\eqref{eq+2.7} (remember that $|x_0|\ge X_\epsilon$) that
$$\baa{rcl}
\mathcal{L}\underline{w}(t,x) & \!\!\le\!\! & \displaystyle\frac{f'(1)}{2}(\delta_{\varepsilon}+2\delta e^{-\delta (t-T_{\varepsilon}-T'_{\varepsilon})})+ \omega \delta^2\tilde{h}'_{\varepsilon}(\underline{\xi}(t,x))-2\tilde{h}'_{\varepsilon}(\underline{\xi}(t,x))\phi'(\underline{\xi}(t,x))+\delta|\tilde{h}''_{\varepsilon}(\underline{\xi}(t,x))|\vspace{3pt}\\
& \!\!\!\! & \displaystyle+|\phi'(\underline{\xi}(t,x))|\frac{N-1}{|x_0|+L_{\varepsilon}+R_{\varepsilon}}+2\delta^2 e^{-\delta (t-T_{\varepsilon}-T'_{\varepsilon})}\vspace{3pt}\\
& \!\!\le\!\! & 0.\eaa$$
If $\underline{\xi}(t,x)>-C$, one has $\tilde{h}_{\varepsilon}(\underline{\xi}(t,x))=1$ and $\Phi(t,x)=\phi(\underline{\xi}(t,x))$ (and these formulas hold in a neighborhood of the point $(t,x)$). Notice that $|x-x_0|\ge |x_0|+L_{\varepsilon}+R_{\varepsilon}$ since~$\underline{\xi}(t,x)\ge -C\ge-C-\xi_\epsilon$. Therefore,
$$\baa{rcl}
\mathcal{L}\underline{w}(t,x) & = & \displaystyle f(\phi(\underline{\xi}(t,x)))\!-\!f(\underline{w}(t,x))\!+\!(\varepsilon\!+\!\omega\delta e^{-\delta (t-T_{\varepsilon}-T'_{\varepsilon})})  \phi'(\underline{\xi}(t,x))\!-\!\phi'(\underline{\xi}(t,x))\frac{N\!-\!1}{|x\!-\!x_0|}\vspace{3pt}\\
& & +2\delta^2 e^{-\delta (t-T_{\varepsilon}-T'_{\varepsilon})}\vspace{3pt}\\
& \le & \displaystyle f(\phi(\underline{\xi}(t,x)))-f(\underline{w}(t,x))+\Big(\frac{\varepsilon}{2}+\omega\delta e^{-\delta (t-T_{\varepsilon}-T'_{\varepsilon})}\Big)  \phi'(\underline{\xi}(t,x))+2\delta^2 e^{-\delta (t-T_{\varepsilon}-T'_{\varepsilon})}\eaa$$
by~\eqref{eq+2.7} (remember that $|x_0|\ge X_\epsilon$). If $-C\le \underline{\xi}(t,x)\le C$, one has $-\phi'(\underline{\xi}(t,x))\ge k$ and it then follows from~\eqref{eq+2.4-} that
$$\mathcal{L}\underline{w}(t,x)\le\Big(\max_{s\in[0,1]}|f'(s)|\Big)(\delta_{\varepsilon}+2\delta e^{-\delta (t-T_{\varepsilon}-T'_{\varepsilon})})-\Big(\frac{\varepsilon}{2}+\omega\delta e^{-\delta (t-T_{\varepsilon}-T'_{\varepsilon})}\Big)\,k+2\delta^2 e^{-\delta (t-T_{\varepsilon}-T'_{\varepsilon})}\le 0.$$
If $\underline{\xi}(t,x)>C$, one has $\Phi(t,x)=\phi(\underline{\xi}(t,x))\le \delta$ (and these formulas hold in a neighborhood of the point $(t,x)$), hence $f(\phi(\underline{\xi}(t,x)))-f(\underline{w}(t,x))\le(f'(0)/2)(\delta_{\varepsilon}+2\delta e^{-\delta (t-T_{\varepsilon}-T'_{\varepsilon})})$ by~\eqref{d}. Thus, it follows from the negativity of $\phi'$ and~\eqref{d} that
$$\mathcal{L}\underline{w}(t,x)\le\frac{f'(0)}{2}(\delta_{\varepsilon}+2\delta e^{-\delta (t-T_{\varepsilon}-T'_{\varepsilon})})+\Big(\frac{\varepsilon}{2}+\omega\delta e^{-\delta (t-T_{\varepsilon}-T'_{\varepsilon})}\Big)\, \phi'(\underline{\xi}(t,x))+2\delta^2 e^{-\delta (t-T_{\varepsilon}-T'_{\varepsilon})}\le 0.$$

In conclusion, one gets that $\mathcal{L}\underline{w}(t,x)=\underline{w}_t(t,x)-\Delta \underline{w}(t,x)-f(\underline{w}(t,x))\le 0$ for all $t>T_{\varepsilon}+T'_{\varepsilon}$ and $x\in\Omega$ such that $\underline{w}(t,x)>0$. By the comparison principle, it follows that $v(t,x)\ge \underline{w}(t,x)$ for all $t\ge T_{\varepsilon}+T'_{\varepsilon}$ and $x\in\overline{\Omega}$. Finally, for any $t\ge T_{\varepsilon}+T'_{\varepsilon}$ and $x\in\overline{B(x_0,(c_f-\varepsilon)(t-T'_{\varepsilon}))}\cap\overline{\Omega}$, one has
$$\underline{\xi}(t,x)\le (c_f-\varepsilon)T_{\varepsilon}-|x_0|-L_{\varepsilon}-R_{\varepsilon}-\xi_{\varepsilon}-C\le-\xi_\epsilon-C\le -C,$$
hence $v(t,x)\ge\underline{w}(t,x)\ge1\!-\!\delta\!-\!\delta_{\varepsilon}\!-\!2\delta\ge1\!-\!4\delta$. The proof of Lemma~\ref{lemma2.2} is thereby complete.
\end{proof}
\vskip 0.3cm

Combining Lemmas~\ref{lemma2.1} and~\ref{lemma2.2}, one then gets the following corollary.

\begin{corollary}\label{corollary2.2}
Let $L>0$ be such that $K\subset B(0,L)$ and let $R_1$, $R_2$ and $R_3$ be the positive real numbers given in Lemma~$\ref{lemma2.1}$. Then, there is $R>0$ such that the following holds. For any~$x_0\in\Omega$ with $B(x_0,R_3)\subset\Omega$ and for any $y\in\Omega\setminus B(0,L+R)=\R^N\setminus B(0,L+R)$, there is a real number~$\tau_{x_0,y}>0$ such that $v_{x_0,R_1}(t,y)\ge 1-2\delta$ for all $t\ge\tau_{x_0,y}$. Moreover, for any $M\ge0$,
$$\sup_{x_0\in\Omega,\,B(x_0,R_3)\subset\Omega,\,|y|\ge L+R,\,|x_0-y|\le M}\tau_{x_0,y}<+\infty.$$
Lastly, for any $x_0\in\Omega$ with $B(x_0,R_3)\subset\Omega$ ,
$$\liminf_{t\rightarrow +\infty} v_{x_0,R_1}(t,x)\ge 1-2\delta\ \text{ locally uniformly for $x\in \Omega\setminus B(0,L+R)$}.$$
\end{corollary}

\begin{proof}
Let $T>0$ be the positive real number given in Lemma~\ref{lemma2.1}. Set $\varepsilon=c_f/2$ and let~$L_{\varepsilon}$,~$R_{\varepsilon}$ be defined as in Lemma~\ref{lemma2.2} for this value $\varepsilon$. Define $R=c_fT/2+L_{\varepsilon}+R_{\varepsilon}+R_{3}-R_2$. Pick any $x_0\in\Omega$ with $B(x_0,R_3)\subset\Omega$, any $y\in \Omega\setminus B(0,L_\epsilon+R)$ and any $z\in\overline{B(y,R_{\varepsilon})}$ (hence, $|y|>\max\big(L_\epsilon+R_{\varepsilon}+R_{3}-R_2,c_fT/2+L_\epsilon+R_\epsilon\big)$ and $|z|>L_\epsilon+R_3-R_2$). It is straightforward to check that there exist an integer $k=k_{x_0,y}\ge 1$ (which depends on $|x_0-y|$, $R_\epsilon$ and the other parameters, but which can be chosen independently of $z\in\overline{B(y,R_\epsilon)}$) and $k$ points $x^1,\cdots,x^k$ in~$\R^N$ such that
\begin{eqnarray*}
\left\{\baa{lll}
&x^1=x_0,&\\
&B(x^i,R_3)\subset\Omega,\ &\text{ for $1\le i\le k$},\\
&B(x^{i+1},R_1)\subset B(x^i,R_2),\ &\text{ for $1\le i\le k-1$},\\
&z\in B(x^k,R_2).&
\eaa
\right.
\end{eqnarray*}
By Lemma~\ref{lemma2.1}, it follows that $v_{x_0,R_1}(T,x)\ge 1-\delta$ for all $x\in\overline{B(x_0,R_2)}=\overline{B(x^1,R_2)}$. Since $B(x^2,R_1)\subset B(x^1,R_2)$ and $B(x^2,R_3)\subset\Omega$, another application of Lemma~\ref{lemma2.1} yields $v_{x_0,R_1}(2T,x)\ge 1-\delta$ for all $x\in\overline{B(x^2,R_2)}$. By an immediate induction, there holds $v_{x_0,R_1}(kT,x)\ge 1-\delta$ for all $x\in\overline{B(x^k,R_2)}$. Thus, $v_{x_0,R_1}(kT,z)\ge 1-\delta$ for all $z\in\overline{B(y,R_{\varepsilon})}$. In fact, by using Lemma~\ref{lemma2.1} and an induction again, one infers that $v_{x_0,R_1}(nT,\cdot)\ge 1-\delta$ in~$\overline{B(y,R_{\varepsilon})}$ for any integer $n$ such that $n\ge k$ (one can pick $x^i=z$ for all $i\ge k$).

Since $|y|\ge L_{\varepsilon}+R_{\varepsilon}$, it follows from the comparison principle and Lemma~\ref{lemma2.2} that
$$v_{x_0,R_1}(nT\!+\!t,x)\!\ge\! v_{R_{\varepsilon}}(y;t,x)\!\ge\! 1\!-\!2\delta\hbox{ for all }0\!\le\! t\!\le\! T_\epsilon\!=\!\frac{2(|y|\!-\!L_\epsilon\!-\!R_\epsilon)}{c_f}\,(\ge\!T)\hbox{ and }x\!\in\!\overline{B(y,c_f t/2)},$$
and for any $n\ge k$. Since $T_\epsilon\ge T$, one infers that $v_{x_0,R_1}(t,y)\ge 1-2\delta$ for any $nT\le t\le (n+1)T$. Since $n$ is any integer such that $n\ge k$, one gets that $v_{x_0,R_1}(t,y)\ge 1-2\delta$ for all $t\ge kT$. Finally, define $\tau=kT>0$ and notice that the integer $k=k_{x_0,y}$ appearing in the proof can be chosen locally constant with respect to $y$ belonging to the (closed) set $\Omega\setminus B(0,L+R)=\R^N\setminus B(0,L+R)$ and also so that $\sup_{x_0\in\Omega,\,B(x_0,R_3)\subset\Omega,\,|y|\ge L+R,\,|x_0-y|\le M}k_{x_0,y}<+\infty$ for any $M\ge0$. The conclusion of Corollary~\ref{corollary2.2} follows.
\end{proof}
\vskip 0.3cm

The second key-lemma is concerned with some upper bounds for solutions which are initially equal to $1$ outside a large ball where they take a small value. For any $x_0\in\R^N$ and $R>0$, let~$w_{x_0,R}(t,x)$ denote the solution of the Cauchy problem
\be\label{defwR}\left\{\baa{lll}
(w_{x_0,R})_t-\Delta w_{x_0,R}=f(w_{x_0,R}), &t>0,\ x\in\overline{\Omega},\vspace{3pt}\\
(w_{x_0,R})_{\nu}=0, &t>0,\ x\in\partial\Omega,\eaa\right.
\ee
with initial condition
$$w_{x_0,R}(0,x)=\delta\hbox{ for }x\in B(x_0,R)\cap\overline{\Omega}\ \hbox{ and }\ w_{x_0,R}(0,x)=1\hbox{ for }x\in\overline{\Omega}\setminus B(x_0,R).$$
By constructing suitable super-solutions which are close to some contracting radially symmetric fronts moving with speeds larger than but close to $c_f$, we derive in the following lemma some upper bounds for the solutions $w_R$ in some contracting balls.

\begin{lemma}\label{lemma2.4}
For any $\varepsilon\in(0,c_f)$, there exist some real numbers $L_{\varepsilon}>L$ and $R_{\varepsilon}>0$ such that
\begin{itemize}
\item[{\rm{(i)}}] for any $R>R_{\varepsilon}$ and any point $x_0\in\Omega$ with $|x_0|\ge R+L_{\varepsilon}$, there holds
\end{itemize}
\be\label{eq+2.11}
w_{x_0,R}(t,x)\le 2\delta\ \text{for all $0\le t\le \tau_{1,\epsilon}=\frac{R\!-\!R_{\varepsilon}}{c_f\!+\!\varepsilon}$ and } x\in\overline{B(x_0,R\!-\!R_{\varepsilon}\!-\!(c_f\!+\!\varepsilon)t)}\,(\subset\overline{\Omega}),
\ee
\begin{itemize}
\item[{\rm{(ii)}}] for any $R>R_{\varepsilon}$ and any point $x_0\in\R^N$ with $R\ge R_\epsilon+|x_0|+L_\epsilon$, there holds
\end{itemize}
$$w_{x_0,R}(t,x)\!\le\!3\delta\text{ for all $0\!\le\!t\!\le\!\tau_{2,\epsilon}\!=\!\frac{R\!-\!R_{\varepsilon}\!-\!|x_0|\!-\!L_{\varepsilon}}{c_f+\varepsilon}$}\text{ and } x\in\overline{B(x_0,R\!-\!R_{\varepsilon}\!-\!(c_f\!+\!\varepsilon)t)}\cap\overline{\Omega}.$$
\end{lemma}

\begin{proof} We fix $\epsilon\in(0,c_f)$ throughout the proof.

(i) Since $\phi''(s)\sim b\,e^{-\lambda s}$ as $s\rightarrow +\infty$ with $\lambda=(c_f+(c_f^2-4f'(0))^{1/2})/2$ and some $b>0$, one can choose $C>0$ such that~\eqref{eq+2.2} holds together with $\phi''\ge 0$ on $[C,+\infty)$. Let $k>0$, $\omega>0$, $\delta_{\varepsilon}>0$, $C_{\varepsilon}>C>0$ and $h_{\varepsilon}(r)$ be defined as in~\eqref{eq+2.4-}-\eqref{defheps1} in Step~1 of the proof of Lemma \ref{lemma2.2}. Let also $R_\epsilon>0$ and $L_\epsilon>L$ be as in~\eqref{defRLeps} in Step~1 of the proof of Lemma \ref{lemma2.2}, namely
\be\label{defRLeps2}
R_{\varepsilon}= \max\big(H_{\varepsilon},h_{\varepsilon}(0)+\omega+C_{\varepsilon}+C\big)\ \hbox{ and }\ L_{\varepsilon}=L-C+C_{\varepsilon}>L.
\ee
Let $R>R_{\varepsilon}$ and define $\tau_{1,\epsilon}:=(R-R_{\varepsilon})/(c_f+\varepsilon)>0$.

For all $(t,x)\in[0,+\infty)\times\R^N$, we set
$$\overline{v}(t,x)=\min\big(\phi(\overline{\zeta}(t,x))\!+\!\delta e^{-\delta t}\!+\!\delta_{\varepsilon},1\big),\hbox{ with }\overline{\zeta}(t,x)=-h_{\varepsilon}(|x|)-(c_f\!+\!\varepsilon)t+\omega e^{-\delta t}-\omega+R-C.$$
Consider any point $x_0\in\Omega$ such that $|x_0|\ge R+L_\epsilon$ and let us check that $\overline{v}(t,x-x_0)$ is a super-solution of the problem satisfied by $w_{x_0,R}(t,x)$ for $0\le t\le \tau_{1,\epsilon}$ and $x\in\overline{\Omega}$.

Let us first check the initial and boundary conditions. At time $0$, it is obvious that $\overline{v}(0,x-x_0)\ge \delta+\delta_{\varepsilon}\ge w_{x_0,R}(0,x)$ if $x\in B(x_0,R)\,(\subset\overline{\Omega})$. If $x\in\overline{\Omega}\setminus B(x_0,R)$, then $|x-x_0|\ge R> R_{\varepsilon}\ge H_{\varepsilon}$ and $\overline{\zeta}(0,x-x_0)= -|x-x_0|+R-C\le -C$, hence $\overline{v}(0,x-x_0)\ge \min(1-\delta+\delta+\delta_{\varepsilon},1)=1\ge w_{x_0,R}(0,x)$. Thus, $\overline{v}(0,x-x_0)\ge w_{x_0,R}(0,x)$ for all $x\in\overline{\Omega}$. On the other hand, for $0\le t\le \tau_{1,\epsilon}$ and $x\in \overline{B(0,L)}\cap\overline{\Omega}$, one has $|x-x_0|\ge |x_0|-L\ge L_{\varepsilon}+R-L\ge H_{\varepsilon}$ and $\overline{\zeta}(t,x-x_0)\le -|x-x_0|+R-C\le L-L_{\varepsilon}-C= -C_{\varepsilon}$, hence $\overline{v}(t,x-x_0)\ge \min(1-\delta_{\varepsilon}+\delta e^{-\delta t}+\delta_{\varepsilon},1)=1$. Thus, $\overline{v}(t,x-x_0)\equiv 1$ for $0\le t\le \tau_{1,\epsilon}$ and~$x\in \overline{B(0,L)}\cap\overline{\Omega}$ and $\overline{v}_\nu(t,x-x_0)=0$ for all $x\in\partial\Omega$ (remember that $\partial\Omega\subset B(0,L)$).

Furthermore, by using similar arguments as in Step~2 of the proof of Lemma~\ref{lemma2.2}, and thanks to the proof of \cite[Lemma 4.2]{H}, one infers that
$$\mathcal{L} \overline{v}(t,x-x_0)=\overline{v}_t(t,x-x_0)-\Delta \overline{v}(t,x-x_0)-f(\overline{v}(t,x-x_0))\ge 0$$
for all $0<t\le \tau_{1,\epsilon}$ and $x\in\overline{\Omega}$ such that $\overline{v}(t,x-x_0)<1$.

By the comparison principle, one concludes that
\be\label{eq+2.13}
w_{x_0,R}(t,x)\le \overline{v}(t,x-x_0)\le \phi(\overline{\zeta}(t,x-x_0))+\delta e^{-\delta t}+\delta_{\varepsilon} \text{ for all $0\le t\le \tau_{1,\epsilon}$ and $x\in\overline{\Omega}$}.
\ee
For $0\le t\le \tau_{1,\epsilon}$ and $x\in\overline{B(x_0,R-R_{\varepsilon}-(c_f+\varepsilon)t)}\,(\subset\overline{\Omega})$, one has
$$\overline{\zeta}(t,x-x_0)\ge -|x-x_0|-h_{\varepsilon}(0)-(c_f+ \varepsilon )t-\omega+R-C\ge R_{\varepsilon}-h_{\varepsilon}(0)-\omega-C \ge C_{\varepsilon}.$$
Then, it follows from~\eqref{eq+2.13} and $\delta_{\varepsilon}\le \delta/2$ that $w_{x_0,R}(t,x)\le 2\delta_{\varepsilon}+\delta\le 2\delta$ for all $0\le t\le \tau_{1,\epsilon}$ and~$x\in\overline{B(x_0,R-R_{\varepsilon}-(c_f+\varepsilon)t)}$. This completes the proof of~\eqref{eq+2.11}.
\vskip 0.2cm

(ii) Let $C>0$, $k>0$, $\omega>0$, $\delta_{\varepsilon}>0$ and $C_{\varepsilon}>0$ be defined as above. Consider a $C^2$ function $\hat{h}_{\varepsilon}:\R\rightarrow [0,1]$ such that, for some $\hat{\xi}_{\varepsilon}>0$,
$$\hat{h}_{\varepsilon}=0\text{ in $[C+\hat{\xi}_{\varepsilon},+\infty)$},\ \ \hat{h}_{\varepsilon}=1\text{ in $(-\infty,C],\ $ and }\ -1\le \hat{h}'_{\varepsilon}\le 0 \text{ in $\R$}.$$
Even if it means increasing $\hat{\xi}_{\varepsilon}$, one can assume without loss of generality that
\be\label{eq+2.15}
\delta (2c_f+\omega\delta)\|\hat{h}'_{\varepsilon}\|_{L^\infty(\R)}+\delta\|\hat{h}''_{\varepsilon}\|_{L^\infty(\R)}+2\|\hat{h}'_{\varepsilon}\|_{L^\infty(\R)}\|\phi'\|_{L^\infty(\R)}\le \frac{|f'(0)|}{4}\delta_{\varepsilon}.
\ee

Let now $L_{\varepsilon}>L$ be such that
\be\label{eq+2.16}
\frac{N-1}{L_{\varepsilon}}\le \frac{\varepsilon}{2}\ \text{ and }\ \frac{N-1}{L_{\varepsilon}}\|\phi'\|_{L^\infty(\R)}\le \frac{|f'(0)|}{4}\delta_{\varepsilon}.
\ee
Denote here
$$R_{\varepsilon}=2C+\omega+\hat{\xi}_{\varepsilon}.$$
Even if it means increasing $L_\epsilon$ and $\hat{\xi}_{\varepsilon}$, one can assume without loss of generality that $L_\epsilon\ge L-C+C_\epsilon$ and that $R_\epsilon\ge\max(H_\epsilon,h_\epsilon(0)+\omega+C_\epsilon+C)$, so that part~(i) of the present lemma holds good with these values $L_\epsilon$ and $R_\epsilon$. Consider any $R>R_{\varepsilon}$ and any point $x_0\in\Omega$ such that $R\ge R_\epsilon+|x_0|+L_\epsilon$, and define
$$\tau_{2,\epsilon}=\frac{R-R_{\varepsilon}-|x_0|-L_{\varepsilon}}{c_f+\varepsilon}$$
(notice that the desired inequality of part (ii) of Lemma~\ref{lemma2.4} is immediate if $\tau_{2,\epsilon}=0$, namely $R=R_\epsilon+|x_0|+L_\epsilon$).

For all $(t,x)\in[0,+\infty)\times\R^N$, let us set
$$\Psi(t,x)=\hat{h}_{\varepsilon}(\overline{\xi}(t,x))\phi(\overline{\xi}(t,x))+(1-\hat{h}_{\varepsilon}(\overline{\xi}(t,x)))\delta\ \hbox{ and }\ \overline{w}(t,x)=\min\big(\Psi(t,x)+\delta_{\varepsilon}+\delta e^{-\delta t},1\big),$$
where
$$\overline{\xi}(t,x)=-|x-x_0|-(c_f+\varepsilon)t+\omega e^{-\delta t}-\omega +R-C,$$
and let us check that $\overline{w}(t,x)$ is a super-solution of the problem satisfied by $w_{x_0,R}(t,x)$ for~$0\le t\le \tau_{2,\epsilon}$ and $x\in\overline{\Omega}$.

Let us first check the initial and boundary conditions. One has $\overline{w}(0,x)\ge \delta_{\varepsilon}+ \delta\ge w_{x_0,R}(0,x)$ for all $x\in B(x_0,R)\cap \overline{\Omega}$. If $x\in\overline{\Omega}\setminus B(x_0,R)$, then $|x-x_0|\ge R$ and $\overline{\xi}(0,x)\le -R+R-C=-C$, hence $\hat{h}_\epsilon(\overline{\xi}(0,x))=1$ and $\overline{w}(0,x)\ge \min(1-\delta+\delta_{\varepsilon}+\delta,1)=1\ge w_{x_0,R}(0,x)$. Thus, $\overline{w}(0,x)\ge w_{x_0,R}(0,x)$ for all $x\in\overline{\Omega}$. On the other hand, for every $0<t\le \tau_{2,\epsilon}$ and $x\in\overline{B(0,L)}\cap\overline{\Omega}$, it follows from the definition of $L_{\varepsilon}$ and $R_{\varepsilon}$ that
$$\overline{\xi}(t,x)\ge -|x_0|-L-(R-R_{\varepsilon}-|x_0|-L_{\varepsilon})-\omega+R-C\ge R_{\varepsilon}+L_{\varepsilon}-L-\omega-C\ge C+\hat{\xi}_{\varepsilon},$$
whence $\hat{h}_{\varepsilon}(\overline{\xi}(t,x))=0$ and $\overline{w}(t,x)=\delta +\delta_{\varepsilon}+\delta e^{-\delta t}\,(<1)$. Therefore, $\overline{w}_\nu(t,x)=0$ for all $0<t\le \tau_{2,\epsilon}$ and $x\in \partial\Omega$.

Let us now check that $\mathcal{L}\overline{w}(t,x)=\overline{w}_t(t,x)-\Delta \overline{w}(t,x)-f(\overline{w}(t,x))\ge  0$ for all $0<t\le T_{\varepsilon}$ and~$x\in\Omega$ such that $\overline{w}(t,x)<1$. If $\overline{\xi}(t,x)>C+\hat{\xi}_{\varepsilon}$, then $\hat{h}_{\varepsilon}(\overline{\xi}(t,x))=0$ and $\Psi(t,x)\equiv \delta$, hence $\overline{w}(t,x)=\delta+\delta_{\varepsilon}+\delta e^{-\delta t}\le 4\delta<1$ (notice that this formula holds good in a neighborhood of the point $(t,x)$). Then, by~\eqref{d}, it follows that $f(\overline{w}(t,x))\le(f'(0)/2)(\delta+\delta_{\varepsilon}+\delta e^{-\delta t})$ and
$$\mathcal{L} \overline{w}(t,x)\ge -\delta^2 e^{-\delta t}-\frac{f'(0)}{2}(\delta+\delta_{\varepsilon}+\delta e^{-\delta t})\ge 0.$$
If $C\le \overline{\xi}(t,x)\le C+\hat{\xi}_{\varepsilon}$, one has $\phi(\overline{\xi}(t,x))\le \delta$ and $\overline{w}(t,x)=\Psi(t,x)+\delta_{\varepsilon}+\delta e^{-\delta t}\le 4\delta$. Then, a straightforward computation leads to
\be\baa{rcl}\label{Lbarw}
\mathcal{L}\overline{w}(t,x) & \!\!=\!\! & (c_f\!+\!\varepsilon\!+\!\omega\delta e^{-\delta t}) \hat{h}'_{\varepsilon}(\overline{\xi}(t,x))(\delta\!-\!\phi(\overline{\xi}(t,x)))\!-\!(\varepsilon\!+\!\omega\delta e^{-\delta t}) \hat{h}_{\varepsilon}(\overline{\xi}(t,x))\phi'(\overline{\xi}(t,x))\vspace{3pt}\\
& \!\!\!\! & +\hat{h}''_{\varepsilon}(\overline{\xi}(t,x))(\delta-\phi(\overline{\xi}(t,x)))-2\hat{h}'_{\varepsilon}(\overline{\xi}(t,x))\phi'(\overline{\xi}(t,x))-\delta^2 e^{-\delta t}\vspace{3pt}\\
& \!\!\!\! & \displaystyle-\hat{h}'_{\varepsilon}(\overline{\xi}(t,x))(\delta-\phi(\overline{\xi}(t,x)))\frac{N-1}{|x-x_0|}+\hat{h}_{\varepsilon}(\overline{\xi}(t,x))\phi'(\overline{\xi}(t,x))\frac{N-1}{|x-x_0|}\vspace{3pt}\\
& \!\!\!\! & +\hat{h}_{\varepsilon}(\overline{\xi}(t,x))f(\phi(\overline{\xi}(t,x)))-f(\overline{w}(t,x))\vspace{3pt}\\
& \!\!\ge\!\! & \delta (2c_f+\omega\delta)\hat{h}'_{\varepsilon}(\overline{\xi}(t,x))-|\hat{h}''_{\varepsilon}(\overline{\xi}(t,x))|\delta-2\hat{h}'_{\varepsilon}(\overline{\xi}(t,x))\phi'(\overline{\xi}(t,x)-\delta^2 e^{-\delta t}\vspace{3pt}\\
& \!\!\!\! & \displaystyle+\hat{h}_{\varepsilon}(\overline{\xi}(t,x))\phi'(\overline{\xi}(t,x))\frac{N-1}{|x-x_0|}+\hat{h}_{\varepsilon}(\overline{\xi}(t,x))f(\phi(\overline{\xi}(t,x)))-f(\overline{w}(t,x)),\eaa
\ee
since $\varepsilon\le c_f$, $\hat{h}_\epsilon\ge0$, $\hat{h}'_{\varepsilon}\le 0$ and $\phi'<0$. Moreover, since $\phi(\overline\xi(t,x))\le \delta$ and $\overline{w}(t,x)\le 4\delta$, one has $f(\phi(\overline{\xi}(t,x)))<0$ and
\begin{align*}
\hat{h}_\epsilon(\overline{\xi}(t,x))f(\phi(\overline{\xi}(t,x)))-f(\overline{w}(t,x))\ge & \ f(\phi(\overline{\xi}(t,x)))-f(\overline{w}(t,x))\\
\ge& \ -\frac{f'(0)}{2}(1\!-\!\hat{h}_{\varepsilon}(\overline{\xi}(t,x)))(\delta\!-\!\phi(\overline{\xi}(t,x)))\!-\!\frac{f'(0)}{2}(\delta_{\varepsilon}\!+\!\delta e^{-\mu t})\\
\ge & \ -\frac{f'(0)}{2}(\delta_{\varepsilon}+\delta e^{-\mu t}).
\end{align*}
Since $\overline{\xi}(t,x)\le C+\hat{\xi}_{\varepsilon}$, one has $|x\!-\!x_0|\!\ge \!-(c_f\!+\!\varepsilon)\tau_{2,\epsilon}\!-\!\omega\!+\!R\!-\!2C-\hat{\xi}_{\varepsilon}=R_{\varepsilon}\!+\!L_{\varepsilon}\!+\!|x_0|\!-\!\omega\!-\!2C\!-\!\hat{\xi}_{\varepsilon}\ge L_{\varepsilon}$. Therefore, it follows from~\eqref{d},~\eqref{eq+2.15},~\eqref{eq+2.16} and the negativity of $\phi'$ that
$$\baa{rcrl}
\mathcal{L}\overline{w}(t,x) & \ge & \delta (2c_f+\omega\delta)\hat{h}'_{\varepsilon}(\overline{\xi}(t,x))-|\hat{h}''_{\varepsilon}(\overline{\xi}(t,x))|\delta-2\hat{h}'_{\varepsilon}(\overline{\xi}(t,x))\phi'(\overline{\xi}(t,x)-\delta^2 e^{-\delta t} & \vspace{3pt}\\
& & \displaystyle+\hat{h}_{\varepsilon}(\overline{\xi}(t,x))\phi'(\overline{\xi}(t,x))\frac{N-1}{L_{\varepsilon}}-\frac{f'(0)}{2}(\delta_{\varepsilon}+\delta e^{-\mu t}) & \ge\ 0.\eaa$$
If $\overline{\xi}(t,x)<C$, one has $\hat{h}_{\varepsilon}(\overline{\xi}(t,x))=1$ and $\overline{w}(t,x)=\phi(\overline{\xi}(t,x))+\delta_{\varepsilon}+\delta e^{-\delta t}$, and this formula holds in a neighborhood of the point $(t,x)$. Notice also that $|x-x_0|\ge|x_0|+L_\epsilon+\hat{\xi}_\epsilon\ge L_{\varepsilon}$, since $\overline{\xi}(t,x)<C$. As a consequence,
$$\begin{array}{rcl}
\mathcal{L}\overline{w}(t,x) & = & \displaystyle\Big(\frac{N-1}{|x-x_0|}-\varepsilon\Big)\phi'(\overline{\xi}(t,x))-\omega\delta e^{-\delta t}\phi'(\overline{\xi}(t,x))-\delta^2 e^{-\delta t}+f(\phi(\overline{\xi}(t,x)))-f(\overline{w}(t,x))\vspace{3pt}\\
& \ge & \displaystyle-\frac{\varepsilon}{2}\phi'(\overline{\xi}(t,x))-\omega\delta e^{-\delta t}\phi'(\overline{\xi}(t,x))-\delta^2 e^{-\delta t}+f(\phi(\overline{\xi}(t,x)))-f(\overline{w}(t,x))
\end{array}$$
by~\eqref{eq+2.16}. If $-C\le \overline{\xi}(t,x)<C$, one has $-\phi'(\overline{\xi}(t,x))\ge k$ and it follows from~\eqref{eq+2.4-} that
$$\mathcal{L}\overline{w}(t,x)\ge \frac{\varepsilon}{2}k +k\omega\delta e^{-\delta t}-\delta^2 e^{-\delta t}-\max_{[0,1]} |f'|(\delta_{\varepsilon}+\delta e^{-\delta t})\ge 0.$$
If $\overline{\xi}(t,x)<-C$, one has $\phi(\overline{\xi}(t,x))\ge 1-\delta$ and $\overline{w}(t,x)\ge 1-\delta$. By~\eqref{d}, one infers that~$f(\phi(\overline{\xi}(t,x)))-f(\overline{w}(t,x))\ge-(f'(1)/2)(\delta_{\varepsilon}+\delta e^{-\delta t})$. Thus, it follows from~\eqref{d} and the negativity of $\phi'$ that $\mathcal{L}\overline{w}(t,x)\ge -\delta^2 e^{-\delta t}-(f'(1)/2)(\delta_{\varepsilon}+\delta e^{-\delta t})\ge 0$.

In conclusion, one gets $\mathcal{L}\overline{w}(t,x)=\overline{w}_t(t,x)-\Delta \overline{w}(t,x)-f(\overline{w}(t,x))\ge0$ for all $0<t\le \tau_{2,\epsilon}$ and $x\in \Omega$ such that $\overline{w}(t,x)<1$. The comparison principle yields~$w_{x_0,R}(t,x)\le \overline{w}(t,x)$ for all $0\le t\le \tau_{2,\epsilon}$ and $x\in\overline{\Omega}$. Then, for any $0\le t\le \tau_{2,\epsilon}$ and $x\in\overline{B(x_0,R-R_{\varepsilon}-(c_f+\varepsilon)t)}\cap\overline{\Omega}$, one has $\overline{\xi}(t,x)\ge -R+R_{\varepsilon}-\omega+R-C= C+\hat{\xi}_{\varepsilon}$, hence $\Psi(t,x)\le\delta$ and $w_{x_0,R}(t,x)\le \overline{w}(t,x)\le \delta+\delta_{\varepsilon}+\delta e^{-\delta t}\le 3\delta$. The proof is thereby complete.
\end{proof}

\begin{remark}\label{remLReps1}{\rm The proofs of Lemmas~\ref{lemma2.2} and~\ref{lemma2.4} show that the conclusions hold for all real numbers $L_\epsilon$ and $R_\epsilon$ larger than the ones appearing in the statements. In other words, one can choose $L_\epsilon>L$ and $R_\epsilon>0$ in such a way that the conclusions of these two lemmas hold simultaneously.}
\end{remark}


\subsection{Proof of Theorem \ref{th1}}\label{sec2.2}

The first step in the proof of Theorem~\ref{th1} is the claim that any transition front connecting $0$ and~$1$ for~\eqref{eq1.1} and propagating completely is in some sense located far away from the obstacle for very negative and very positive times.

\begin{lemma}\label{lemma2.6}
Let $L>0$ is given in~\eqref{defL}. Let $u$ be a transition front of~\eqref{eq1.1} connecting $0$ and~$1$ and satisfying~\eqref{complete}, that is, $u(t,x)\rightarrow 1$ as $t\rightarrow +\infty$ locally uniformly with respect to $x\in\overline{\Omega}$. Then, for every $\rho\ge 0$, there exist some real numbers $T_1<T_2$ such that
$$\Omega\cap B(0,L+\rho)\subset\Omega^-_t\hbox{ for all }t\le T_1,\ \hbox{ and }\ \Omega\cap B(0,L+\rho)\subset \Omega^+_t\hbox{ for all }t\ge T_2.$$
\end{lemma}

\begin{proof}
First of all, for any $\rho\ge 0$, since $\Omega\cap B(0,L+\rho)$ is bounded and $u(t,x)\rightarrow 1$ as $t\rightarrow +\infty$ locally uniformly with respect to $x\in\overline{\Omega}$, we claim the existence of $T_2\in\R$ such that
$$\Omega\cap B(0,L+\rho)\subset \Omega^+_t\text{ for all $t\ge T_2$}.$$
Indeed, otherwise, there are some sequences $(t_n)_{n\in\N}$ in $\R$ and $(x_n)_{n\in\N}$ in $\Omega\cap B(0,L+\rho)$ such that $t_n\to+\infty$ as $n\to+\infty$, and $x_n\not\in\Omega^+_{t_n}$ for every $n\in\N$. From~\eqref{eq1.4} and~\eqref{eq1.7}, there is a sequence $(y_n)_{n\in\N}$ such that $\sup_{n\in\N}d_\Omega(y_n,x_n)<+\infty$ while $y_n\in\Omega^-_{t_n}$ and $u(t_n,y_n)\le1/2$ for every~$n\in\N$. The Harnack inequality applied to $1-u$ implies that $\sup_{n\in\N}u(t_n+1,y_n)<1$. On the other hand, the convergence $u(t,\cdot)\to1$ locally uniformly in $\overline{\Omega}$ as $t\to+\infty$ and the boundedness of the sequences $(x_n)_{n\in\N}$ and $(d_\Omega(y_n,x_n))_{n\in\N}$ yield $\lim_{n\to+\infty}u(t_n+1,y_n)=1$, a contradiction. As a consequence, the second part of the conclusion of Lemma~\ref{lemma2.6} has been proved.

As far as the first assertion is concerned, we actually prove it for every $\rho\ge0$ large enough, and it will then hold automatically for every $\rho\ge0$. Let $R_1$, $R_2$, $R_3$ and $T$ be the four positive real numbers given in Lemma~\ref{lemma2.1}, let $R>0$ be given in Corollary~\ref{corollary2.2}, and let $x_0$ be a fixed point in $\Omega$ such that $B(x_0,R_3)\subset\Omega$. Consider now any positive real number $\rho$ large enough so that
\be\label{choicerho}
\rho>R_3\ \hbox{ and }L+\rho\ge|x_0|+R_1+M_\delta,
\ee
and assume by way of contradiction that the first conclusion in Lemma~\ref{lemma2.6} is false for that value of $\rho$. Then, from~\eqref{eq1.3}, two cases may occur: either there is a sequence $(t_n)_{n\in\N}\rightarrow -\infty$ such that $\Omega\cap B(0,L+\rho)\cap \Gamma_{t_n}\neq\emptyset$ for each $n\in\N$, or there is a sequence $(t_n)_{n\in\N}\rightarrow -\infty$ such that~$\Omega\cap B(0,L+\rho)\subset \Omega^+_{t_n}$ for each $n\in\N$.

{\it Case~1: $\Omega\cap B(0,L+\rho)\cap \Gamma_{t_n}\neq \emptyset$ for each $n\in\N$.} For each $n\in\N$, pick a point $x_n\in\Omega\cap B(0,L\!+\!\rho)\cap\Gamma_{t_n}$. By~\eqref{eq1.4}, for each $n\!\in\!\N$, there exist $y_n\!\in\!\Omega^+_{t_n}$ and $r\!>\!0$ such~that
$$d_{\Omega}(x_n,y_n)\le r\ \text{ and }\ d_{\Omega}(y_n,\Gamma_{t_n})\ge \max\big(2(L+\rho)+M_{\delta}+R_1,2(L+\rho)+R_2\big).$$
This implies that $B(y_n,R_2)\subset \R^N\setminus B(0,L+\rho)$ and $d_{\Omega}(y,\Gamma_{t_n})\ge M_{\delta}$ for all $y\in B(y_n,R_1)$. Since~$\rho\ge R_3$, one gets that $B(y_n,R_1)\subset B(y_n,R_2)\subset B(y_n,R_3)\subset B(y_n,\rho)\subset\Omega\setminus B(0,L)\subset\Omega$. Thus,
\be\label{eq-2.12}
u(t_n,y)\ge 1-\delta\ \text{ for all $y\in B(y_n,R_1)$}.
\ee
Since $x_n\in B(0,L+\rho)$ and $d_{\Omega}(x_n,y_n)\le r$, the sequence $(d_{\Omega}(x_0,y_n))_{n\in\N}$ is bounded. Thus, there exists an integer $k\ge 1$ such that, for every $n\in\mathbb{N}$,  there exist $k$ points $y_n^1,\cdots,y_n^k$ in $\R^N$ satisfying $y_n^1=y_n$, $y_n^k=x_0$ and
\begin{eqnarray*}
\left\{\baa{lll}
B(y_n^i,R_3)\subset \Omega &\text{ for $1\le i\le k$,}\\
B(y_n^{i+1},R_1)\subset B(y_n^i,R_2) &\text{ for $1\le i\le k-1$}.
\eaa
\right.
\end{eqnarray*}
By~\eqref{eq-2.12} and Lemma~\ref{lemma2.1}, the comparison principle yields $u(t_n+T,y)\ge v_{y_n,R_1}(T,y)\ge 1-\delta$ for all $y\in\overline{B(y_n,R_2)}$. Since $B(y_n^2,R_1)\subset B(y_n,R_2)$ and $B(y_n^2,R_3)\subset \Omega$, one can apply Lemma~\ref{lemma2.1} again and get that $u(t_n,+2T,y)\ge v_{y^2_n,R_1}(T,y)\ge 1-\delta$ for all $y\in \overline{B(y_n^2,R_2)}$. By induction, it follows that $u(t_n+kT,y)\ge 1-\delta$ for all $y\in\overline{B(y_n^k,R_2)}=\overline{B(x_0,R_2)}$. This implies that $u(t_n+kT,y)\ge v_{x_0,R_1}(0,y)$ for all $y\in\overline{\Omega}$. By the comparison principle, one gets that $u(t,y)\ge v_{x_0,R_1}(t-t_n-kT,y)$ for all $y\in\overline{\Omega}$ and $t>t_n+kT$. Since $t_n\rightarrow -\infty$ as $n\rightarrow+\infty$, Corollary~\ref{corollary2.2} then yields
$$u(t,y)\ge \liminf_{t_n\rightarrow -\infty} v_{x_0,R_1}(t-t_n-kT,y)\ge 1-2\delta>0\ \text{ for all $t\in\R$}.$$
Since this holds for every $y\in\Omega\setminus B(0,L+R)$, this contradicts Definition~\ref{TF} and~\eqref{eq1.4}.

{\it Case~2: $\Omega\cap B(0,L+\rho)\subset \Omega^+_{t_n}$ for each $n\in\N$.} By~\eqref{choicerho}, it follows that $u(t_n,x)\ge 1-\delta$ for all $x\in B(x_0,R_1)$. The comparison principle then implies that $u(t,x)\ge v_{x_0,R_1}(t-t_n,x)$ for all $x\in\overline{\Omega}$ and $t>t_n$. Since $t_n\rightarrow -\infty$ as $n\rightarrow+\infty$, Corollary~\ref{corollary2.2} yields
$$u(t,x)\ge \liminf_{t_n\rightarrow -\infty} v_{x_0,R_1}(t-t_n,x)\ge 1-2\delta>0\ \text{ for all $t\in\R$ and $x\in \Omega\setminus B(0,L+R)$},$$
which again contradicts Definition \ref{TF} and~\eqref{eq1.4}.
\end{proof}

\begin{remark}\label{rem2.6}{\rm The above proof shows that the first conclusion of Lemma~\ref{lemma2.6} holds for any transition front of~\eqref{eq1.1} connecting $0$ and $1$, even if it is not assumed to satisfy~\eqref{complete}.}
\end{remark}

Before carrying out the proof of Theorem~\ref{th1}, let us notice that arguments based on the Harnack inequality and standard parabolic estimates, similar to those of~\cite[Lemma~3 and Remark~3]{GH} and~\cite[Propositions~1.2 and~4.2]{HR2}, imply that, for any transition front $u(t,x)$ connecting~$0$ and~$1$, the interfaces $(\Gamma_{t})_{t\in\R}$ have uniformly bounded local oscillations in the sense of the Hausdorff distance, namely
\be\label{eq+2.19}
\forall\,\sigma>0,\ \sup\big\{d_{\Omega}(x,\Gamma_s) : t\in\R,\ x\in\Gamma_t,\ |t-s|\le \sigma\big\}<+\infty.
\ee
\vskip 0.3cm

\begin{proof}[Proof of Theorem \ref{th1}] It is based on the key-results of Section~\ref{sec2.1} and the comparison of the transition front $u$ with some solutions of the type $v_{x_0,R}$ and $w_{x_0,R}$ defined in~\eqref{vrx0} and~\eqref{defwR} far away from the interfaces $\Gamma_t$ at some times $t$.

Our goal is to prove that $d(\Gamma_t,\Gamma_s)/|t-s|\to c_f$ as $|t-s|\rightarrow +\infty$. We shall first show
\be\label{eq-2.13}
\liminf_{|t-s|\rightarrow +\infty} \frac{d_\Omega(\Gamma_t,\Gamma_s)}{|t-s|}\ge c_f
\ee
and then
\be\label{eq-2.14}
\limsup_{|t-s|\rightarrow +\infty} \frac{d_\Omega(\Gamma_t,\Gamma_s)}{|t-s|}\le c_f.
\ee
Throughout the proof, $\delta\in(0,1/4)$ is given as in~\eqref{d}.
\vskip 0.3cm

{\it Step 1: some notations.} Consider any $\varepsilon\in (0,c_f/3)$. Let $R_1>0$, $R_2>0$ and $R_3>0$ be as in Lemma~\ref{lemma2.1}, and let $L_{\varepsilon}>L$, $R_{\varepsilon}>0$ and $X_\epsilon>0$ be such that the conclusions of Lemmas~\ref{lemma2.2} and~\ref{lemma2.4} hold (see Remark~\ref{remLReps1}). Define
\be\label{eq-2.15-}
D_{\varepsilon}=L_{\varepsilon}+r_{M_{\delta}+R_{\varepsilon}}+2R_{\varepsilon}+X_\epsilon+M_{\delta}+R_3-R_2>L>0.
\ee
By Lemma~\ref{lemma2.6}, there exist two real numbers $T^1_{\varepsilon}<T^2_{\varepsilon}$ such that
\be\label{eq-2.15}
\Omega\cap B(0,D_{\varepsilon})\subset \Omega_t^- \text{ for all $t\le T^1_{\varepsilon}\ $ and }\ \Omega\cap B(0,D_{\varepsilon})\subset \Omega_t^+ \text{ for all $t\ge T^2_{\varepsilon}$}.
\ee

{\it Step 2: the lower estimate~\eqref{eq-2.13}.} Firstly, we claim that
\be\label{eq-2.16}
\liminf_{t<s\le T^1_{\varepsilon},\ |t-s|\rightarrow +\infty} \frac{d_{\Omega}(\Gamma_{t},\Gamma_{s})}{|t-s|}\ge c_f-2\varepsilon.
\ee
Assume by contradiction that there exist two sequences $(t_k)_{k\in\mathbb{N}}$ and $(s_k)_{k\in\mathbb{N}}$ such that $t_k<s_k\le T^1_{\varepsilon}$, $s_k-t_k\rightarrow +\infty$ as $k\rightarrow +\infty$ and $d_{\Omega}(\Gamma_{t_k},\Gamma_{s_k})<(c_f-2\varepsilon)(s_k-t_k)$ for all~$k\in\mathbb{N}$. By the definition of the distance, there exist two sequences $(x_k)_{k\in\N}$ and $(z_k)_{k\in\N}$ such that $x_k\in\Gamma_{t_k}$, $z_k\in\Gamma_{s_k}$, and $d_{\Omega}(x_k,z_k)<(c_f-2\varepsilon)(s_k-t_k)$ for all $k\in\mathbb{N}$. By~\eqref{eq1.4}-\eqref{eq1.5}, for each $k\in\N$, there is $y_k\in\Omega^+_{t_k}$ such that
\be\label{eq-2.17}
d_{\Omega}(x_k,y_k)\le r_{M_{\delta}+R_{\varepsilon}} \text{ and } d_{\Omega}(y_k,\Gamma_{t_k})\ge M_{\delta}+R_{\varepsilon}.
\ee
By~\eqref{eq-2.15}, one has that $\inf\big\{|x|: x\in\Gamma_{t_k}\big\}\ge D_{\varepsilon}$ for all $k\in\mathbb{N}$. Thus, it follows from~\eqref{eq-2.15-} and~\eqref{eq-2.17} that $B(y_k,R_{\varepsilon})\subset \R^N\setminus B(0,L_{\varepsilon})$ and $d_{\Omega}(B(y_k,R_{\varepsilon}),\Gamma_{t_k})\ge M_{\delta}$. Therefore, $u(t_k,y)\ge 1-\delta$ for all $y\in B(y_k,R_{\varepsilon})$. Then, by the comparison principle, one gets that $u(t,y)\ge v_{R_{y_k,\varepsilon}}(t-t_k,y)$ for every $t>t_k$ and $y\in\overline{\Omega}$. By Lemma \ref{lemma2.2}, for each $k\in\N$, there is $T_{\varepsilon,k}=(|y_k|-R_{\varepsilon}-L_{\varepsilon})/(c_f-\varepsilon)$ such that
\be\label{eq-2.19}
u(t,y)\!\ge\!v_{y_k,R_{\varepsilon}}(t-t_k,y)\!\ge\!1\!-\!2\delta\text{ for all $0\!\le\!t\!-\!t_k\!\le\!T_{\varepsilon,k}$ and $y\!\in\!\overline{B(y_k,(c_f\!-\!\varepsilon)(t\!-\!t_k))}\subset\overline{\Omega}$}.
\ee

If $T_{\varepsilon}^1-t_k> T_{\varepsilon,k}$ for some $k\in\N$, it follows from~\eqref{eq-2.19} that $u(t_k+T_{\varepsilon,k},y)\ge 1-2\delta$ for all $y\in\overline{B(y_k,(c_f-\varepsilon)
T_{\varepsilon,k})}$. Since $T_{\varepsilon,k}=
(|y_k|-R_{\varepsilon}-L_{\varepsilon})/(c_f-\varepsilon)$, there is a point $\tilde{y}_k\in\overline{B(y_k,(c_f-\varepsilon)T_{\varepsilon,k})}$ such that $L_{\varepsilon}+R_{\varepsilon}\le |\tilde{y}_k|\le L_{\varepsilon}+2R_{\varepsilon}$ and
\be\label{eq-2.21}
u(t_k+T_{\varepsilon,k},\tilde{y}_k)\ge 1-2\delta.
\ee
However, we know that $\Omega\cap B(0,R_{\varepsilon})\subset\Omega\cap B(0,D_\epsilon)\subset \Omega^-_{t_k+T_{\varepsilon,k}}$. By~\eqref{eq-2.15-}, it follows that $\tilde{y}_k\in \Omega\cap B(0,D_{\varepsilon})$, that $d(\tilde{y}_k,\Gamma_{t_k+T_{\varepsilon,k}})\ge M_{\delta}$ and that $u(t_k+T_{\varepsilon,k},\tilde{y}_k)\le \delta$, contradicting~\eqref{eq-2.21}.

Therefore, $0<s_k-t_k\le T^1_{\varepsilon}-t_k\le T_{\varepsilon,k}$ for each $k\in\N$. Then, it follows from~\eqref{eq-2.19} that $\overline{B(y_k,(c_f-\varepsilon)(s_k-t_k))}\subset\overline{\Omega}$ and
\be\label{eq-2.22}
u(s_k,y)\ge 1-2\delta \text{ for all $y\in\overline{B(y_k,(c_f-\varepsilon)(s_k-t_k))}$}.
\ee
Since $z_k\in\Gamma_{s_k}$, there exist $r_{M_{\delta}}>0$ and $y'_k\in\Omega_{s_k}^-$ such that $d_{\Omega}(z_k,y'_k)\le r_{M_{\delta}}$ and $d_{\Omega}(y'_k,\Gamma_{s_k})\ge M_{\delta}$, whence
\be\label{eq-2.23}
u(s_k,y'_k)\le \delta.
\ee
Since $d_{\Omega}(x_k,y_k)\le r_{M_{\delta}+R_{\varepsilon}}$ and $d_{\Omega}(x_k,z_k)< (c_f-2\varepsilon)(s_k-t_k)$, one has
$$d_{\Omega}(y_k,y'_k)\le r_{M_{\delta}+R_{\varepsilon}}+(c_f-2\varepsilon)(s_k-t_k)+r_{M_{\delta}}\le (c_f-\varepsilon)(s_k-t_k)\ \text{ for large $k$}.$$ By~\eqref{eq-2.22}, one infers that $u(s_k,y'_k)\ge 1-2\delta$ for large $k$, contradicting~\eqref{eq-2.23}. As a consequence, the claim~\eqref{eq-2.16} has been proved.
\vskip 0.2cm

Secondly, we claim that
\be\label{eq-2.24}
\liminf_{T^2_{\varepsilon}\le t<s,\ |t-s|\rightarrow +\infty} \frac{d_{\Omega}(\Gamma_{t},\Gamma_{s})}{|t-s|}\ge c_f-2\varepsilon.
\ee
Assume by contradiction that there exist two sequences $(t_k)_{k\in\mathbb{N}}$ and $(s_k)_{k\in\mathbb{N}}$ such that $T^2_{\varepsilon}\le t_k<s_k$, $s_k-t_k\rightarrow +\infty$ as $k\rightarrow +\infty$ and $d_{\Omega}(\Gamma_{t_k},\Gamma_{s_k})<(c_f-2\varepsilon)(s_k-t_k)$ for all $k\in\mathbb{N}$. There exist then two sequences $(x_k)_{k\in\N}$ and $(z_k)_{k\in\N}$ such that $x_k\in\Gamma_{t_k}$, $z_k\in\Gamma_{s_k}$, and $d_{\Omega}(x_k,z_k)<(c_f-2\varepsilon)(s_k-t_k)$ for every $k\in\mathbb{N}$. By~\eqref{eq1.4}-\eqref{eq1.5}, for each $k\in\N$, there exist $y_k\in \Omega^+_{t_k}$ and $y'_k\in \Omega^-_{s_k}$ such that $d_{\Omega}(x_k,y_k)\le r_{M_{\delta}+R_{\varepsilon}}$, $d_{\Omega}(y_k,\Gamma_{t_k})\ge M_{\delta}+R_{\varepsilon}$, $d_{\Omega}(z_k,y'_k)\le r_{M_{\delta}}$, and $d_{\Omega}(y'_k,\Gamma_{s_k})\ge M_{\delta}$. By~\eqref{eq-2.15}, one has that $|x_k|$, $|z_k|\ge D_{\varepsilon}$ and $|y_k|\ge D_{\varepsilon}-r_{M_{\delta}+R_{\varepsilon}}\ge\max(R_\epsilon+L_\epsilon,X_\epsilon)$ and $|y'_k|\ge D_{\varepsilon}-r_{M_{\delta}}$. Then, $B(y_{k},R_{\varepsilon})\subset \R^N\setminus B(0,L_{\varepsilon})$ and $d_{\Omega}(B(y_k,R_{\varepsilon}),\Gamma_{t_k})\ge M_{\delta}$. It follows from Definition \ref{TF} that
\be\label{eq-2.26}
u(t_k,y)\ge 1-\delta \text{ for all $y\in\overline{B(y_k,R_{\varepsilon})}$},
\ee
while
\be\label{eq-2.27}
u(s_k,y'_k)\le \delta.
\ee
Since $\Omega\cap B(0,D_{\varepsilon})\subset \Omega^+_{t}$ for $t\ge T^2_{\varepsilon}$ and $D_{\varepsilon}\ge M_{\delta}+L+R_3-R_2$, one has that $u(t,x)\ge 1-\delta$ for all $t\ge T^2_{\varepsilon}$ and $x\in\overline{B(0,L+R_3-R_2)}\cap\overline{\Omega}$. It then follows from the comparison principle and Lemma~\ref{lemma2.2} that there exist $T_{\varepsilon,k}=(|y_k|-R_{\varepsilon}-L_{\varepsilon})/(c_f-\varepsilon)$ and a real number $T'_{\varepsilon}\ge0$ (independent of $k$) such that~\eqref{eq-2.19} holds and
\be\label{uty}
u(t,y)\ge1-4\delta\ \text{ for all $t-t_k\ge T_{\varepsilon,k}+T'_{\varepsilon}$ and $y\in\overline{B(y_k,(c_f-\varepsilon)(t-t_k-T'_{\varepsilon}))}\cap\overline{\Omega}$}.
\ee
We now consider three cases according to the position of the real number $s_k-t_k$ with respect to~$T_{\varepsilon,k}$ and $T_{\varepsilon,k}+T'_\epsilon$. Firstly, if $s_k-t_k\le T_{\varepsilon,k}$ for a subsequence of large $k$, then $u(s_k,y'_k)\ge 1-2\delta$ by~\eqref{eq-2.19} since
$$d_{\Omega}(y_k,y'_k)\le d_{\Omega}(y_k,x_k)+d_{\Omega}(x_k,z_k)+d_{\Omega}(z_k,y'_k)\le r_{M_{\delta}+R_{\varepsilon}}+(c_f-2\varepsilon)(s_k-t_k)+r_{M_{\delta}}< (c_f-\varepsilon)(s_k-t_k)$$
for a subsequence of large $k$, a contradiction with~\eqref{eq-2.27}. Secondly, if $s_k-t_k\ge T_{\varepsilon,k}+T'_{\varepsilon}$ for a subsequence of large $k$, then $u(s_k,y'_k)\ge 1-4\delta$ by~\eqref{uty} since
$$d_{\Omega}(y_k,y'_k)\le r_{M_{\delta}+R_{\varepsilon}}+(c_f-2\varepsilon)(s_k-t_k)+r_{M_{\delta}}< (c_f-\varepsilon)(s_k-t_k-T'_{\varepsilon})$$
for a subsequence of large $k$, a contradiction with~\eqref{eq-2.27}. Lastly, if $T_{\varepsilon,k}<s_k-t_k< T_{\varepsilon,k}+T'_{\varepsilon}$ for a subsequence of large $k$, then $u(s_k,y'_k)\ge 1-2\delta$ by~\eqref{eq-2.19} since
$$d_{\Omega}(y_k,y'_k)\!\le\!r_{M_{\delta}\!+\!R_{\varepsilon}}\!+\!(c_f\!-\!2\varepsilon)(s_k\!-\!t_k)\!+\!r_{M_{\delta}}\!\le\!(c_f\!-\!\varepsilon)(T_{\varepsilon,k}\!+\!T'_{\varepsilon})\!-\!\varepsilon(s_k\!-\!t_k)\!+\!r_{M_{\delta}\!+\!R_{\varepsilon}}\!+\!r_{M_{\delta}}\!<\!(c_f\!-\!\varepsilon)T_{\varepsilon,k}$$
for a subsequence of large $k$, contradicting~\eqref{eq-2.27}. As a consequence, the claim~\eqref{eq-2.24} is proved.
\vskip 0.2cm

Thirdly, we prove that
\be\label{liminf3eps}
\liminf_{|t-s|\rightarrow +\infty} \frac{d_\Omega(\Gamma_t,\Gamma_s)}{|t-s|}\ge c_f-3\varepsilon.
\ee
Assume by contradiction that there exist two sequences $(t_k)_{k\in\mathbb{N}}$ and $(s_k)_{k\in\mathbb{N}}$ of real numbers such that $t_k<s_k$, $s_k-t_k\rightarrow +\infty$ as $k\rightarrow +\infty$ and
\be\label{eq-2.28}
d_{\Omega}(\Gamma_{t_k},\Gamma_{s_k})<(c_f-3\varepsilon)(s_k-t_k) \text{ for all $k\in\mathbb{N}$}.
\ee

Then, six cases may occur up to extraction of a subsequence, namely, case~1: $t_k<s_k\le T_{\varepsilon}^1$ for all $k\in\N$; case~2: $t_k<T_{\varepsilon}^1<s_k<T_{\varepsilon}^2$ for all $k\in\N$; case~3: $t_k<T_{\varepsilon}^1<T_{\varepsilon}^2\le s_k$ for all $k\in\N$; case~4: $T^1_\epsilon\le t_k<s_k\le T^2_\epsilon$ for all $k\in\N$; case~5: $T_{\varepsilon}^1\le t_k\le T_{\varepsilon}^2<s_k$ for all $k\in\N$; case~6: $T_{\varepsilon}^2<t_k<s_k$ for all $k\in\N$. In fact, we have already shown that cases~1 and~6 are impossible for large $k$ by~\eqref{eq-2.16} and~\eqref{eq-2.24}. Furthermore, case~4 is also immediately ruled out for large $k$, since $s_k-t_k\to+\infty$ as $k\to+\infty$.

Now, consider case 2. In this case, there holds $t_k\to-\infty$ as $k\to+\infty$ and it follows from~\eqref{eq-2.16} that $d_{\Omega}(\Gamma_{t_k},\Gamma_{T_{\varepsilon}^1})\ge (c_f-5\varepsilon/2) (T_{\varepsilon}^1-t_k)$ for large~$k$. Since $|s_k-T^1_\epsilon|\le T^2_\epsilon-T^1_\epsilon$ for all $k$, pro\-perty~\eqref{eq+2.19} yields the existence of a constant $M>0$ such that $d_{\Omega}(\Gamma_{t_k},\Gamma_{s_k})\ge d_{\Omega}(\Gamma_{t_k},\Gamma_{T_{\varepsilon}^1})-M$ for all $k\in\N$, hence $d_{\Omega}(\Gamma_{t_k},\Gamma_{s_k})\ge(c_f-3\varepsilon) (s_k-t_k)$ for large~$k$, a contradiction with~\eqref{eq-2.28}. Similarly, one can reach a contradiction with~\eqref{eq-2.28} in case~5.

Only case~3 remains to be considered. Notice that we can then assume without loss of generality that $t_k\rightarrow -\infty$ and~$s_k\rightarrow+\infty$ as $k\rightarrow+\infty$ (otherwise, by decreasing $T_{\varepsilon}^1$ and increasing $T_{\varepsilon}^2$, one can reduce case~3 to case~2 or case~5). As above, there exist some sequences $(x_k)_{k\in\N}$, $(z_k)_{k\in\N}$, $(y_k)_{k\in\N}$ and $(y'_k)_{k\in\N}$ such that $x_k\in\Gamma_{t_k}$, $z_k\in\Gamma_{s_k}$, $y_k\in\Omega^+_{t_k}$, $y'_k\in\Omega^-_{s_k}$, $d_{\Omega}(x_k,y_k)\le r_{M_{\delta}+R_{\varepsilon}}$, $d_{\Omega}(z_k,y'_k)\le r_{M_{\delta}}$ and~\eqref{eq-2.26}-\eqref{eq-2.27} hold. Then, by Lemma~\ref{lemma2.2},
$$u(t,y)\ge 1-2\delta\ \text{ for all $0\le t-t_k\le T_{\varepsilon,k}:=\frac{|y_k|-R_{\varepsilon}-L_{\varepsilon}}{c_f-\varepsilon}$ and $y\in\overline{B(y_k,(c_f-\varepsilon)(t-t_k))}\ (\subset\overline{\Omega})$}.$$
As above, one then infers that $T^1_{\varepsilon}-t_k\le T_{\varepsilon,k}$, which implies that $T^2_{\varepsilon}-(t_k+T_{\varepsilon,k})\le T^2_{\varepsilon}-T^1_{\varepsilon}$. Since $\Omega\cap B(0,D_{\varepsilon})\subset\Omega^+_t$ for $t\ge T^2_{\varepsilon}$ and $D_{\varepsilon}\ge M_{\delta}+L+R_3-R_2$, one gets that $u(t,y)\ge 1-\delta$ for all $t\ge T^2_{\varepsilon}$ and $y\in\overline{B(0,L+R_3-R_2)}\cap\overline{\Omega}$. Therefore, by Lemma \ref{lemma2.2}, there exists a real number $T'_{\varepsilon}\ge T^2_{\varepsilon}-T^1_{\varepsilon}$ such that
$$u(t,y)\ge 1-4\delta\ \text{ for all $t-t_k\ge T_{\varepsilon}+T'_{\varepsilon}$ and $y\in\overline{B(y_k,(c_f-\varepsilon)(t-t_k-T'_{\varepsilon}))}\cap\overline{\Omega}$}.$$
Then, by the remainder of the proof of~\eqref{eq-2.24}, one concludes that~\eqref{eq-2.28} is impossible in case~3.

As a conclusion,~\eqref{liminf3eps} holds for every $\epsilon\in(0,c_f/3)$ and the proof of~\eqref{eq-2.13} is thereby complete.
\vskip 0.3cm

{\it Step 3: the upper estimate~\eqref{eq-2.14}.} In order to show~\eqref{eq-2.14}, we shall prove that
\be\label{limsup3}
\limsup_{|t-s|\rightarrow +\infty} \frac{d_{\Omega}(\Gamma_t,\Gamma_s)}{|t-s|}\le c_f+2\varepsilon.
\ee
Assume by contradiction that there exist two sequences $(t_k)_{k\in\mathbb{N}}$ and $(s_k)_{k\in\mathbb{N}}$ of real numbers such that $t_k<s_k$, $s_k-t_k\rightarrow +\infty$ as $k\rightarrow +\infty$ and
\be\label{eq-2.36}
d_{\Omega}(\Gamma_{t_k},\Gamma_{s_k})>(c_f+2\varepsilon)(s_k-t_k) \text{ for all $k\in\mathbb{N}$}.
\ee
Notice that each interface $\Gamma_t$ is automatically unbounded (since otherwise by~\eqref{eq1.3} either~$\Omega^-_t$ or~$\Omega^+_t$ would be bounded). In particular, for each $k\in\N$, there is a point $z_k\in\Gamma_{s_k}$ such that~$|z_k|\ge(c_f+2\epsilon)(s_k-t_k)+r_{M_\delta}+L$ and, by~\eqref{eq1.4}-\eqref{eq1.5}, there are some points $y^{\pm}_k\in\Omega^\pm_{s_k}$ such that $d_\Omega(y^\pm_k,\Gamma_{s_k})\ge M_\delta$ and $d_\Omega(z_k,y^\pm_k)\le r_{M_\delta}$. Hence, $u(s_k,y^+_k)\ge1-\delta$, $u(s_k,y^-_k)\le\delta$ and
$$|y^\pm_k|\ge|z_k|-r_{M_\delta}\ge(c_f+2\epsilon)(s_k-t_k)+L.$$
On the other hand, since $d(z_k,\Gamma_{t_k})\ge d_{\Omega}(\Gamma_{s_k},\Gamma_{t_k})>(c_f+2\varepsilon)(s_k-t_k)$ and since
$$|z_k|\ge(c_f+2\epsilon)(s_k-t_k)+r_{M_\delta}+L>(c_f+2\epsilon)(s_k-t_k)+L,$$
it follows that, up to extraction of a subsequence, either $B(z_k,(c_f+2\epsilon)(s_k-t_k))\subset\Omega^+_{t_k}$ for all~$k\in\N$, or $B(z_k,(c_f+2\epsilon)(s_k-t_k))\subset\Omega^-_{t_k}$ for all $k\in\N$.

Consider first the case $B(z_k,(c_f+2\epsilon)(s_k-t_k))\subset\Omega^+_{t_k}$ for all $k\in\N$. Since $d_\Omega(z_k,y^-_k)\le r_{M_\delta}$, one infers that $B_k:=B(y^-_k,(c_f+2\epsilon)(s_k-t_k)-r_{M_\delta}-M_\delta)\subset\Omega^+_{t_k}$ (for all $k$ large enough so that~$(c_f+2\epsilon)(s_k-t_k)-r_{M_\delta}-M_\delta>0$) and $d_\Omega(B_k,\Gamma_{t_k})\ge M_\delta$, hence $u(t_k,x)\ge1-\delta$ for all $x\in B_k$. Therefore, $u(t_k,\cdot)\ge v_{y^-_k,R_\epsilon}(0,\cdot)$ in $\overline{\Omega}$ for all $k$ large enough so that~$(c_f+2\epsilon)(s_k-t_k)-r_{M_\delta}-M_\delta\ge R_\epsilon$. Since $|y^-_k|\ge(c_f+2\epsilon)(s_k-t_k)+L\ge R_\epsilon+L_\epsilon$ for $k$ large and $s_k-t_k\le(|y^-_k|-R_\epsilon-L_\epsilon)/(c_f-\epsilon)$ for $k$ large, Lemma~\ref{lemma2.2} applied with $x_0=y^-_k$ together with the maximum principle implies that $u(s_k,y^-_k)\ge v_{y^-_k,R_\epsilon}(s_k-t_k,y^-_k)\ge1-2\delta$ for all $k$ large enough, a contradiction with $u(s_k,y^-_k)\le\delta$. Therefore, the first case is ruled out.

Consider now the second case $B(z_k,(c_f+2\epsilon)(s_k-t_k))\subset\Omega^-_{t_k}$ for all $k\in\N$. Denote
$$\rho_k:=(c_f+2\epsilon)(s_k-t_k)-r_{M_\delta}-M_\delta-L_\epsilon+L.$$
Since $d_\Omega(z_k,y^+_k)\le r_{M_\delta}$ and $L_\epsilon>L$, one infers that $B(y^+_k,\rho_k)\subset\Omega^-_{t_k}$ (for all $k$ large enough so that $\rho_k>0$) and $d_\Omega(B(y^+_k,\rho_k),\Gamma_{t_k})\ge M_\delta$, hence $u(t_k,x)\le\delta$ for all~$x\in B(y^+_k,\rho_k)$. Therefore, $u(t_k,\cdot)\le w_{y^+_k,\rho_k}(0,\cdot)$ in $\overline{\Omega}$ for $k$ large. Since $\rho_k>R_\epsilon$ for $k$ large, since~$|y^+_k|\ge(c_f+2\epsilon)(s_k-t_k)+L\ge\rho_k+L_\epsilon$ and since $s_k-t_k\le(\rho_k-R_\epsilon)/(c_f+\epsilon)$ for~$k$ large, Lemma~\ref{lemma2.4} applied with $x_0=y^+_k$ and $R=\rho_k$ together with the maximum principle implies that~$u(s_k,y^+_k)\le w_{y^+_k,\rho_k}(s_k-t_k,y^+_k)\le2\delta$ for all $k$ large enough, a contradiction with~$u(s_k,y^+_k)\ge1-\delta$. Therefore, the second case is ruled out too.

As a conclusion,~\eqref{limsup3} holds for every $\epsilon>0$ small enough, implying the desired upper estimate~\eqref{eq-2.14}. Together with~\eqref{eq-2.13}, the proof of Theorem~\ref{th1} is thereby complete.
\end{proof}


\subsection{Proof of Corollaries~\ref{cor1} and~\ref{cor3}}\label{sec2.3}

This subsection is devoted to the proof of Corollaries~\ref{cor1} and~\ref{cor3} on the global mean speed of almost-planar transition fronts and transition fronts in ``large" domains.
\vskip 0.3cm

\begin{proof}[Proof of Corollary~\ref{cor1}] It easily follows from Theorem~\ref{th1} and Corollary~\ref{corollary2.2}. We recall that a transition front $u$ is called almost-planar if for every $t\in\R$, the set $\Gamma_t$ can be chosen as
\be\label{defHt}
\Gamma_t=H_t\cap\Omega\ \hbox{ with }\ H_t=\{x\in\R^N: x\cdot e_t=\xi_t\},
\ee
for some vector $e_t\in\mathbb{S}^{N-1}$ and some real number $\xi_t$. Let $\delta\in(0,1/4)$ be defined as in~\eqref{d} and let~$R_1$, $R_2$ and $R_3$ be as in Lemma~\ref{lemma2.1}. Now, since $\R^N\setminus\Omega=K$ is compact and smooth,~\eqref{eq1.3} yields the existence of $x_0\in \Omega_0^+$ such that $B(x_0,R_3)\subset\Omega$ and $d_{\Omega}(B(x_0,R_1),\Gamma_0)\ge M_{\delta}$, which gives in particular $u(0,\cdot)\ge 1-\delta$ in $B(x_0,R_1)$. The comparison principle yields $u(t,x)\ge v_{x_0,R_1}(t,x)$ for all $t>0$ and $x\in\overline{\Omega}$. By Corollary~\ref{corollary2.2}, there is $R>0$ such that
\be\label{L+R}
\liminf_{t\rightarrow +\infty}u(t,x)\!\ge\!\liminf_{t\rightarrow +\infty}v_{x_0,R_1}(t,x)\!\ge\!1\!-\!2\delta\ \hbox{ locally uniformly in }x\!\in\!\Omega\setminus B(0,L+R).
\ee

This implies that hyperplanes $H_t$ defined in~\eqref{defHt} are far away from the obstacle $K$ for large time~$t$, in the sense that $\lim_{t\to+\infty}|\xi_t|=+\infty$. Indeed, if not, there exist a sequence~$(t_k)_{k\in\N}\to+\infty$ and a bounded sequence $(x_k)_{k\in\N}$ of points in $\R^N$ such that $x_k\in\Gamma_{t_k}=H_{t_k}\cap\Omega$ for all~$k\in\N$. By~\eqref{eq1.5} and since each $H_{t_k}$ is a hyperplane, there is a bounded sequence $(x'_k)_{k\in\N}$ of points in $\R^N$ such that, for every $k\in\N$, $x'_k$ belongs to the closed set $\Omega\setminus B(0,L+R)$ and~$B(x'_k,M_\delta+R+2L)\subset\Omega^-_{t_k}$, hence $u(t_k,x'_k)\le\delta$. This contradicts~\eqref{L+R}, since $0<\delta<1/4<1/3$. As a consequence, $\lim_{t\to+\infty}|\xi_t|=+\infty$.

The same arguments then imply that, for any $\rho>0$, there holds $\Omega\cap B(0,L+\rho)\subset \Omega_t^+$ for large time $t$. Therefore, due to Definition~\ref{TF}, the propagation of the front is complete in the sense of~\eqref{complete}. Corollary~\ref{cor1} then immediately follows from Theorem~\ref{th1}.
\end{proof}
\vskip 0.3cm

Before doing the proof of Corollary~\ref{cor3}, we prove an important auxiliary lemma stating that the propagation of any transition front is at least partial (see also Remark~\ref{rem13}).

\begin{lemma}\label{lemmap}
Let $\Omega=\R^N\setminus K$, where $K$ is a compact set with smooth boundary. For any transition front $u$ of~\eqref{eq1.1} connecting $0$ and $1$, there is a solution $p:\overline{\Omega}\to(0,1]$ of~\eqref{eqp} satisfying~\eqref{liminfp}.
\end{lemma}

\begin{proof}
It is based on the existence of a compactly supported stationary sub-solution of the elliptic version of~\eqref{eq1.1} in $\R^N$ and on the comparison of $u$ with some shifts of that sub-solution.

First of all, remember that $\theta_2\in(0,1)$ is defined in~\eqref{deftheta12}. We first claim that there are $R>0$ and a function $\psi\in C^2\big(\overline{B(0,R)}\big)$ solving
\be\label{eqphi}\left\{\baa{rcl}
\Delta\psi+f(\psi) & = & 0\ \hbox{ in }\overline{B(0,R)},\vspace{3pt}\\
0\ \le\ \psi & < & 1 \ \hbox{ in }\overline{B(0,R)},\vspace{3pt}\\
\psi & = & 0\ \hbox{ on }\partial B(0,R),\vspace{3pt}\\
\displaystyle\mathop{\max}_{\overline{B(0,R)}}\psi\ =\ \psi(0) & > & \theta_2.\eaa\right.
\ee
The proof is standard, based on~\cite{BL}, so we just sketch it. For each $r>0$, there is a minimizer $\psi_r\in H^1_0(B(0,r))$ of the energy functional $J_r$ defined in $H^1_0(B(0,r))$ by
$$J_r(\varphi)=\frac{1}{2}\int_{B(0,r)}|\nabla\varphi|^2-\int_{B(0,r)}F(\varphi),\ \ J_r(\psi_r)=\min_{\varphi\in H^1_0(B(0,r))}J_r(\varphi),$$
where $F(s)=\int_0^s\overline{f}(t)dt$ and $\overline{f}:\R\to\R$ is the function $f$ extended by $0$ outside the interval $[0,1]$. Furthermore, with this extension of $f$, there is a minimizer $\psi_r$ ranging in the interval $[0,1]$ and, from elliptic regularity, such a function $\psi_r$ is a classical $C^2(\overline{B(0,r)},[0,1])$ solution of~$\Delta\psi_r+f(\psi_r)=0$ in $\overline{B(0,r)}$ with $\psi_r=0$ on $\partial B(0,r)$. The strong elliptic maximum principle also yields $\psi_r<1$ in $\overline{B(0,r)}$ and, either $\psi_r\equiv0$ in $\overline{B(0,r)}$, or $\psi_r>0$ in $B(0,r)$. In both cases,~$\psi_r$ is a radially symmetric and nonincreasing function of $|x|$ (from the standard results of~\cite{GNN} in the latter case). In particular, $\max_{\overline{B(0,r)}}\psi_r=\psi_r(0)\in[0,1)$ in both cases. Let us now show that
$$\psi_r(0)\to1\ \hbox{ as }r\to+\infty,$$
which will then provide $R>0$ and a solution $\psi$ of~\eqref{eqphi}. Assume by way of contradiction that there are $\eta\in(0,1)$, a sequence $(r_k)_{k\in\N}$ of positive real numbers converging to $+\infty$ and a sequence $(\psi_{r_k})_{k\in\N}$ of functions such that each $\psi_{r_k}\in C^2(\overline{B(0,r_k)},[0,1))$ minimizes $J_{r_k}$ in $H^1_0(B(0,r_k))$ and $\max_{\overline{B(0,r_k)}}\psi_{r_k}=\psi_{r_k}(0)\le\eta<1$. Let $\eta_0\in[0,\eta]\ (\subset[0,1))$ be such that $F(\eta_0)=\max_{[0,\eta]}F$ and remember from the paragraph after~\eqref{F2} that $\varsigma:=\int_{\eta_0}^1f(s)ds=F(1)-F(\eta_0)>0$. Hence, $F(s)\le F(\eta_0)=F(1)-\varsigma$ for all $s\in[0,\eta]$ and $J_{r_k}(\psi_{r_k})\ge-(F(1)-\varsigma)\,\alpha_Nr_k^N$ for all $k\in\N$, where $\alpha_N>0$ denotes the Lebesgue measure of the $N$-dimensional unit ball $B(0,1)$. On the other hand, after assuming without loss of generality that $r_k>1$ for every $k\in\N$, consider the function~$\varphi_k$ defined by $\varphi_k(x)=1$ for $x\in B(0,r_k-1)$ and $\varphi_k(x)=r_k-|x|$ for $x\in\overline{B(0,r_k)}\setminus B(0,r_k-1)$. For each $k\in\N$, one has $\varphi_k\in H^1_0(B(0,r_k))$, hence
$$\baa{rcl}
J_{r_k}(\psi_{r_k}) & \le & \displaystyle J_{r_k}(\varphi_k)=\frac{\alpha_N}{2}\,(r_k^N-(r_k-1)^N)-F(1)\alpha_N(r_k-1)^N-\int_{B(0,r_k)\setminus B(0,r_k-1)}F(\varphi_k)\vspace{3pt}\\
& \le & \displaystyle\alpha_N\,\Big(\frac{1}{2}+\max_{[0,1]}|F|\Big)\,(r_k^N-(r_k-1)^N)-F(1)\alpha_N(r_k-1)^N.\eaa$$
This contradicts the inequality $J_{r_k}(\psi_{r_k})\ge-(F(1)-\varsigma)\,\alpha_Nr_k^N$ for $k$ large enough, since $\lim_{k\to+\infty}r_k=+\infty$ and $\varsigma>0$. As a consequence, $\max_{\overline{B(0,r)}}\psi_r=\psi_r(0)\to1$ as $r\to+\infty$, and there exists a solution $\psi$ of~\eqref{eqphi} for some $R>0$.

Let $R>0$ and $\psi$ be fixed as in the previous paragraph, and denote $\mu=1-\psi(0)>0$. From~\eqref{eq1.3} and the smoothness and compactness of $\R^N\setminus\Omega=K\ (\subset B(0,L))$, there is a point $x_0\in\Omega^+_0$ such that $|x_0|>L+R$ and $d_{\Omega}(x_0,\Gamma_0)\ge M_\mu+R$, hence $u(0,\cdot)\ge1-\mu$ in~$\overline{B(x_0,R)}\ (\subset\Omega^+_0\subset\Omega)$. Let $\underline{u}$ be the solution of the Cauchy problem~\eqref{vrx0} with initial condition~$\underline{u}(0,\cdot)$ defined by
$$\underline{u}(0,x)=\psi(x-x_0)\hbox{ if }x\in\overline{B(x_0,R)}\ (\subset\Omega)\ \hbox{ and }\ \underline{u}(0,x)=0\hbox{ if }x\in\overline{\Omega}\setminus\overline{B(x_0,R)}.$$
Since $0\le\psi\le1-\mu$ in $\overline{B(0,R)}$, one has $0\le\underline{u}(0,\cdot)\le u(0,\cdot)<1$ in $\overline{\Omega}$, hence $0\le\underline{u}(t,x)\le u(t,x)<1$ for all $t>0$ and $x\in\overline{\Omega}$ from the comparison principle. Moreover, since $\psi$ solves~\eqref{eqphi} with $\overline{B(x_0,R)}\subset\Omega$ and $f(0)=0$, the function $\underline{u}(0,\cdot)$ is a strict generalized sub-solution of the stationary problem corresponding to~\eqref{eq1.1}, hence $\underline{u}(t,x)$ is (strictly) increasing with respect to $t>0$ for every $x\in\overline{\Omega}$, together with $0<\underline{u}(t,x)<1$ for all $t>0$ and $x\in\overline{\Omega}$. It then follows from standard parabolic estimates that
$$\underline{u}(t,x)\to p(x)\hbox{ as }t\to+\infty\hbox{ locally uniformly in }x\in\overline{\Omega}$$
(hence, $\liminf_{t\to+\infty}u(t,x)\ge p(x)$ locally uniformly in $x\in\overline{\Omega}$), where $p:\overline{\Omega}\to(0,1]$ is a classical solution of $\Delta p+f(p)=0$ in $\overline{\Omega}$ and $p_\nu=0$ on $\partial\Omega$. Furthermore, $p>\underline{u}(0,\cdot)$ in $\overline{\Omega}$.

In order to complete the proof of Lemma~\ref{lemmap}, it only remains to show that $p(x)\to1$ as $|x|\to+\infty$. Since $p>\underline{u}(0,\cdot)$ in $\overline{\Omega}$ and $\overline{B(x_0,R)}$ is included in the open set $\Omega$, it follows from the definition of $\underline{u}(0,\cdot)$ that there is $\kappa_*>1$ such that $p>\psi(\cdot-\kappa x_0)$ in $\overline{B(\kappa x_0,R)}$ for all $\kappa\in[1,\kappa_*]$. Since $|x_0|>L+R$ and $\R^N\setminus\Omega=K\subset B(0,L)$, notice that $\overline{B(\kappa x_0,R)}\subset\Omega$ for all $\kappa\ge1$ and call
$$\kappa^*=\sup\big\{\kappa\ge1 : p>\psi(\cdot-\kappa'x_0)\hbox{ in }\overline{B(\kappa'x_0,R)}\hbox{ for all }\kappa'\in[1,\kappa]\big\}.$$
One has $\kappa^*\ge\kappa_*>1$. We claim that $\kappa^*=+\infty$. Assume not. Then, by continuity and by definition of $\kappa^*$, there holds $p\ge\psi(\cdot-\kappa^*x_0)$ in $\overline{B(\kappa^*x_0,R)}$ with equality somewhere in~$\overline{B(\kappa^*x_0,R)}$. Since $p>0$ in $\overline{\Omega}$ and $\psi(\cdot-\kappa^*x_0)=0$ on $\partial B(\kappa^*x_0,R)$, there is a interior point $x^*\in B(\kappa^*x_0,R)$ such that $p(x^*)=\psi(x^*-\kappa^*x_0)$, and then $p\equiv\psi(\cdot-\kappa^*x_0)$ in $\overline{B(\kappa^*x_0,R)}$ from the strong maximum principle. This is impossible on $\partial B(\kappa^*x_0,R)$. Therefore, $\kappa^*=+\infty$ and $p>\psi(\cdot-\kappa x_0)$ in $\overline{B(\kappa x_0,R)}$ for all $\kappa\ge1$. Similarly, one can show with the same type of sliding method that, for any $\kappa\ge1$ and for any rotation $\mathcal{R}$ of $\R^N$, there holds $p>\psi(\cdot-\mathcal{R}(\kappa x_0))$ in $\overline{B(\mathcal{R}(\kappa x_0),R)}\ (\subset\Omega)$. This implies in particular that
$$p(x)>\psi(0)>\theta_2\ \hbox{ for all }x\in\R^N\hbox{ with }|x|\ge|x_0|.$$
Remember also that $p\le1$ in $\overline{\Omega}$.

To conclude that $p(x)\to1$ as $|x|\to+\infty$, assume by way of contradiction that there is a sequence $(x_k)_{k\in\N}$ of points in $\overline{\Omega}$ such that $|x_k|\to+\infty$ as $k\to+\infty$ and $\limsup_{k\to+\infty}p(x_k)<1$.  From standard elliptic estimates, up to extraction of a subsequence, the functions $x\mapsto p(x+x_k)$ converge in $C^2_{loc}(\R^N)$ to a classical solution $p_{\infty}$ of $\Delta p_\infty+f(p_\infty)=0$ in $\R^N$ with $\theta_2<\psi(0)\le p_\infty(x)\le1$ for all $x\in\R^N$ and $p_\infty(0)<1$. The comparison principle implies that $p_\infty(x)\ge\varrho(t)$ for all $t\ge0$ and $x\in\R^N$, where $\varrho'(t)=f(\varrho(t))$ for all $t\ge0$ and $\varrho(0)=\psi(0)$. But since $f>0$ on $(\theta_2,1)$ and $f(1)=0$, one has $\lim_{t\to+\infty}\varrho(t)=1$, hence $p_{\infty}\equiv 1$ in $\R^N$, contradicting $p_\infty(0)<1$.

As a conclusion $p(x)\to1$ as $|x|\to+\infty$ and the proof of Lemma~\ref{lemmap} is thereby complete.
\end{proof}
\vskip 0.3cm

\begin{proof}[Proof of Corollary~\ref{cor3}] Let $\Omega=\R^N\setminus K$, where $K$ is a compact set with smooth boundary. Without loss of generality, one can assume that $K$ is not empty (the conclusion in the case $\Omega=\R^N$ follows from~\cite{H}). From Theorem~\ref{th1} and Lemma~\ref{lemmap}, the proof of Corollary~\ref{cor3} will be complete once one shows that there is $R_0>0$ such that, for any $R\ge R_0$ and for any $x_0\in\R^N$, any solution $p:\overline{R\Omega+x_0}\to(0,1]$ of~\eqref{eqp} (with $R\Omega+x_0$ instead of $\Omega$) satisfies $p\equiv1$ in $\overline{R\Omega+x_0}$.

To do so, assume by way of contradiction that there are a sequence $(R_k)_{k\in\N}$ of positive real numbers converging to $+\infty$, a sequence $(x_k)_{k\in\N}$ in $\R^N$ and a sequence $(p_k)_{k\in\N}$ of classical solutions $p_k:\overline{R_k\Omega+x_k}\to(0,1]$ of
$$\Delta p_k+f(p_k)=0\hbox{ in }\overline{R_k\Omega+x_k},\ \ (p_k)_{\nu}=0\hbox{ on }\partial(R_k\Omega+x_k)$$
with $\lim_{|x|\to+\infty}p_k(x)=1$ and $p_k\not\equiv1$ in $\overline{R_k\Omega+x_k}$ for each $k\in\N$. Therefore, for each $k\in\N$, there is a point $y_k\in\overline{R_k\Omega+x_k}$ such that $p_k(y_k)=\min_{\overline{R_k\Omega+x_k}}p_k\in(0,1)$. From the strong maximum principle and Hopf lemma, and since $f>0$ on $(\theta_2,1)$, one infers that $p_k(y_k)\le\theta_2$.

Now, since $K=\R^N\setminus\Omega$ is compact and smooth, there is $r>0$ such that, for any $y\in\overline{\Omega}$ and for any $r'\in(0,r]$, there is a continuous path $\gamma:[0,+\infty)\to\Omega$ such that
$$y\in\overline{B(\gamma(0),r')},\ \ B(\gamma(s),r')\subset\Omega\hbox{ for all }s\ge0,\ \hbox{ and }\ \lim_{s\to+\infty}|\gamma(s)|=+\infty.$$
Furthermore, if $B(y,r)\subset\Omega$, one can take $\gamma(0)=y$. Consider also $R>0$ and a solution $\psi$ of~\eqref{eqphi}, as given in the proof of Lemma~\ref{lemmap}. Observe that $r$ and $R$ are independent of $k$.

Take any $k\in\N$ large enough so that $r\,R_k\ge R$ and consider any point $y\in\overline{R_k\Omega+x_k}$ such that $B(y,R)\subset R_k\Omega+x_k$. There is then a continuous path $\gamma:[0,+\infty)\to\,R_k\Omega+x_k$ such that $\gamma(0)=y$, $B(\gamma(s),R)\subset R_k\Omega+x_k$ for all $s\ge0$, and $\lim_{s\to+\infty}|\gamma(s)|=+\infty$. Since $\lim_{|x|\to+\infty}p_k(x)=1$ and $\max_{\overline{B(0,R)}}\psi=\psi(0)<1$, there holds $p_k>\psi(\cdot-\gamma(s))$ in $\overline{B(\gamma(s),R)}$ for all $s$ large enough. Since $p_k>0$ in $\overline{\Omega}$ and $\psi=0$ on $\partial B(0,R)$, the same type of sliding method as in the proof of Lemma~\ref{lemmap} then implies that $p_k>\psi(\cdot-\gamma(s))$ in $\overline{B(\gamma(s),R)}$ for all $s\ge0$. As a consequence,
$$p_k(y)\ge\psi(y-\gamma(0))=\psi(0)>\theta_2\ \hbox{ for all }y\in\overline{R_k\Omega+x_k}\hbox{ such that }B(y,R)\subset R_k\Omega+x_k.$$
Since $p_k(y_k)\le\theta_2$, one concludes that there is $z_k\in\partial(R_k\Omega+x_k)$ such that $|z_k-y_k|<R$.

From standard elliptic estimates up to the boundary, it follows that, for any $\rho>0$ and any $\alpha\in(0,\beta]$, the restrictions of the functions $p_k$ on $\overline{B(z_k,\rho)}\cap\overline{R_k\Omega+x_k}$ are uniformly bounded in~$C^{2,\alpha}\big(\overline{B(z_k,\rho)}\cap\overline{R_k\Omega+x_k}\big)$. Denote $H=\big\{x\in\R^N:x_1>0\big\}$. Therefore, up to extraction of a subsequence and up to rotation of the frame, the functions $p_k(\cdot+z_k)$ converge in $C^2_{loc}(H)$ to a classical solution $p_{\infty}:\overline{H}\to[0,1]$ of
$$\Delta p_{\infty}+f(p_{\infty})=0\hbox{ in }\overline{H},\ \ (p_{\infty})_\nu=0\hbox{ on }\partial H,$$
together with $p_{\infty}(x)\ge\psi(0)>\theta_2$ for all $x\in\overline{H}$ with $x_1\ge R$ and $p_{\infty}(\zeta,0,\cdots,0)\le\theta_2$ for some~$\zeta\in[0,R]$. Notice in particular that $p_\infty>0$ in $\overline{H}$ from the strong maximum principle and Hopf lemma.

Finally, as in the last paragraph of the proof of Lemma~\ref{lemmap}, one gets that $p_{\infty}(x_1,\cdots,x_N)\to1$ as $x_1\to+\infty$ uniformly with respect to $(x_2,\cdots,x_N)\in\R^{N-1}$. In particular, denoting $X(s)=(s,0,\cdots,0)$ for $s\ge0$, there holds $p_{\infty}>\psi(\cdot-X(s))$ in $\overline{B(X(s),R)}$ for all $s$ large enough. Denote
$$s_*=\inf\big\{s\ge0: p_{\infty}>\psi(\cdot-X(s'))\hbox{ in }\overline{B(X(s'),R)}\cap\overline{H}\hbox{ for all }s'\ge s\big\}\in[0,+\infty).$$
Since $p_{\infty}(X(\zeta))\le\theta_2<\psi(0)$, one gets $s_*>\zeta$, hence $s_*>0$ and $B(X(s_*),R)\cap H$ is a non-empty open set. Furthermore, by continuity, there is $x_*\in\overline{B(X(s_*),R)}\cap\overline{H}$ such that $p_\infty(x_*)=\psi(x_*-X(s_*))$. Since $p_\infty>0$ in $\overline{H}$ and $\psi(\cdot-X(s_*))=0$ on $\partial B(X(s_*),R)$, one has $x_*\in B(X(s_*),R)$. If $x_*\in B(X(s_*),R)\cap H$, it then follows from the strong interior maximum principle that $p_\infty\equiv\psi(\cdot-X(s_*))$ in $\overline{B(X(s_*),R)}\cap\overline{H}$, which is impossible on the non-empty set $\partial B(X(s_*),R)\cap\overline{H}$. Therefore, one can assume without loss of generality that $p_\infty>\psi(\cdot-X(s_*))$ in $B(X(s_*),R)\cap H$ and $x_*\in B(X(s_*),R)\cap\partial H$. The Hopf lemma then yields $0=(p_\infty)_\nu(x_*)<\nu\cdot\nabla\psi(x_*-X(s_*))$, with $\nu=(-1,0,\cdots,0)$. On the other hand, $s_*>0$ and the function $\psi$ is radially symmetric and nonincreasing in $|x|$ in $B(0,R)$, hence $\nu\cdot\nabla\psi(x_*-X(s_*))=-\frac{\partial\psi}{\partial x_1}(x_*-X(s_*))\le0$, a contradiction.

As a conclusion, one has reached a contradiction in all cases. The existence of the sequences $(R_k)_{k\in\N}$, $(x_k)_{k\in\N}$ and $(p_k)_{k\in\N}$ is thus ruled out and the proof of Corollary~\ref{cor3} is complete.
\end{proof}


\section{Transition fronts emanating from planar fronts in domains with multiple branches: proof of Theorem~\ref{th10}}

This section is devoted to the proof of Theorem \ref{th10}. In this section, $\Omega$ is a smooth domain with~$m\,(\ge2)$ cylindrical branches in the sense of~\eqref{branches}, $I$ and $J$ are non-empty sets of $\{1,\cdots,m\}$ such that $I\cap J=\emptyset$ and $I\cup J=\{1,\cdots,m\}$, and $u:\R\times\overline{\Omega}\to(0,1)$ is a time-increasing solution of~\eqref{eq1.1} emanating from the planar fronts in the branches $\mathcal{H}_i$ with $i\in I$ in the sense of~\eqref{frontlike}, namely $u(t,x)-\phi(-x\cdot e_i-c_f t+\sigma_i)\to 0$ uniformly in $\overline{\mathcal{H}_i}\cap\overline{\Omega}$ for every $i\in I$ and $u(t,x)\to 0$ uniformly in $\overline{\Omega\setminus \mathop{\cup}_{i\in I}\mathcal{H}_i}$ as $t\to-\infty$. We also assume that $u$ satisfies~\eqref{complete}, that is, $u(t,x)\to1$ locally uniformly in $x\in\overline{\Omega}$ as $t\to+\infty$. 

We shall prove that $u$ is a transition front with sets $\Gamma_t$ and $\Omega_t^{\pm}$ defined by~\eqref{eq+1.11} and~\eqref{eq+1.12}. These properties then immediately imply that $u$ has a global mean speed equal to $c_f$. We begin in Section~\ref{sec31} with some estimates on the large time behavior of $u$ in the branches $\mathcal{H}_j$ with~$j\in J$. We complete the proof of Theorem~\ref{th10} in Section~\ref{sec32}. Lastly, we prove in Section~\ref{sec33} additional quantitative estimates of the time-derivative of~$u$, which will be useful for the proof of Theorem~\ref{th6} in Section~\ref{sec4.2}.


\subsection{Large time estimates in the branches $\mathcal{H}_j$ with $j\in J$}\label{sec31}

We recall that $L>0$ is given in~\eqref{branches}.

\begin{lemma}\label{lemma4.1} There exist $t_1\in\R$, $t_2\in\R$, $\tau_1\in\R$, $\delta>0$ and $\mu>0$ such that, for every $j\in J$,
\be\label{eq4.1}
u(t,x)\le \phi(x\cdot e_j-c_f (t-t_1)+\tau_1)+\delta e^{-\delta (t-t_1)}+\delta e^{- \mu (x\cdot e_j-L)}
\ee
for $t\ge t_1$ and $x\in\overline{\mathcal{H}_j}$ such that $x\cdot e_j\ge L$ and
\be\label{eq4.2}
u(t,x)\ge \phi(x\cdot e_j-c_f (t-t_2)-L)-\delta e^{-\delta (t-t_2)}-\delta e^{- \mu (x\cdot e_j-L)}
\ee
for $t\ge t_2$ and $x\in\overline{\mathcal{H}_j}$ such that $x\cdot e_j\ge L$.
\end{lemma}

\begin{proof}
It is inspired by the construction of sub- and super-solutions in~\cite{FM}, to which some exponentially decaying terms in the direction $e_j$ are added.

{\it Step 1: choice of some parameters.} Take $\mu>0$ and then $\delta>0$ such that
\be\label{eq+3.4}
0<\mu<\sqrt{\min\Big(\frac{|f'(0)|}{2},\frac{|f'(1)|}{2}\Big)}
\ee
and
\be\label{eq+4.2}
0\!<\!\delta\!<\!\min\Big(\mu c_f,\frac{\theta_1}{4},\frac{1\!-\!\theta_2}4,\frac{|f'(0)|}{2},\frac{|f'(1)|}{2}\Big),\ f'\!\le\!\frac{f'(0)}{2} \text{ in $[0,3\delta]$, } f'\!\le\!\frac{f'(1)}{2} \text{ in $[1\!-\!3\delta,1]$},
\ee
where $0<\theta_1\le\theta_2<1$ are given in~\eqref{deftheta12}. Let $C>0$ and $k>0$ be as in~\eqref{eq+2.2}-\eqref{eq+2.4-}, and let $\omega>0$ and $R\ge\omega+2C\,(>2C)$ be large enough so that
\be\label{eq+3.6}
k\,\omega\ge2\delta+\max_{[0,1]} |f'|\ \hbox{ and }\ \Big(\!\max_{[0,1]} |f'|+\mu^2\Big) e^{-\mu (R-2C)}\le\Big(\!\max_{[0,1]} |f'|+\mu^2\Big) e^{-\mu (R-\omega-2C)}\le \delta.
\ee

{\it Step 2: proof of~\eqref{eq4.1}.} Since $u(t,x)\rightarrow 0$ as $t\rightarrow -\infty$ uniformly in $\overline{\Omega\setminus \cup_{i\in I}\mathcal{H}_i}$ from \eqref{frontlike}, there exists $t_1\in \R$ such that, for every $j\in J$,
\be\label{eq+4.3}
u(t_1,x)\le \delta\ \  \text{ for all $x\in\overline{\mathcal{H}_j}$ such that $x\cdot e_j\ge L$}
\ee
(notice that by~\eqref{branches} such points $x$ belong to $\overline{\Omega}$). Fix now an index $j\in J$. For $t\ge t_1$ and~$x\in\overline{\mathcal{H}_j}$ with $x\cdot e_j\ge L$, let us set
$$\overline{v}(t,x)=\min\left(\phi(\overline{\xi}(t,x))+\delta e^{-\delta (t-t_1)}+\delta e^{-\mu (x\cdot e_j-L)},1\right),$$
where
$$\overline{\xi}(t,x)=x\cdot e_j-c_f (t-t_1)+\omega e^{-\delta (t-t_1)}-\omega-L-R+C.$$
Let us check that $\overline{v}(t,x)$ is a super-solution of the problem satisfied by $u(t,x)$ for $t\ge t_1$ and~$x\in\overline{\mathcal{H}_j}$ with $x\cdot e_j\ge L$.

We first verify the initial and boundary conditions. At time $t_1$, it follows from~\eqref{eq+4.3} that $\overline{v}(t_1,x)\ge \delta\ge u(t_1,x)$ for all $x\in\overline{\mathcal{H}_j}$ with $x\cdot e_j\ge L$. On the other hand, for $t\ge t_1$ and $x\in\overline{\mathcal{H}_j}$ with $x\cdot e_j=L$, one has $\overline{\xi}(t,x)\le -R+C\le -C$ since $R\ge 2C$, whence $\overline{v}(t,x)\ge \min(1-\delta+\delta e^{-\delta (t-t_1)}+\delta,1)=1>u(t,x)$. Lastly, it is immediate to see that~$\overline{v}_\nu(t,x)=0$ for $t\ge t_1$ and $x\in\partial\mathcal{H}_j$ with $x\cdot e_j>L$ and $\overline{v}(t,x)<1$, since $\nu(x)\cdot e_j=0$.

Let us now check that $\mathcal{L}\overline{v}(t,x)=\overline{v}_t(t,x)-\Delta \overline{v}(t,x)-f(\overline{v}(t,x))\ge 0$ for every $t\ge t_1$ and~$x\in\overline{\mathcal{H}_j}$ such that $x\cdot e_j\ge L$ and $\overline{v}(t,x)<1$. After a straightforward computation we get
\begin{align*}
\mathcal{L}\overline{v}(t,x)=f(\phi(\overline{\xi}(t,x)))-f(\overline{v}(t,x))-\omega\delta e^{-\delta (t-t_1)}\phi'(\overline{\xi}(t,x))-\delta^2 e^{-\delta (t-t_1)}-\delta \mu^2 e^{-\mu (x\cdot e_j-L)}.
\end{align*}
If $\overline{\xi}(t,x)<-C$, one has $1>\phi(\overline{\xi}(t,x))\ge 1-\delta$ and then $1>\overline{v}(t,x)\ge 1-\delta$. From \eqref{eq+4.2} we obtain $f(\phi(\overline{\xi}(t,x)))-f(\overline{v}(t,x))\ge -(f'(1)/2)(\delta e^{-\delta (t-t_1)}+\delta e^{-\mu (x\cdot e_j-L)}).$ It then follows from \eqref{eq+3.4}-\eqref{eq+4.2} and the negativity of $\phi'$ and $f'(1)$ that
\begin{align*}
\mathcal{L}\overline{v}(t,x)\ge -\frac{f'(1)}{2}(\delta e^{-\delta (t-t_1)}+\delta e^{-\mu (x\cdot e_j-L)})-\delta^2 e^{-\delta (t-t_1)}-\delta \mu^2 e^{-\mu (x\cdot e_j-L)}\ge 0.
\end{align*}
A similar argument leads to $\mathcal{L}\overline{v}(t,x)\!\ge\!0$ if $\overline{\xi}(t,x)\!>\!C$, since $0\!<\!\phi(\overline{\xi}(t,x))\!\le\!\delta$ and $0\!<\!\overline{v}(t,x)\!\le\!3\delta$. If $-C\!\le\!\overline{\xi}(t,x)\!\le\!C$,  then $x\cdot e_j-L\ge c_f (t-t_1)+R-2C$ and $e^{-\mu (x\cdot e_j-L)}\le e^{-\mu(c_f (t-t_1)+R-2C)}$. Moreover, $-\phi'(\overline{\xi}(t,x))\ge k$ and $f(\phi(\overline{\xi}(t,x)))-f(\overline{v}(t,x))\ge -(\max_{[0,1]}|f'|)(\delta e^{-\delta (t-t_1)}+\delta e^{-\mu (x\cdot e_j-L)})$. Then one infers from~\eqref{eq+3.4}-\eqref{eq+3.6} that
$$\baa{rcl}
\mathcal{L}\overline{v}(t,x) & \ge & \displaystyle-\max_{[0,1]}|f'|(\delta e^{-\delta (t-t_1)}+\delta e^{-\mu (x\cdot e_j-L)})+k\omega\delta e^{-\delta (t-t_1)}-\delta^2 e^{-\delta (t-t_1)}-\delta \mu^2 e^{-\mu (x\cdot e_j-L)}\vspace{3pt}\\
& \ge & \displaystyle\delta e^{-\delta (t-t_1)}\Big(-\max_{[0,1]}|f'|-\delta+k\omega\Big)-\delta\Big(\max_{[0,1]}|f'|+\mu^2\Big)\,e^{-\mu(c_f (t-t_1)+R-2C)}\vspace{3pt}\\
& \ge & \displaystyle\delta e^{-\delta (t-t_1)}\Big(-\max_{[0,1]}|f'|-2\delta+k\omega\Big)\ge 0.\eaa$$

As a consequence, we arrive at $\mathcal{L}\overline{v}(t,x)=\overline{v}_t(t,x)-\Delta \overline{v}(t,x)-f(\overline{v}(t,x))\ge 0$ for all~$t\ge t_1$ and $x\in\overline{\mathcal{H}_j}$ with $x\cdot e_j\ge L$ and $\overline{v}(t,x)<1$. The comparison principle then yields
$$u(t,x)\le\overline{v}(t,x)\le\phi(x\cdot e_j-c_f(t-t_1)+\omega e^{-\delta(t-t_1)}-\omega-L-R+C)+\delta e^{-\delta (t-t_1)}+\delta e^{-\mu (x\cdot e_j-L)}$$
for all $t\ge t_1$ and $x\in\overline{\mathcal{H}_j}$ with $x\cdot e_j\ge L$, and finally~\eqref{eq4.1} holds with $\tau_1=-\omega-L-R+C$ since~$\phi$ is decreasing.

{\it Step 3: proof of~\eqref{eq4.2}.} Since $u(t,\cdot)\rightarrow 1$ as $t\rightarrow +\infty$ locally uniformly in $\overline \Omega$ by \eqref{complete}, there exists $t_2\in \R$ such that, for every $j\in J$,
$$u(t,x)\ge 1-\delta\ \  \text{for all $t\ge t_2$ and $x\in\overline{\mathcal{H}_j}$ with $L\le x\cdot e_j\le L+R$,}$$
where $R\ge \omega+2C$ is as in~\eqref{eq+3.6}. For $t\ge t_2$ and $x\in\overline{\mathcal{H}_j}$ with $x\cdot e_j\ge L$, let us set
$$\underline{v}(t,x)=\max\left(\phi(\underline{\xi}(t,x))-\delta e^{-\delta (t-t_2)}-\delta e^{-\mu (x\cdot e_j-L)},0\right),$$
where
$$\underline{\xi}(t,x)=x\cdot e_j-c_f (t-t_2)-\omega e^{-\delta (t-t_2)}+\omega-L-R+C.$$
Let us check $\underline{v}(t,x)$ is a sub-solution of the problem satisfied by $u(t,x)$ for $t\ge t_2$ and $x\in\overline{\mathcal{H}_j}$ with $x\cdot e_j\ge L$.

Let us first check the initial and boundary conditions. At time $t_2$, on the one hand, there holds $\underline{v}(t_2,x)\le 1-\delta\le u(t_2,x)$ for all $x\in\overline{\mathcal{H}_j}$ such that $L\le x\cdot e_j\le L+R$. On the other hand, for $x\in\overline{\mathcal{H}_j}$ such that $x\cdot e_j\ge L+R$, since $\underline{\xi}(t_2,x)\ge L+R-L-R+C=C$, one has $\underline{v}(t_2,x)\le \max\left(\delta-\delta-\delta e^{-\mu (x\cdot e_j-L)},0\right)=0\le u(t_2,x)$. Therefore, $\underline{v}(t_2,x)\le u(t_2,x)$ for all~$x\in\overline{\mathcal{H}_j}$ with $x\cdot e_j\ge L$. Now, for $t\ge t_2$ and~$x\in\overline{\mathcal{H}_j}$ with $x\cdot e_j=L$, one has $\underline{v}(t,x)\le1-\delta\le u(t,x)$. Moreover, it can be easily deduced from the definition of $\underline{v}$ that~$\underline{v}_\nu(t,x)=0$ for $t\ge t_2$ and~$x\in\partial\mathcal{H}_j$ with $x\cdot e_j>L$ and $\underline{v}(t,x)>0$.

Finally, with similar arguments as in Step~2, one can prove that $\underline{v}_t(t,x)\!-\!\Delta \underline{v}(t,x)\!-\!f(\underline{v}(t,x))\!\le\!0$ for all $t\ge t_2$ and $x\in\overline{\mathcal{H}_j}$ such that $x\cdot e_j\ge L$ and $\underline{v}(t,x)>0$. Property~\eqref{eq4.2} then follows from the maximum principle, the inequalities $R\ge\omega+2C\ge\omega+C$ and the negativity of $\phi'$. The proof of Lemma~\ref{lemma4.1} is thereby complete.
\end{proof}
\vskip 0.3cm

With the same token, the following lemma actually holds.

\begin{lemma}\label{lemma4.1+}
For any $\varepsilon>0$, there exists $t_{\varepsilon}\in\R$ such that, for every $j\in J$,
$$u(t,x)\ge \phi(x\cdot e_j-c_f (t-t_{\varepsilon})-L)-\varepsilon e^{-\delta (t-t_{\varepsilon})}-\varepsilon e^{- \mu (x\cdot e_j-L)}$$
for all $t\ge t_{\varepsilon}$ and $x\in\overline{\mathcal{H}_j}$ such that $x\cdot e_j\ge L$, with the same constants $\delta>0$ and $\mu>0$ as in Lemma~$\ref{lemma4.1}$.
\end{lemma}

\begin{proof}
Let $\mu>0$, $\delta>0$ and $\omega>0$ are defined as in \eqref{eq+3.4}-\eqref{eq+3.6}. For any $\varepsilon>0$, let
\be\label{defhateps}
\hat{\varepsilon}=\frac{\varepsilon}{\delta}.
\ee
By Lemma \ref{lemma4.1}, the conclusion of Lemma~\ref{lemma4.1+} holds for $\varepsilon\ge \delta$ with the same time $t_\epsilon=t_2$ as in Lemma~\ref{lemma4.1}. It remains to consider the case $0<\varepsilon< \delta$. In this case, there is $C_\varepsilon\ge C>0$ such that $\phi\ge 1-\varepsilon$ in $(-\infty,-C_\varepsilon]$ and $\phi\le \varepsilon$ in $[C_\varepsilon,+\infty)$. Let then $R_\epsilon$ be large enough so that~$R_\epsilon\ge\hat{\epsilon}\omega+2C_\epsilon$ and
$$\Big(\max_{[0,1]} |f'|+\mu^2\Big) e^{-\mu (R_\epsilon-\hat{\epsilon}\omega-2C_\epsilon)}\le \delta.$$
Since $u(t,x)\rightarrow 1$ as $t\rightarrow +\infty$ locally uniformly in $\overline{\Omega}$, one can follow the proof of~\eqref{eq4.2} to show that there exists $t_{\varepsilon}\in\R$ such that, for every $j\in J$, the function
$$\underline{v}(t,x)=\max\left(\phi(\underline{\xi}(t,x))-\hat{\varepsilon}\delta e^{-\delta (t-t_{\varepsilon})}-\hat{\varepsilon}\delta e^{-\mu (x\cdot e_j-L)},0\right)$$
with $\underline{\xi}(t,x)=x\cdot e_j-c_f (t-t_{\varepsilon})-\hat{\varepsilon}\omega e^{-\delta (t-t_{\varepsilon})}+\hat{\varepsilon}\omega-L-R_\epsilon+C_\epsilon$, is a sub-solution of the problem satisfied by $u(t,x)$ for $t\ge t_{\varepsilon}$ and $x\in\overline{\mathcal{H}_j}$ with $x\cdot e_j\ge L$. The conclusion then follows from the comparison principle as in Step~3 of the proof of Lemma~\ref{lemma4.1}.
\end{proof}

\begin{remark}{\rm An upper bound similar to~\eqref{eq4.1} could also be obtained for all large $t$ with $\epsilon>0$ instead of $\delta$ in the factors of the exponential terms, and with a shift $\tau_\epsilon=-\epsilon\omega/\delta-L-R_\epsilon+C_\epsilon$ instead of $\tau_1$, for some $C_\epsilon>0$ and $R_\epsilon>0$. But the quantities $\tau_\epsilon$ may not be bounded from below independently of $\epsilon$ and then may not be replaced by a quantity independent of $\epsilon$.}
\end{remark}

The next lemma is about the stability of the planar front in any branch $\mathcal{H}_j$.

\begin{lemma}\label{lemma4.2}
There is $M\ge0$ such that, if there are $j\in J$, $\varepsilon>0$, $t_0\in\R$ and $\tau\in\R$ such that
$$\sup_{x\in\overline{\mathcal{H}}_j,\,x\cdot e_j\ge L}|u(t_0,x)-\phi(x\cdot e_j-c_f t_0 +\tau)|\le \varepsilon$$
together with $\phi(L-c_ft_0+\tau)\ge1-\epsilon$ and $u(t,x)\ge1-\epsilon$ for all $t\ge t_0$ and $x\in\overline{\mathcal{H}_j}$ with $x\cdot e_j=L$, then it holds
$$\sup_{x\in\overline{\mathcal{H}_j},\,x\cdot e_j\ge L} |u(t,x)-\phi(x\cdot e_j-c_f t+\tau)|\le M\,\varepsilon\ \ \text{for all $t\ge t_0$}.$$
\end{lemma}

\begin{proof}
Let $\mu>0$, $\delta>0$ and $\omega>0$ are defined as in \eqref{eq+3.4}-\eqref{eq+3.6}. Let $j$, $\epsilon$, $t_0$ and $\tau$ be as in the statement, and define $\hat{\epsilon}$ as in~\eqref{defhateps}. Since $\sup_{x\in\overline{\mathcal{H}_j},\, x\cdot e_j\ge L}|u(t_0,x)-\phi(x\cdot e_j-c_f t_0 +\tau)|\le \varepsilon$, a similar argument to those of Lemmas~\ref{lemma4.1} and~\ref{lemma4.1+} imply that the following functions
$$\max\left(\phi(x\cdot e_j-c_f t-\omega\hat{\varepsilon} e^{-\delta (t-t_0)}+\omega \hat{\varepsilon}+\tau)-\hat{\varepsilon}\delta e^{-\delta (t-t_0)}-\hat{\varepsilon}\delta e^{- \mu (x\cdot e_j-L)}, 0\right)$$
and
$$\min\left(\phi(x\cdot e_j-c_f t+\omega\hat{\varepsilon} e^{-\delta (t-t_0)}-\omega \hat{\varepsilon}+\tau)+\hat{\varepsilon}\delta e^{-\delta (t-t_0)}+\hat{\varepsilon}\delta e^{- \mu (x\cdot e_j-L)},1\right)$$
are respectively a sub-solution and a super-solution of the problem satisfied by $u(t,x)$ for all $t\ge t_0$ and $x\in\overline{\mathcal{H}_j}$ with $x\cdot e_j\ge L$ (notice in particular that the assumptions made on $\phi(L-c_ft_0+\tau)$ and on~$u(t,x)$ with $t\ge t_0$ and $x\cdot e_j=L$ imply that $u$ is always trapped between the sub- and the super-solutions at these points). It then follows  that
\begin{align*}
\phi(x\cdot e_j-c_f t-\omega\hat{\varepsilon} e^{-\delta (t-t_0)}+&\omega \hat{\varepsilon}+\tau)-\hat{\varepsilon}\delta e^{-\delta (t-t_0)}-\hat{\varepsilon}\delta e^{- \mu (x\cdot e_j-L)}\le u(t,x)\\
&\le \phi(x\cdot e_j-c_f t+\omega\hat{\varepsilon} e^{-\delta (t-t_0)}-\omega \hat{\varepsilon}+\tau)+\hat{\varepsilon}\delta e^{-\delta (t-t_0)}+\hat{\varepsilon}\delta e^{- \mu (x\cdot e_j-L)}
\end{align*}
for all $t\ge t_0$ and $x\in\overline{\mathcal{H}_j}$ such that $x\cdot e_j\ge L$. For these $t$ and $x$, since $\phi'<0$, one infers that
$$u(t,x)\le \phi(x\cdot e_j-c_f t-\omega \hat{\varepsilon}+\tau)+2\hat{\varepsilon}\delta\le \phi(x\cdot e_j-c_f t+\tau)+\omega \hat{\varepsilon}\|\phi'\|_{L^{\infty}(\R)}+2\hat{\varepsilon}\delta.$$
Similarly, one can prove that $u(t,x)\ge \phi(x\cdot e_j-c_f t+\tau)-\omega \hat{\varepsilon}\|\phi'\|_{L^{\infty}(\R)}-2\hat{\varepsilon}\delta$. As a consequence, one has
$$\sup_{x\in\overline{\mathcal{H}_j},\, x\cdot e_j\ge L} |u(t,x)-\phi(x\cdot e_j-c t+\tau)|\le \omega \hat{\varepsilon}\|\phi'\|_{L^{\infty}(\R)}+2\hat{\varepsilon}\delta=M\epsilon\ \text{ for all $t\ge t_0$},$$
with the constant $M=\omega\|\phi'\|_{L^{\infty}(\R)}/\delta+2$ being independent of $j$, $\epsilon$, $t_0$ and $\tau$.
\end{proof}
\vskip 0.3cm

The last auxiliary lemma is a Liouville-type result for the transition fronts connecting $0$ and~$1$ for equation~\eqref{eq1.1} in straight infinite cylinders.

\begin{lemma}\label{lemma3.2-}
Assume that $\Omega_\infty=\R\times\omega=\big\{x\in\R^N: x_1\in\R,\,x'\in\omega\subset\R^{N-1}\big\}$ is a smooth straight cylinder with bounded open connected section $\omega$ and let $v:\R\times\overline{\Omega_\infty}\to(0,1)$ be a transition front connecting $0$ and $1$ for~\eqref{eq1.1} in $\Omega_\infty$. Then there are $\sigma\in\{-1,1\}$ and $\tau^*$ such that
$$v(t,x)=\phi(\sigma x_1-c_f t+\tau^*)\ \hbox{ for all }x=(x_1,x')\in\overline{\Omega_\infty}.$$
\end{lemma}

\begin{proof}
Let $(\Omega^{\pm}_{v,t})_{t\in\R}$ and $(\Gamma_{v,t})_{t\in\R}$ be the sets satisfying~\eqref{eq1.3}-\eqref{eq1.7} and associated with the transition front~$v$. First of all, it immediately follows from~\eqref{eq1.3}-\eqref{eq1.4} that there exists $R>0$ such that, for every $t\in\R$ and every two points $(x_1,x')$ and $(y_1,y')$ on the same connected component of $\Gamma_{v,t}$, there holds $|x_1-y_1|\le R$. One then infers from~\eqref{eq1.3} and~\eqref{eq1.6} that, even if it means redefining the interfaces $\Gamma_{v,t}$, one can assume without loss of generality that
$$\Gamma_{v,t}=\bigcup_{k=1}^{n_t}\big\{x=(x_1,x')\in\Omega_\infty: x_1=\xi_{t,k}\big\}\ \hbox{ for every }t\in\R,$$
where $\xi_{t,1}<\cdots<\xi_{t,n_t}$ and the integers $n_t$ are bounded uniformly with respect to $t\in\R$. Since~$c_f>0$, one can then follow the arguments used in~\cite[Theorem~2.6]{H} (concerned with~\eqref{eq1.1} in~$\R^N$) to conclude that $v$ is a planar front $v(t,x)\equiv\phi(\pm x_1-c_ft+\tau^*)$ for some $\tau^*\in\R$ (here, the Neumann boundary conditions on $\partial\Omega_\infty$ do not make additional difficulties, since the solution $v$ can be compared with sub- and super-solutions depending only on the variable $x_1$).
\end{proof}


\subsection{Proof of Theorem~\ref{th10}}\label{sec32}

Let $t_1\in\R$, $t_2\in\R$, $\tau_1\in\R$, $\delta>0$ and $\mu>0$ be as in Lemma \ref{lemma4.1}. For every $j\in J$, $t\ge\max(t_1,t_2)$ and $x\in\overline{\mathcal{H}_j}$ with $x\cdot e_j\ge L$, there holds
\be\label{eq+4.9}\baa{l}
\phi(x\cdot e_j-c_f (t-t_2)-L)-\delta e^{-\delta (t-t_2)}-\delta e^{- \mu (x\cdot e_j-L)}\vspace{3pt}\\
\qquad\qquad\qquad\le u(t,x)\le \phi(x\cdot e_j-c_f (t-t_1)+\tau_1)+\delta e^{-\delta (t-t_1)}+\delta e^{- \mu (x\cdot e_j-L)}.\eaa
\ee

Consider now any sequence $(t_n)_{n\in\mathbb{N}}$ such that $t_n\rightarrow +\infty$ as $n\rightarrow +\infty$, and consider any~$j\in J$. Remember that $\mathcal{H}_j=\mathcal{H}_{e_j,\omega_j,x_j}$ is a straight half-cylinder as in~\eqref{defbranch}. For every $n\in\N$, let~$\mathcal{H}_{j}^n=\mathcal{H}_j-c_ft_ne_j$ be the shifted half-cylinder in the direction $-e_j$. The half-cylinders~$\mathcal{H}_j^n$ converge to a straight open cylinder $\mathcal{H}_j^\infty$ parallel to $e_j$ as $n\to+\infty$. From standard parabolic estimates, up to extraction of a subsequence, the functions
$$u_n(t,y)=u(t+t_n,y+c_ft_ne_j),$$
defined in $\R\times\overline{\Omega-c_ft_ne_j}$, converge locally uniformly in $(t,y)\in\R\times\overline{\mathcal{H}_j^\infty}$ to a solution $u_{\infty}(t,y)$~of
\begin{eqnarray*}\left\{\baa{lll}
(u_{\infty})_t-\Delta u_{\infty}=f(u_{\infty}), &t\in\R,\ y\in\mathcal{H}_j^{\infty},\vspace{3pt}\\
(u_{\infty})_\nu=0, &t\in\R,\ y\in\partial\mathcal{H}_j^{\infty}.\eaa\right.
\end{eqnarray*}
It follows from \eqref{eq+4.9} that
\begin{align*}
\phi(y\cdot e_j -c_f (t-t_2)-L)\le u_{\infty}(t,y)\le \phi(y\cdot e_j-c_f (t-t_1)+\tau_1)
\end{align*}
for all $(t,y)\in\R\times\overline{\mathcal{H}_j^{\infty}}$. In particular, $u_\infty$ is a transition front connecting $0$ and $1$ for~\eqref{eq1.1} in the straight cylinder $\mathcal{H}_j^\infty$. Lemma~\ref{lemma3.2-} (applied to a rotated cylinder) then yields the existence of~$\tau_j\in\R$ such that $u_{\infty}(t,y)=\phi(y\cdot e_j-c_f t+\tau_j)$ for all $(t,y)\in\R\times\overline{\mathcal{H}_j^{\infty}}$. Therefore,
\be\label{2.22}
u_n(t,y)\rightarrow \phi(y\cdot e_j-c_f t+\tau_j)\ \ \text{locally uniformly in $\R\times\overline{\mathcal{H}_j^{\infty}}$ as $n\rightarrow +\infty$}.
\ee

Pick now any $\varepsilon>0$, and let $K>0$ such that $\phi\ge 1-\varepsilon/2$ in $(-\infty,-K]$ and $\phi\le \varepsilon/2$ in~$[K,+\infty)$. Define $K_1=\max(K-c_ft_1-\tau_1,K-\tau_j)$ and $K_2=\min(-K-c_ft_2+L,-K-\tau_j)<K_1$. It then follows from \eqref{2.22} that
\be\label{eq+4.12}
\sup_{y\in\overline{\mathcal{H}_j^{n}},\ K_2\le y\cdot e_j\le K_1}|u_n(0,y)-\phi(y\cdot e_j+\tau_j)|\le\varepsilon\ \  \text{for $n$ large enough}.
\ee
Since $t_n\rightarrow +\infty$ as $n\rightarrow +\infty$,~\eqref{eq+4.9} implies that, for $n$ large enough,
$$\left\{\baa{ll}
0<u_n(0,y)\le \varepsilon & \text{for all $y\in\overline{\mathcal{H}_j^n}$ such that $y\cdot e_j\ge K_1$},\vspace{3pt}\\
1-\varepsilon\le u_n(0,y) <1 & \displaystyle\text{for all $y\in\overline{\mathcal{H}_j^n}$ such that $K_2-\frac{c_f}{2} t_n\le y\cdot e_j\le K_2$}.\eaa\right.$$
Since $K_1\ge K-\tau_j$ and $K_2\le -K-\tau_j$, one has $0<\phi(y\cdot e_j+\tau_j)\le \varepsilon/2\le\epsilon$ for all $y\in\overline{\mathcal{H}_j^n}$ with~$y\cdot e_j\ge K_1$, and $1-\varepsilon\le1-\epsilon/2\le \phi(y\cdot e_j+\tau_j)<1$ for all $y\in\overline{\mathcal{H}_j^n}$ with $y\cdot e_j\le K_2$. It then can be deduced that, for $n$ large enough,
$$|u_n(0,y)-\phi(y\cdot e_j+\tau_j)|\le \varepsilon\ \  \text{for all $y\in\overline{\mathcal{H}_j^n}$ with $y\cdot e_j\ge K_1$ or $K_2-\frac{c_f}{2} t_n\le y\cdot e_j\le K_2$}.$$
Together with \eqref{eq+4.12} and the definitions of $u_n(t,y)$ and $\mathcal{H}_j^n$, one gets that, for $n$ large enough,
\be\label{eq+4.14}
|u(t_n,x)-\phi(x\cdot e_j-c_f t_n+\tau_j)|\le \varepsilon\ \  \text{for all $x\in\overline{\mathcal{H}_j}$ such that $x\cdot e_j\ge K_2+\frac{c_f}{2} t_n$}.
\ee
Notice that $\phi(x\cdot e_j-c_f t_n+\tau_j)\ge 1-\varepsilon/2\ge1-\epsilon$ for $n$ large enough and for all $x\in\overline{\mathcal{H}_j}$ such that $L\le x\cdot e_j\le K_2+c_ft_n/2$. Moreover, from Lemma~\ref{lemma4.1+} one infers that, for $n$ large enough, $u(t_n,x)\ge\phi(x\cdot e_j-c_f(t_n-t_{\varepsilon})-L)-\varepsilon e^{-\delta(t_n-t_{\varepsilon})}-\varepsilon e^{-\mu(x\cdot e_j-L)}\ge 1-3\varepsilon$ for all $x\in\overline{\mathcal{H}_j}$ with $L\le x\cdot e_j\le K_2+c_ft_n/2$, where $t_\epsilon$ is given in Lemma~\ref{lemma4.1+}. Then for $n$ large enough we get that
$$|u(t_n,x)-\phi(x\cdot e_j-c_f t_n+\tau_j)|\le 3\varepsilon\ \ \text{for all $x\in\overline{\mathcal{H}_j}$ with $L\le x\cdot e_j\le K_2+\frac{c_f}{2} t_n$}.$$
Together with~\eqref{eq+4.14} one infers that, for $n$ large enough,
$$|u(t_n,x)-\phi(x\cdot e_j-c_f t_n+\tau_j)|\le 3\varepsilon\ \ \text{for all $x\in\overline{\mathcal{H}_j}$ with $x\cdot e_j\ge L$.}$$
Furthermore, since $\phi(-\infty)=1$ and $u(t,\cdot)\to1$ as $t\to+\infty$ locally uniformly in $\overline{\Omega}$, one infers that, for $n$ large enough, $\phi(L-c_ft_n+\tau_j)\ge1-3\epsilon$ and $u(t,x)\ge1-3\epsilon$ for all $t\ge t_n$ and $x\in\overline{\mathcal{H}_j}$ with $x\cdot e_j=L$. It then follows from Lemma~\ref{lemma4.2} that, for $n$ large enough, $|u(t,x)-\phi(x\cdot e_j-c_f t+\tau_j)|\le 3M\epsilon$ for all $t\ge t_n$ and $x\in\overline{\mathcal{H}_j}$ with $x\cdot e_j\ge L$, where the constant $M\ge0$ is given in Lemma~\ref{lemma4.2}. Since $\varepsilon>0$ is arbitrary, one infers that
$$u(t,x)-\phi(x\cdot e_j-c_f t+\tau_j)\to0\ \ \text{uniformly for $x\in\overline{\mathcal{H}_j}$ such that $x\cdot e_j\ge L$, as $t\rightarrow +\infty$}.$$

Since $j\in J$ was arbitrary and since $u(t,\cdot)\to1$ as $t\to+\infty$ locally uniformly in $\overline{\Omega}$, one obtains that $u(t,x)-\phi(x\cdot e_j-c_f t+\tau_j)\to0$ uniformly in $\cup_{j\in J}(\overline{\mathcal{H}_j}\cap\overline{\Omega})$ as $t\rightarrow +\infty$. Furthermore, since $u$ is increasing in time $t$, it also follows from~\eqref{complete} and~\eqref{frontlike} that $u(t,\cdot)\to1$ as $t\to+\infty$ uniformly in $\overline{\Omega\setminus\cup_{j\in J}\mathcal{H}_j}$. As a consequence,~\eqref{largetime} holds. Finally, with~\eqref{frontlike} and~\eqref{largetime}, it is elementary to check that $u$ is a transition front in the sense of Definition \ref{TF} with sets $\Gamma_t$ and~$\Omega_t^{\pm}$ defined by \eqref{eq+1.11} and~\eqref{eq+1.12}. The proof of Theorem~\ref{th10} is thereby complete.\hfill$\Box$


\subsection{Lower estimates of the time-derivatives of the solutions consi\-dered in Theorem~\ref{th10}}\label{sec33}

This section is devoted to the proof of some lower bounds for the time-derivatives of the solutions~$u$ considered in Theorem~\ref{th10}. The following lemma will actually be used later in Section~\ref{sec4.2} for the proof of Theorem~\ref{th6}.

\begin{lemma}\label{lemut}
Let $\Omega$, $I$, $J$ and $u$ be as in Theorem~$\ref{th10}$. Then, for every $0\!<\!a\!\le\! b\!<\!1$, there holds
\be\label{infab}
\inf_{(t,x)\in\R\times\overline{\Omega},\, a\le u(t,x)\le b}\,u_t(t,x)>0.
\ee
\end{lemma}

\begin{proof}
First of all, from the regularity assumptions on $f$ and $\Omega$ and from standard parabolic estimates, the function $u_t$ is a classical solution of the equation $(u_t)_t=\Delta u_t+f'(u)u_t$ in $\R\times\overline{\Omega}$ with Neumann boundary conditions $(u_t)_\nu=0$ on $\R\times\partial\Omega$. The strong parabolic maximum principle applied to the non-negative function~$u_t$ then yields $u_t>0$ in $\R\times\overline{\Omega}$ (the function $u_t$ can indeed not be identically $0$ because of~\eqref{frontlike}).

Assume now by way of contradiction that~\eqref{infab} does not hold. Then there is a sequence $(t_k,x_k)_{k\in\N}$ in~$\R\times\overline{\Omega}$ such that $a\le u(t_k,x_k)\le b$ for all $k\in\N$ and $u_t(t_k,x_k)\to0$ as $k\to+\infty$. Three cases may then occur up to extraction of a subsequence: either $t_k\to t_\infty\in\R$, or $t_k\to-\infty$, or $t_k\to+\infty$.

Consider first the case $t_k\to t_\infty\in\R$. Notice that, since $u$ is time-increasing and satisfies~\eqref{frontlike} and~\eqref{largetime}, it immediately follows that $u(t,x)\to1$ as $|x|\to+\infty$ with $x\in\cup_{i\in I}\overline{\mathcal{H}_i}$ and $u(t,x)\to0$ as $|x|\to+\infty$ with $x\in\cup_{j\in J}\overline{\mathcal{H}_j}$, locally uniformly in $t\in\R$. Therefore, since $0<a\le u(t_k,x_k)\le b<1$, the assumption $t_k\to t_\infty\in\R$ implies that the sequence $(x_k)_{k\in\N}$ is bounded and thus converges, up to extraction of a subsequence, to a point $x_\infty\in\overline{\Omega}$. As a consequence, $u_t(t_\infty,x_\infty)=0$, which is impossible.

Consider then the case $t_k\to-\infty$. By~\eqref{frontlike}, one gets that, up to extraction of a subsequence, there is $i\in I$ such that $x_k\in\overline{\mathcal{H}_i}$ for all $k\in\N$, with $\sup_{k\in\N}\big|x_k\cdot e_i-c_f|t_k|\big|<+\infty$. Therefore, using again~\eqref{frontlike} and standard parabolic estimates, the functions $(t,x)\mapsto u(t+t_k,x+x_k)$ converge in $C^{1,2}_{t,x}(\R\times\overline{\mathcal{H}_i^\infty})$ locally, up to extraction of another subsequence, to the front $\phi(x\cdot e_i-c_ft+\sigma)$ for some $\sigma\in\R$, where $\mathcal{H}_i^\infty$ is a straight infinite cylinder parallel to $e_i$ (as in the proof of Theorem~\ref{th10} in Section~\ref{sec32}). In particular, $u_t(t_k,x_k)\to-c_f\phi'(\sigma)>0$ as $k\to+\infty$, a contradiction. This case is then ruled out too.

The case $t_k\to+\infty$ can be ruled out similarly to the previous one, using now~\eqref{largetime} instead of~\eqref{frontlike}. The proof of Lemma~\ref{lemut} is thereby complete.
\end{proof}


\section{Transition fronts in domains with multiple branches: further properties}\label{sec4}

This section is devoted to the proof of Theorem~\ref{th6} and Corollaries~\ref{th5} and~\ref{cor4} on the existence and uniqueness of the global mean speed of any transition front connecting $0$ and $1$ in domains with multiple cylindrical branches. Throughout this section, $\Omega$ is a smooth domain with $m\,(\ge2)$ cylindrical branches. Keeping in mind~\eqref{F1}-\eqref{F2}, we consider
$$0<\theta_1\le\theta_2<1\ \hbox{ and }\ \delta>0$$
as in~\eqref{deftheta12}-\eqref{d}. Section~\ref{sec41} is devoted to the proof of some key-lemmas mimicking those of Section~\ref{sec2.1} but adapted to the new type of geometry. Theorem~\ref{th6} and Corollaries~\ref{th5} and~\ref{cor4} are proved in the following sections.


\subsection{Key-lemmas}\label{sec41}

This section is devoted to the proof of 3 key-lemmas on the spreading and contracting speeds of the solutions of the Cauchy problem which are initially close to $1$, resp. to $0$, in some sections of the branches $\mathcal{H}_i$ or in a large central region, and are equal to $0$, resp. $1$, elsewhere. We first remind that $L>0$ is given in~\eqref{branches}. For any branch $\mathcal{H}_i$ and for any $l>0$ and $R>0$ such that~$l\ge R+L$, let $v_{i,l,R}(t,x)$ denote the solution of the Cauchy problem
\begin{eqnarray}\label{3.2}\left\{\baa{lll}
(v_{i,l,R})_t-\Delta v_{i,l,R}=f(v_{i,l,R}), &t>0,\ x\in\overline{\Omega},\vspace{3pt}\\
(v_{i,l,R})_\nu=0, &t>0,\ x\in\partial\Omega\eaa\right.
\end{eqnarray}
with initial condition
$$v_{i,l,R}(0,x)=1-\delta\hbox{ for }x\in\overline{\mathcal{H}_i}\hbox{ with }l-R<x\cdot e_i<l+R\ \hbox{ and }\ v_{i,l,R}(0,x)=0\hbox{ elsewhere in }\overline{\Omega}.$$

\begin{lemma}\label{lemma3.3}
For any $\varepsilon\in(0,c_f)$, there exist some real numbers $L_{\varepsilon}>L$ and $R_{\varepsilon}>0$, such that for any $i\in\{1,\cdots,m\}$ and any $l\ge R_{\varepsilon}+L_{\varepsilon}$, there holds
\be\label{vilReps}
v_{i,l,R_{\varepsilon}}(t,x)\!\ge\!1\!-\!2\delta\ \text{for all $0\le t\le T_{\varepsilon}=\frac{l\!-\!R_{\varepsilon}\!-\!L_{\varepsilon}}{c_f-\varepsilon}$ and $x\in\overline{\mathcal{H}_i}$ with $|x\cdot e_i\!-\!l|\!\le\!(c_f\!-\!\varepsilon)t$}
\ee
$($notice that these points $x$ belong to $\overline{\mathcal{H}_i}\cap\overline{\Omega}$ since $x\cdot e_i\ge l-(c_f-\epsilon)T_\epsilon=R_\epsilon+L_\epsilon>L$$)$ and
\be\label{vilRepsbis}
v_{i,l,R_{\varepsilon}}(t,x)\ge 1-3\delta\ \text{for all $t\ge T_{\varepsilon}$ and $x\in\overline{\mathcal{H}_i}$ with $R_{\varepsilon}+L_{\varepsilon}\le x\cdot e_i\le l+(c_f-\varepsilon)t$}.
\ee
\end{lemma}

\begin{proof}
Take any $\varepsilon\in(0,c_f)$. Let the parameters $C>0$, $k>0$ and $\omega>0$ be as in~\eqref{eq+2.2}-\eqref{eq+2.4-} in Step~1 of the proof of Lemma~\ref{lemma2.2}. Define
\be\label{5.13-}
0<\delta_{\varepsilon}=\min\Big(\frac{\varepsilon k}{4\max_{[0,1]}|f'|},\frac{\delta}{2}\Big)\le\min\Big(\frac{c_fk}{4\max_{[0,1]}|f'|},\frac{\delta}{2}\Big)
\ee
and let $\alpha\in(0,1]$ be such that
\be\label{defalpha}
0\le f(s)\le\min(|f'(0)|,|f'(1)|)\,\delta_\epsilon=\min(-\delta_\epsilon f'(0),-\delta_\epsilon f'(1))\ \hbox{ for all }s\in[1-\alpha\delta_\epsilon,1].
\ee
Let $C_\epsilon>C>0$ be such that
\be\label{Cepsbis}
\phi\ge1-\alpha\delta_\epsilon\hbox{ and }|\phi''|\le\min(|f'(0)|,|f'(1)|)\,\delta_\epsilon\hbox{ in }(-\infty,-C_\epsilon],\ \hbox{ and }\ \phi\le\delta_\epsilon\hbox{ in }[C_\epsilon,+\infty),
\ee
and let $h_\epsilon: [0,+\infty)\rightarrow (0,+\infty)$ be a $C^2([0,+\infty))$ function satisfying
\be\label{defheps}\left\{\baa{l}
0\le h'_{\varepsilon}\le 1\hbox{ and }\displaystyle h''_{\varepsilon}\le \frac{\varepsilon}{2} \text{ in $[0,+\infty)$},\vspace{3pt}\\
h'_{\varepsilon}=0 \text{ in a neighborhood of $0$},\ \ h_{\varepsilon}(r)=r \text{ in $[H_{\varepsilon},+\infty)$ for some $H_{\varepsilon}>0$}.\eaa\right.
\ee
Furthermore, define
\be\label{defRLepsbis}
R_{\varepsilon}=\max\big(H_{\varepsilon},h_{\varepsilon}(0)+2\omega+C_{\varepsilon}+C\big)>0\ \hbox{ and }\ L_{\varepsilon}=L-C+C_{\varepsilon}>L>0
\ee
(the $2\omega$ term in the definition of $R_\epsilon$, instead of only $\omega$ in the similar definition~\eqref{defRLeps}, is used at the end of the proof of the present lemma, see Footnote~\ref{footnote4} below). Finally, consider any~$i\in\{1,\cdots,m\}$ and any real number $l\ge R_\epsilon+L_\epsilon$, define~$T_\epsilon\ge0$ as in~\eqref{vilReps}, and let us show the lower bounds~\eqref{vilReps}-\eqref{vilRepsbis} for the function $v_{i,l,R_\epsilon}$ with these parameters (notice that~\eqref{vilReps} is actually immediate if $T_\epsilon=0$, namely $l=R_\epsilon+L_\epsilon$).

Let now $\underline{v}:[0,T_\epsilon]\times\overline{\Omega}\to[0,1)$ be the function defined by
\begin{eqnarray*}
\underline{v}(t,x)=\left\{\baa{lll}
\max\left(\phi(\underline{\zeta}(t,x))-\delta e^{-\delta t}-\delta_{\varepsilon},0\right) &\text{for $t\in[0,T_\epsilon]$ and $x\in\overline{\mathcal{H}_i}\cap\overline{\Omega}$},\vspace{3pt}\\
0&\text{for $t\in[0,T_\epsilon]$ and $x\in\overline{\Omega}\setminus\overline{\mathcal{H}_i}$},\eaa\right.
\end{eqnarray*}
where
$$\underline{\zeta}(t,x)=h_{\varepsilon}(|x\cdot e_i-l|) -(c_f-\varepsilon) t -\omega e^{-\delta t} +\omega-R_{\varepsilon}+C.$$
Notice that $\underline{v}$ is continuous. Indeed, since $l-L\ge l-L_{\varepsilon}\ge R_{\varepsilon}\ge H_{\varepsilon}$, it follows that, for any $0\le t\le T_{\varepsilon}$ and $x\in\overline{\mathcal{H}_i}$ with $x\cdot e_i\le L$, one has $\underline{\zeta}(t,x)\ge l-L-(c_f-\varepsilon)T_{\varepsilon}-R_{\varepsilon}+C=L_{\varepsilon}-L+C=C_{\varepsilon}$, hence $\phi(\underline{\zeta}(t,x))\le \delta_\varepsilon$ and $\underline{v}(t,x)=\max\left(\phi(\underline{\zeta}(t,x))-\delta e^{-\delta t}-\delta_{\varepsilon},0\right)=0$. The continuity of $\underline{v}$ then follows from its definition and from~\eqref{branches}. Observe also that $\underline{v}$ is of class $C^2$ in the set where it is positive.

Firstly, one can then follow the same arguments as in Step~2 of the proof of Lemma~\ref{lemma2.2} to show that the function $\underline{v}$ is a sub-solution of the problem satisfied by $v_{i,l,R_{\varepsilon}}(t,x)$ for $0\le t\le T_{\varepsilon}$ and $x\in\overline{\Omega}$. It then follows from the comparison principle that
\be\label{vReps}
v_{i,l,R_{\varepsilon}}(t,x)\ge \underline{v}(t,x)\ \hbox{ for all }0\le t\le T_{\varepsilon}\hbox{ and }x\in\overline\Omega.
\ee
For any $0\le t\le T_{\varepsilon}$ and $x\in\overline{\mathcal{H}_i}$ with $|x\cdot e_i-l|\le (c_f-\varepsilon)t$, one has $x\in\overline{\Omega}$ (as noticed in the statement of the lemma) and $\underline{\zeta}(t,x)\le |x\cdot e_i-l|+h_{\varepsilon}(0)-(c_f-\varepsilon)t+\omega-R_{\varepsilon}+C\le -C_{\varepsilon}$, hence $\phi(\underline{\zeta}(t,x))\ge 1-\alpha\delta_{\varepsilon}\ge1-\delta_\epsilon$ by~\eqref{Cepsbis}. Since $\delta_{\varepsilon}\le \delta/2$, we then get $v_{i,l,R_{\varepsilon}}(t,x)\ge \underline{v}(t,x)\ge 1-\delta_{\varepsilon}-\delta e^{-\delta t}-\delta_{\varepsilon}\ge 1-2\delta$ for all $0\le t\le T_{\varepsilon}$ and $x\in\overline{\mathcal{H}_i}$ such that $|x\cdot e_i-l|\le  (c_f-\varepsilon)t$. This provides the desired lower bound~\eqref{vilReps}.

Secondly, in order to show the second lower bound~\eqref{vilRepsbis}, let us set, for $t\ge T_{\varepsilon}$,
\begin{eqnarray*}
\underline{w}(t,x)=\left\{\baa{lll}
\max\left(\phi(\xi_1(t,x))+\phi(\xi_2(t,x))-1-2\delta_{\varepsilon}-\delta e^{-\delta (t-T_{\varepsilon})},0\right) &\text{for $x\in\overline{\mathcal{H}_i}\cap\overline{\Omega}$},\vspace{3pt}\\
0 &\text{for $x\in\overline{\Omega}\setminus\overline{\mathcal{H}_i}$},\eaa\right.
\end{eqnarray*}
where
$$\left\{\baa{rcl}
\xi_1(t,x) & = & -x\cdot e_i -\omega e^{-\delta (t-T_{\varepsilon})}+L_{\varepsilon} +C+2\omega+h_{\varepsilon}(0),\vspace{3pt}\\
\xi_2(t,x) & = & x\cdot e_i -(c_f-\varepsilon) t-\omega e^{-\delta (t-T_{\varepsilon})}-l+2\omega+h_{\varepsilon}(0)+C-R_{\varepsilon}.\eaa\right.$$
Notice that for all $t\ge T_{\varepsilon}$ and $x\in\overline{\mathcal{H}_i}$ with $x\cdot e_i\le L$, one has $\xi_1(t,x)\ge -L+L_{\varepsilon}+C+\omega+h_{\varepsilon}(0)\ge C_{\varepsilon}$, hence $\phi(\xi_1(t,x))\leq \delta_\varepsilon$ and
$$\max\big(\phi(\xi_1(t,x))+\phi(\xi_2(t,x))-1-2\delta_{\varepsilon}-\delta e^{-\delta (t-T_{\varepsilon})},0\big)\le \max(\delta_{\varepsilon}+1-1-2\delta_{\varepsilon}-\delta e^{-\delta (t-T_{\varepsilon})},0)=0.$$
Therefore, $\underline{w}$ is continuous in $[T_\epsilon,+\infty)\times\overline{\Omega}$ (and of class $C^2$ in the set where it is positive).

Let us prove that $\underline{w}(t,x)$ is a sub-solution of the problem satisfied by $v_{i,l,R_{\varepsilon}}(t,x)$ for $t\ge T_{\varepsilon}$ and $x\in\overline{\Omega}$. Let us first check the initial (at time $T_\epsilon$) and boundary conditions. First of all, at time $T_{\varepsilon}$, one has $\xi_2(T_{\varepsilon},x)\le l-l +\omega+h_{\varepsilon}(0)+C-R_{\varepsilon}\le -C_{\varepsilon}$ for all $x\in\overline{\mathcal{H}_i}\cap\overline{\Omega}$ such that~$x\cdot e_i\le l$, hence $\phi(\xi_2(T_{\varepsilon},x))\ge 1-\alpha\delta_{\varepsilon}\ge1-\delta_\epsilon$ by~\eqref{Cepsbis}. Moreover, for all such $x$, one has
$$\underline{\zeta}(T_{\varepsilon},x)\le |x\cdot e_i-l|+h_{\varepsilon}(0)-(c_f-\varepsilon)T_{\varepsilon}+\omega-R_{\varepsilon}+C=-x\cdot e_i+h_{\varepsilon}(0)+L_{\varepsilon}+\omega+C=\xi_1(T_{\varepsilon},x).$$
Since $\phi$ is decreasing, it follows that $\phi(\underline{\zeta}(T_{\varepsilon},x))\ge \phi(\xi_1(T_{\varepsilon},x))$ and, together with~\eqref{vReps},
$$\underline{w}(T_{\varepsilon},x)\le \underline{v}(T_{\varepsilon},x)\le v_{i,l,R_{\varepsilon}}(T_{\varepsilon},x)\ \  \text{for all $x\in\overline{\mathcal{H}_i}\cap\overline{\Omega}$ such that $x\cdot e_i\le l$}.$$
Since $l\ge L_{\varepsilon}+R_{\varepsilon}$, one also has $\xi_1(T_{\varepsilon},x)\le-l+L_{\varepsilon}+C+\omega+h_{\varepsilon}(0)\le -R_{\varepsilon}+C+\omega+h_{\varepsilon}(0)\le -C_{\varepsilon}$ for all $x\in\overline{\mathcal{H}_i}$ with $x\cdot e_i\ge l$ (notice that these points belong to $\overline{\Omega}$ by~\eqref{branches}), hence~$\phi(\xi_1(t,x))\ge 1-\delta_{\varepsilon}\ge1-\alpha\delta_{\varepsilon}$. Moreover, for all such $x$,
$$\underline{\zeta}(T_{\varepsilon},x)\!\le\!|x\cdot e_i-l|+h_{\varepsilon}(0)-(c_f-\varepsilon)T_{\varepsilon}+\omega-R_{\varepsilon}+C\!=\!x\cdot e_i-l+h_{\varepsilon}(0)-(c_f-\varepsilon)T_{\varepsilon}+\omega-R_{\varepsilon}+C\!=\!\xi_2(T_{\varepsilon},x).$$
Since $\phi$ is decreasing, it follows that $\phi(\underline{\zeta}(T_{\varepsilon},x))\ge \phi(\xi_2(t,x))$ and
$$\underline{w}(T_{\varepsilon},x)\le \underline{v}(T_{\varepsilon},x)\le v_{i,l,R_{\varepsilon}}(T_{\varepsilon},x)\ \  \text{for all $x\in\overline{\mathcal{H}_i}$ such that $x\cdot e_i\ge l$}.$$
Therefore, one infers that
$$\underline{w}(T_{\varepsilon},x)\le \underline{v}(T_{\varepsilon},x)\le v_{i,l,R_{\varepsilon}}(T_{\varepsilon},x)\ \  \text{for all $x\in\overline\Omega$}.$$
For all $t\ge T_\epsilon$ and $x\in\overline{\mathcal{H}_i}\cap\overline{\Omega}$ with $x\cdot e_i\le L$, one has $\xi_1(t,x)\ge -L+L_\varepsilon+C+\omega+h_\varepsilon(0)\ge C_\varepsilon$, hence $\phi(\xi_1(t,x))\le \delta_\varepsilon$ and $\underline{w}(t,x)=0$. One also has $\underline{w}(t,x)=0$ for all $t\ge T_\epsilon$ and $x\in\overline{\Omega}\setminus\overline{\mathcal{H}_i}$. Moreover, it is immediate to see that $ \underline{w}_\nu(t,x)=0$ for all $t\ge T_{\varepsilon}$ and $x\in\partial\mathcal{H}_i$ with $x\cdot e_i>L$. As a consequence, $\underline{w}_\nu(t,x)=0$ for all $t\ge T_\epsilon$ and $x\in\partial\Omega$ such that $\underline{w}(t,x)>0$.

Let us now verify that
$$\mathcal{L} \underline{w}(t,x)=\underline{w}_t(t,x) -\Delta \underline{w}(t,x) -f(\underline{w}(t,x))\le 0\hbox{ for all }t\ge T_{\varepsilon}\hbox{ and }x\in\overline{\mathcal{H}_i}\cap\overline{\Omega}\hbox{ with }\underline{w}(t,x)>0$$
(remember that $\underline{w}\equiv 0$ in $[T_\epsilon,+\infty)\times(\overline{\Omega}\setminus\overline{\mathcal{H}_i})$). For such $(t,x)$, a direct computation yields
\be\baa{rcl}\label{eq5.14}
\mathcal{L} \underline{w}(t,x) & = & -\phi''(\xi_1(t,x))+\omega\delta e^{-\delta (t-T_{\varepsilon})}\phi'(\xi_1(t,x))+\varepsilon\phi'(\xi_2(t,x))+\omega\delta e^{-\delta (t-T_{\varepsilon})}\phi'(\xi_2(t,x))\,\nonumber\vspace{3pt}\\
& & +\delta^2 e^{-\delta (t-T_{\varepsilon})}+f(\phi(\xi_2(t,x)))-f(\underline{w}(t,x))\vspace{3pt}\\
& \le & c_f \phi'(\xi_1(t,x))+\omega\delta e^{-\delta (t-T_{\varepsilon})}\phi'(\xi_1(t,x))+\delta^2 e^{-\delta (t-T_{\varepsilon})}+f(\phi(\xi_1(t,x)))-f(\underline{w}(t,x))\vspace{3pt}\\
& & +f(\phi(\xi_2(t,x))),\eaa
\ee
since $\phi$ is decreasing and $\phi''+c_f\phi'+f(\phi)=0$.

Consider first the case $x\cdot e_i\le l$. One has $\xi_2(t,x)\le l-(c_f-\varepsilon)T_{\varepsilon}-l+2\omega+h_{\varepsilon}(0)+C-R_{\varepsilon}\le -C_{\varepsilon}$, hence $\phi(\xi_2(t,x))\ge 1-\alpha\delta_{\varepsilon}$ and $f(\phi(\xi_2(t,x)))\le\min(|f'(0)|, |f'(1)|)\delta_\varepsilon=\min(-\delta_\epsilon f'(0),-\delta_\epsilon f'(1))$ by~\eqref{defalpha}. If $\xi_1(t,x)<-C$, then $1>\phi(\xi_1(t,x))\ge 1-\delta$ and then
$$1>\phi(\xi_1(t,x))\ge\underline{w}(t,x)\ge 1-\delta+1-\alpha\delta_{\varepsilon}-1-2\delta_{\varepsilon}-\delta\ge 1-4\delta.$$
It follows from \eqref{d} that $f(\phi(\xi_1(t,x)))-f(\underline{w}(t,x))\le(f'(1)/2)(1-\phi(\xi_2(t,x))+2\delta_{\varepsilon}+\delta e^{-\delta (t-T_{\varepsilon})})$ and, together with $\phi'<0$,
$$\mathcal{L} \underline{w}(t,x)\le  \delta^2 e^{-\delta (t-T_{\varepsilon})}+\frac{f'(1)}{2}(1-\phi(\xi_2(t,x))+2\delta_{\varepsilon}+\delta e^{-\delta (t-T_{\varepsilon})})-\delta_\varepsilon f'(1)\le 0.$$
If $-C\le \xi_1(t,x)\le C$, then $\phi'(\xi_1(t,x))\le-k$ and $f(\phi(\xi_2(t,x)))\le \delta_{\varepsilon} \max_{[0,1]}|f'|$. It then follows from the definition of $\omega$ in~\eqref{eq+2.4-}, from~\eqref{5.13-} and from the property $\alpha\in(0,1]$ that
$$\mathcal{L} \underline{w}(t,x)\le -c_f k -k\omega\delta e^{-\delta (t-T_{\varepsilon})}+\delta^2 e^{-\delta (t-T_{\varepsilon})}+\max_{[0,1]}|f'|(\alpha\delta_\epsilon+2\delta_{\varepsilon}+\delta e^{-\delta (t-T_{\varepsilon})}))+\delta_{\varepsilon}\max_{[0,1]}|f'|\le 0.$$
If $\xi_1(t,x)>C$, then $0<\phi(\xi_1(t,x))\le \delta$, $\underline{w}(t,x)\le \delta-2\delta_{\varepsilon}\le\delta$, and $\underline{w}(t,x)\le\phi(\xi_1(t,x))$. It follows from~\eqref{d} that $f(\phi(\xi_1(t,x)))-f(\underline{w}(t,x))\le(f'(0)/2)(1-\phi(\xi_2(t,x))+2\delta_{\varepsilon}+\delta e^{-\delta (t-T_{\varepsilon})})$ and
$$\mathcal{L} \underline{w}(t,x)\le  \delta^2 e^{-\delta (t-T_{\varepsilon})}+\frac{f'(0)}{2}(1-\phi(\xi_2(t,x))+2\delta_{\varepsilon}+\delta e^{-\delta (t-T_{\varepsilon})})-\delta_\varepsilon f'(0)\le 0.$$

Consider now the case $x\cdot e_i\ge l$. One then has
$$\xi_1(t,x)\le -l+L_{\varepsilon}+C+2\omega+h_{\varepsilon}(0)\le -R_{\varepsilon}+C+2\omega+h_{\varepsilon}(0)\le -C_{\varepsilon},$$
hence $\phi(\xi_1(t,x))\ge 1-\alpha\delta_{\varepsilon}\ge1-\delta$ and
$$|\phi''(\xi_1(t,x))|\le\min(|f'(0)|,|f'(1)|)\delta_\epsilon=\min(-\delta_\epsilon f'(0),-\delta_\epsilon f'(1))\le\delta_\epsilon\max_{[0,1]}|f'|$$
by~\eqref{Cepsbis}. If $\xi_2(t,x)<-C$, then $1>\phi(\xi_2(t,x))\ge 1-\delta$ and $\phi(\xi_2(t,x))\ge\underline{w}(t,x)\ge 1-4\delta$. It follows from~\eqref{d} that $f(\phi(\xi_2(t,x)))-f(\underline{w}(t,x))\le(f'(1)/2)(1-\phi(\xi_1(t,x))+2\delta_{\varepsilon}+\delta e^{-\delta (t-T_{\varepsilon})})$ and, together with~\eqref{eq5.14} and $\phi'<0$,
$$\mathcal{L}\underline{w}(t,x)\le -\delta_{\varepsilon}f'(1)+\delta^2 e^{-\delta (t-T_{\varepsilon})}+\frac{f'(1)}{2}(1-\phi(\xi_1(t,x))+2\delta_{\varepsilon}+\delta e^{-\delta (t-T_{\varepsilon})})\le 0.$$
Similarly, there holds $\mathcal{L}\underline{w}(t,x)\le 0$ if $\xi_2(t,x)>C$. Finally, if $-C\le \xi_2(t,x)\le C$, then~$\phi'(\xi_2(t,x))\le-k$ and, together with~\eqref{5.13-} and~\eqref{eq5.14},
$$\baa{rcl}
\mathcal{L}\underline{w}(t,x) & \le & \displaystyle-\phi''(\xi_1(t,x))-\varepsilon k -\omega k\delta e^{-\delta (t-T_{\varepsilon})} +\delta^2 e^{-\delta (t-T_{\varepsilon})}+\max_{[0,1]} |f'| (3\delta_{\varepsilon}+\delta e^{-\delta (t-T_{\varepsilon})})\vspace{3pt}\\
& \le & \displaystyle4\delta_{\varepsilon}\max_{[0,1]} |f'| -\varepsilon k -\delta e^{-\delta (t-T_{\varepsilon})}\Big(\omega k-\delta-\max_{[0,1]} |f'|\Big)\le 0.\eaa$$

Finally, one concludes that $\mathcal{L}\underline{w}(t,x)\le 0$ for all $t\ge T_{\varepsilon}$ and $x\in\overline{\mathcal{H}_i}\cap\overline{\Omega}$. The comparison principle then yields $v_{i,l,R_{\varepsilon}}(t,x)\ge \underline{w}(t,x)$ for all $t\ge T_{\varepsilon}$ and $x\in\overline{\Omega}$. Consider now any $t\ge T_\epsilon$ and $x\in\overline{\mathcal{H}_i}$ such that $R_{\varepsilon}+L_{\varepsilon}\le x\cdot e_i\le l +(c_f-\varepsilon)t$. There holds $\xi_1(t,x)\le -R_{\varepsilon}-L_{\varepsilon}+L_{\varepsilon}+C+2\omega+h_{\varepsilon}(0)\le -C_{\varepsilon}$ and $\xi_2(t,x)\le l-l+2\omega+h_{\varepsilon}(0)+C-R_{\varepsilon}\le -C_{\varepsilon}$,\footnote{\label{footnote4}In these upper bounds $\xi_1(t,x)\le-C_\epsilon$ and $\xi_2(t,x)\le-C_\epsilon$, we use the $2\omega$ term in the definition~\eqref{defRLepsbis} of $R_\epsilon$.} hence $\phi(\xi_1(t,x))\ge 1-\alpha\delta_\varepsilon\ge1-\delta_\epsilon$ and~$\phi(\xi_2(t,x))\ge 1-\alpha\delta_\varepsilon\ge 1-\delta_\varepsilon$. Since $\delta_{\varepsilon}\le \delta/2$ one gets that
$$v_{i,l,R_{\varepsilon}}(t,x)\ge\underline{w}(t,x)\ge 1-\delta_{\varepsilon}+1-\delta_{\varepsilon}-1-2\delta_{\varepsilon}-\delta e^{-\delta (t-T_{\varepsilon})}\ge 1-3\delta,$$
which is the desired inequality. The proof of Lemma~\ref{lemma3.3} is thereby complete.  
\end{proof}
\vskip 0.3cm

For any branch $\mathcal{H}_i$, and for any $l>0$ and $R>0$ such that $l-R\ge L$, let now $w_{i,l,R}(t,x)$ denote the solution of the Cauchy problem
\begin{eqnarray*}\left\{\baa{lll}
(w_{i,l,R})_t-\Delta w_{i,l,R}=f(w_{i,l,R}), &t>0,\ x\in\overline{\Omega},\vspace{3pt}\\
(w_{i,l,R})_{\nu}=0, &t>0,\ x\in\partial\Omega\eaa\right.
\end{eqnarray*}
with initial condition $w_{i,l,R}(0,x)=\delta$ for $x\in\overline{\mathcal{H}_i}$ such that $l-R<x\cdot e_i<l+R$ and $w_{i,l,R}(0,x)=1$ elsewhere in $\overline{\Omega}$.

\begin{lemma}\label{lemma5.2}
For any $\varepsilon\in(0,c_f)$, there exist some real numbers $L_{\varepsilon}>L$ and $R_{\varepsilon}>0$ such that, for any $i\in\{1,\cdots,m\}$, any $R>R_{\varepsilon}$ and any $l\ge R+L_{\varepsilon}$, there holds
$$w_{i,l,R}(t,x)\le 2\delta\ \text{for all $0\le t\le T_{\varepsilon}=\frac{R\!-\!R_{\varepsilon}}{c_f\!+\!\varepsilon}$ and $x\in\overline{\mathcal{H}_i}$ with $|x\cdot e_i\!-\!l|\le R\!-\!R_{\varepsilon}\!-\!(c_f\!+\!\varepsilon)t$}$$
$($notice that these points $x$ belong to $\overline{\mathcal{H}_i}\cap\overline{\Omega}$ since $x\cdot e_i\ge l-R+R_\epsilon+(c_f+\epsilon)T_\epsilon=l\ge R+L_\epsilon>L$$)$.
\end{lemma}

\begin{proof}
Take any $\varepsilon\in(0,c_f)$. Let the positive parameters $C$, $k$, $\omega$, $\delta_{\varepsilon}$, and $C_{\varepsilon}\ge C$ be defined as in~\eqref{eq+2.2}-\eqref{ghsa1} in Step~1 of the proof of Lemma~\ref{lemma2.2}. Let $h_\epsilon:[0,+\infty)\to(0,+\infty)$ be a $C^2$ function satisfying~\eqref{defheps} with some $H_\epsilon>0$, and let $R_\epsilon>0$ and $L_\epsilon>L$ be defined as in~\eqref{defRLeps} in Step~1 in the proof of Lemma~\ref{lemma2.2}, namely
$$R_{\varepsilon}=\max\big(H_{\varepsilon},h_{\varepsilon}(0)+\omega+C_{\varepsilon}+C\big)>0\ \hbox{ and }L_{\varepsilon}=L-C+C_{\varepsilon}>L>0.$$

Consider now any $i\in\{1,\cdots,m\}$, any $R>R_{\varepsilon}$ and any $l\ge R+L_\epsilon$. Set $T_{\varepsilon}=(R-R_{\varepsilon})/(c_f+\varepsilon)$ and define the function
\begin{eqnarray*}
\overline{w}(t,x)=\left\{\baa{lll}
\min\left(\phi(\overline{\zeta}(t,x))+\delta e^{-\delta t}+\delta_{\varepsilon},1\right)  &\text{ for $t\ge0$ and $x\in\overline{\mathcal{H}_i}\cap\overline{\Omega}$},\vspace{3pt}\\
1 &\text{ for $t\ge0$ and $x\in\overline{\Omega}\setminus\overline{\mathcal{H}_i}$,}\eaa\right.
\end{eqnarray*}
where
$$\overline{\zeta}(t,x)=-h_{\varepsilon}(|x\cdot e_i-l|) -(c_f+\varepsilon) t+\omega e^{-\delta t}-\omega+R-C.$$
Pick any $0\le t\le T_{\varepsilon}$ and $x\in\overline{\mathcal{H}_i}\cap\overline{\Omega}$ such that $x\cdot e_i\le L$. Since $l\ge R+ L_{\varepsilon}> L\ge x\cdot e_i$ and $l-L\ge l-L_{\varepsilon}\ge R>R_{\varepsilon}\ge H_{\varepsilon}$, there holds $h_{\varepsilon}(|x\cdot e_i-l|)= |x\cdot e_i-l|=l-x\cdot e_i$, $\overline{\zeta}(t,x)\le -l+x\cdot e_i+R-C\le -L_{\varepsilon}+L-C=-C_{\varepsilon}$ and  $\phi(\overline{\zeta}(t,x))\ge 1-\delta_\varepsilon$, hence~$\overline{w}(t,x)=\min\left(\phi(\overline{\zeta}(t,x))+\delta e^{-\delta t}+\delta_{\varepsilon},1\right)=1$. Given its definition, the function $\overline{w}$ is then continuous in $[0,+\infty)\times\overline{\Omega}$ (and of class $C^2$ in the set where it is less than $1$).

One can then follow the same argument as for the proof of~\eqref{eq+2.11} in Lemma~\ref{lemma2.4} to get that $\overline{w}$ is a super-solution of the problem satisfied by $w_{i,l,R}(t,x)$ for $0\le t\le T_{\varepsilon}$ and~$x\in\overline\Omega$. It then can be inferred from the comparison principle that $w_{i,l,R_{\varepsilon}}(t,x)\le \overline{w}(t,x)$ for all $0\le t\le T_{\varepsilon}$ and $x\in\overline\Omega$. Finally, take any $0\le t\le T_{\varepsilon}$ and $x\in\overline{\mathcal{H}_i}$ such that~$|x\cdot e_i-l|\le R-R_{\varepsilon}- (c_f+\varepsilon)t$. Notice that $x\cdot e_i\ge l-R+R_\epsilon\ge L_\epsilon+R_\epsilon>L$, hence $x\in\overline{\Omega}$. Furthermore, $\overline{\zeta}(t,x)\ge -|x\cdot e_i-l|-h_{\varepsilon}(0)-(c_f+\varepsilon)t-\omega+R-C\ge R_\epsilon-h_\epsilon(0)-\omega-C\ge C_{\varepsilon}$ and $\phi(\overline{\zeta}(t,x))\le\delta_{\varepsilon}$. It then follows from $\delta_{\varepsilon}\le \delta/2$ that
$$w_{i,l,R_{\varepsilon}}(t,x)\le \overline{w}(t,x)\le \delta_{\varepsilon}+\delta e^{-\delta t}+\delta_{\varepsilon}\le 2\delta.$$
The proof of Lemma~\ref{lemma5.2} is thereby complete.
\end{proof}
\vskip 0.3cm

The last lemma of this section is devoted to the proof of some upper estimates for contracting solutions which are initially close to $0$ in a large central region and in some parts of all branches. For any $R>L$, let $\widetilde{w}_R(t,x)$ denote the solution of the Cauchy problem
\be\label{defwtildeR}\left\{\baa{lll}
(\widetilde{w}_R)_t-\Delta \widetilde{w}_R=f(\widetilde{w}_R), &t>0,\ x\in\overline{\Omega},\vspace{3pt}\\
(\widetilde{w}_R)_\nu=0, &t>0,\ x\in\partial\Omega\eaa\right.
\ee
with initial condition
$$\widetilde{w}_R(0,x)=\delta\hbox{ for }x\in\overline{\Omega}\cap\Big(\overline{B(0,L)}\cup\bigcup_{i=1}^m\big\{x\in\overline{\mathcal{H}_i}: x\cdot e_i<R\big\}\Big)\hbox{ and }\widetilde{w}_R(0,x)=1\hbox{ elsewhere in }\overline{\Omega}.$$

\begin{lemma}\label{lemma5.3+}
For any $\varepsilon\in(0,c_f)$, there exists $R_{\varepsilon}>0$ such that, for any $R\ge R_{\varepsilon}+L$, there holds
\begin{align*}
\widetilde{w}_R(t,x)\le 3\delta&\ \text{ for all $0\le t\le T_{\varepsilon}=\frac{R-R_{\varepsilon}-L}{c_f+\varepsilon}$}\vspace{3pt}\\
&\ \ \text{and $x\in\overline{\Omega}\cap\Big(\overline{B(0,L)}\cup\bigcup_{i=1}^m\big\{x\in\overline{\mathcal{H}_i}: x\cdot e_i\le R-R_{\varepsilon}-(c_f+\varepsilon) t\big\}\Big)$}.
\end{align*}
\end{lemma}

\begin{proof}
Take any $\varepsilon\in (0,c_f)$. Given $\delta>0$ as in~\eqref{d}, let $C>0$, $k>0$, $\omega>0$, $\delta_\epsilon>0$ and $C_\epsilon\ge C$ be as in~\eqref{eq+2.2}-\eqref{ghsa1}. Consider a $C^2$ function $\hat{h}_{\varepsilon}:\R\rightarrow [0,1]$ such that
\be\label{e5.30}\left\{\baa{l}
\hat{h}_{\varepsilon}=1\ \text{in $(-\infty,C]$},\ \ \hat{h}_{\varepsilon}=0\ \text{in $[C+\xi_{\varepsilon}-1,+\infty)$, }\  -1\le \hat{h}'_{\varepsilon}\le 0\  \text{in $\R$},\vspace{3pt}\\
\displaystyle\delta (2c_f+\omega\delta)\|\hat{h}'_{\varepsilon}\|_{L^\infty(\R)}+\delta\|\hat{h}''_{\varepsilon}\|_{L^\infty(\R)}+2||\hat{h}'_{\varepsilon}\|_{L^\infty(\R)}\|\phi'\|_{L^\infty(\R)}\le \frac{|f'(0)|}{2}\delta_{\varepsilon}\eaa\right.
\ee
for some $\xi_{\varepsilon}>1$, and let
\be\label{defReps2}
R_{\varepsilon}=2C+\omega+\xi_{\varepsilon}.
\ee
Finally, pick any $R\ge R_{\varepsilon}+L$, and define $T_{\varepsilon}=(R-R_{\varepsilon}-L)/(c_f+\varepsilon).$

For all $t\in[0,T_\epsilon]$ and $x\in\overline{\Omega}$, let us set
\begin{eqnarray*}
\Phi(t,x)=\left\{\baa{lll}
\hat h_{\varepsilon}(\xi_i(t,x))\phi(\xi_i(t,x))+(1\!-\!\hat h_{\varepsilon}(\xi_i(t,x)))\delta &\text{if $x\in\overline{\mathcal{H}_i}$ with $x\cdot e_i>L$, $i=1,\cdots,m$},\vspace{3pt}\\
\delta &\text{otherwise}\eaa\right.
\end{eqnarray*}
and
$$\overline{w}(t,x)=\min\big(\Phi(t,x)+\delta_{\varepsilon}+\delta e^{-\delta t},1\big),$$
where
$$\xi_i(t,x)=-x\cdot e_i-(c_f+\varepsilon)t+\omega e^{-\delta t}-\omega +R-C.$$
We shall show that $\overline{w}$ is a super-solution of the problem satisfied by $\tilde w_R$ in $[0,T_{\varepsilon}]\times\overline{\Omega}$. Notice first that for any $i\in\{1,\cdots,m\}$, any $0\le t\le T_{\varepsilon}$ and any $x\in\overline{\mathcal{H}_i}\cap\overline{\Omega}$ with $x\cdot e_i\le L$, one has~$\xi_i(t,x)\ge -L-R+L+R_{\varepsilon}-\omega+R-C=C+\xi_{\varepsilon}>C+\xi_\epsilon-1$, hence $\hat{h}_\epsilon(\xi_i(t,x))=0$ and this equality also holds in a neighborhood of $(t,x)$ in $[0,T_\epsilon]\times\overline{\Omega}$. Owing to their definitions, the function $\Phi$ and $\overline{w}$ are then continuous in $[0,T_\epsilon]\times\overline{\Omega}$ and the function $\Phi$ is of class $C^2$. Furthermore,
$$\overline{w}(t,x)=\delta\!+\!\delta_{\varepsilon}\!+\!\delta e^{-\delta t}\le 3\delta\ \text{ for all $t\!\in\![0,T_{\varepsilon}]$ and $x\in\overline{\Omega}\cap\!\Big(\overline{B(0,L)}\cup\bigcup_{i=1}^m\big\{x\in\overline{\mathcal{H}_i}: x\cdot e_i\!\le\!L\big\}\Big)$}.$$

Let us now check the initial and boundary conditions. We have $\overline{w}(0,x)\ge \delta_{\varepsilon}+\delta\ge \widetilde{w}_R(0,x)$ for all $x\in\overline{\mathcal{H}_i}\cap\overline{\Omega}$ with $x\cdot e_i<R$ for some $i$, and for all $x\in\overline{\Omega}\cap\overline{B(0,L)}$. Furthermore, for all $x\in\overline{\mathcal{H}_i}$ with $x\cdot e_i\ge R$ for some $i$, one has $\xi_i(0,x)\le-C$, hence $\hat h_\epsilon(\xi_i(0,t))=1$, $\phi(\xi_i(t,x))\ge 1-\delta$, and $\overline{w}(0,x)\ge \min(1-\delta+\delta_{\varepsilon}+\delta,1)=1\ge \widetilde{w}_R(0,x)$. As a result, $\overline{w}(0,\cdot)\ge \widetilde{w}_R(0,\cdot)$ in $\overline\Omega$. On the other hand, since $\overline{w}(t,x)=\delta+\delta_{\varepsilon}+\delta e^{-\delta t}$ for all $0\le t\le T_{\varepsilon}$ and $x\in\overline{\Omega}\cap\big(\overline{B(0,L)}\cup\cup_{i=1}^m\big\{x\in\overline{\mathcal{H}_i}: x\cdot e_i\le L\big\}\big)$ and since each $\mathcal{H}_i$ is parallel to $e_i$, we have~$\overline{w}_\nu(t,x)=0$ for all $(t,x)\in[0,T_\epsilon]\times\partial \Omega$ such that $\overline{w}(t,x)<1$.

Let us finally check that
$$\mathcal{L}\overline{w}(t,x)=\overline{w}_t(t,x)-\Delta\overline{w}(t,x)-f(\overline{w}(t,x))\ge 0$$
for every $0\le t\le T_{\varepsilon}$ and $x\in\overline\Omega$ such that $\overline{w}(t,x)<1$. Pick any such $(t,x)$ in this paragraph. If $x\in\overline{B(0,L)}$ or if $x\in\overline{\mathcal{H}_i}$ with $x\cdot e_i\le L$ for some $i$, one has $\overline{w}(t,x)=\delta+\delta_\epsilon+\delta e^{-\delta t}$ (and this equality holds in a neighborhood of $(t,x)$ in $[0,T_\epsilon]\times\overline{\Omega}$) and it then follows from~\eqref{d} that~$f(\overline{w}(t,x))\le(f'(0)/2)(\delta+\delta_{\varepsilon}+\delta e^{-\delta t})$ and
$$\mathcal{L}\overline{w}(t,x)=-\delta^2 e^{-\delta t}-f(\overline{w}(t,x))\ge-\delta^2 e^{-\delta t}-\frac{f'(0)}{2}(\delta+\delta_{\varepsilon}+\delta e^{-\delta t})\ge \Big(\!\!-\delta-\frac{f'(0)}{2}\Big)\delta e^{-\delta t}\ge 0.$$
Assume now that $x\in\overline{\mathcal{H}_i}$ with $x\cdot e_i>L$ for some $i$. If $\xi_i(t,x)<-C$, one has $\hat{h}_{\varepsilon}(\xi_i(t,x))=1$, $\phi(\xi_i(t,x))\ge 1-\delta$ and $1>\overline{w}(t,x)=\phi(\xi_i(t,x))+\delta_{\varepsilon}+\delta e^{-\delta t}\ge 1-\delta$ (and these formulas hold in a neighborhood of $(t,x)$ in $[0,T_\epsilon]\times\overline{\Omega}$). Then from~\eqref{d} one gets that $f(\phi(\xi_i(t,x)))-f(\overline{w}(t,x))\ge -(f'(1)/2)(\delta_{\varepsilon}+\delta e^{-\delta t})\ge 0$ and a straightforward calculation gives
$$\baa{rcl}
\mathcal{L}\overline{w}(t,x) & = &-\varepsilon\phi'(\xi_i(t,x))-\omega\delta e^{-\delta t}\phi'(\xi_i(t,x))-\delta^2 e^{-\delta t}+f(\phi(\xi_i(t,x)))-f(\overline{w}(t,x))\vspace{3pt}\\
& \ge &-\varepsilon\phi'(\xi_i(t,x))-\omega\delta e^{-\delta t}\phi'(\xi_i(t,x))-\delta^2e^{-\delta t}-\frac{f'(1)}{2}(\delta_{\varepsilon}+\delta e^{-\delta t})\ge 0,\eaa$$
since $\phi'<0$, $f'(1)<0$ and $\delta<|f'(1)|/2$. If $-C\le \xi_i(t,x)<C$, one still has $\hat{h}_{\varepsilon}(\xi_i(t,x))=1$ and $\overline{w}(t,x)=\phi(\xi_i(t,x))+\delta_{\varepsilon}+\delta e^{-\delta t}$ (and these formulas hold in a neighborhood of $(t,x)$ in $[0,T_\epsilon]\times\overline{\Omega}$). Furthermore, $-\phi'(\xi_i(t,x))\ge k$ and $f(\phi(\xi_i(t,x)))-f(\overline{w}(t,x))\ge -\max_{[0,1]} |f'|(\delta_{\varepsilon}+\delta e^{-\delta t})$. Then from~\eqref{eq+2.4-} one has
\begin{align*}
\mathcal{L}\overline{w}(t,x)=& \ -\varepsilon\phi'(\xi_i(t,x))-\omega\delta e^{-\delta t}\phi'(\xi_i(t,x))-\delta^2 e^{-\delta t}+f(\phi(\xi_i(t,x)))-f(\overline{w}(t,x))\vspace{3pt}\\
\ge & \ \varepsilon k+\omega k \delta e^{-\delta t}-\delta^2 e^{-\delta t}-\max_{[0,1]} |f'|(\delta_{\varepsilon}+\delta e^{-\delta t})\ge 0.
\end{align*}
Lastly, if $\xi_i(t,x)\ge C$, one has $0<\phi(\xi_i(t,x))\le \delta$, hence $0<\phi(\xi_i(t,x))\le\Phi(t,x)\le \delta$ and $0<\phi(\xi_i(t,x)\le\Phi(t,x)<\overline{w}(t,x)=\Phi(t,x)+\delta_{\varepsilon}+\delta e^{-\delta t}\le 3\delta$. It follows from~\eqref{d} that~$f(\phi(\xi_i(t,x)))<0$ and, together with $\hat{h}_\epsilon\le1$,
$$\baa{rcl}
\hat{h}_\epsilon(\xi_i(t,x))f(\phi(\xi_i(t,x)))\!-\!f(\overline{w}(t,x)) & \!\!\ge\!\! & f(\phi(\xi_i(t,x)))-f(\overline{w}(t,x))\vspace{3pt}\\
& \!\!\ge\!\! & \displaystyle-\frac{f'(0)}{2}(1\!-\!\hat{h}_{\varepsilon}(\xi_i(t,x)))(\delta\!-\!\phi(\xi_i(t,x)))\!-\!\frac{f'(0)}{2}(\delta_{\varepsilon}\!+\!\delta e^{-\delta t})\vspace{3pt}\\
& \!\!\ge\!\! & \displaystyle-\frac{f'(0)}{2}(\delta_{\varepsilon}+\delta e^{-\delta t}).\eaa$$
Since $0<\varepsilon\le c_f$, $\hat{h}_\epsilon\ge0$, $\hat{h}'_{\varepsilon}\le0$, $\phi>0$ and $\phi'<0$ in $\R$, it then follows from~\eqref{d} and~\eqref{e5.30}, as in~\eqref{Lbarw}, that
$$\baa{rcl}
\mathcal{L}\overline{w}(t,x) & \!\!=\!\! &(c_f+\varepsilon+\omega\delta e^{-\delta t}) \hat h'_{\varepsilon}(\xi_i(t,x))(\delta-\phi(\xi_i(t,x)))-(\varepsilon+\omega\delta e^{-\delta t}) \hat h_{\varepsilon}(\xi_i(t,x))\phi'(\xi_i(t,x))\vspace{3pt}\\
& \!\!\!\! & +\,\hat h''_{\varepsilon}(\xi_i(t,x))(\delta-\phi(\xi_i(t,x)))-2\hat h'_{\varepsilon}(\xi_i(t,x))\phi'(\xi_i(t,x))-\delta^2 e^{-\delta t}\vspace{3pt}\\
& \!\!\!\! & +\,\hat h_{\varepsilon}(\xi_i(t,x))f(\phi(\xi_i(t,x)))-f(\overline{w}(t,x))\vspace{3pt}\\
& \!\!\ge\!\! & -\,\delta (2c_f+\omega\delta)|\hat h'_{\varepsilon}(\xi_i(t,x))|-|\hat h''_{\varepsilon}(\xi_i(t,x))|\delta-2|\hat h'_{\varepsilon}(\xi_i(t,x))||\phi'(\xi_i(t,x))|\vspace{3pt}\\
& \!\!\!\! & \displaystyle-\,\delta^2 e^{-\delta t} -\frac{f'(0)}{2}(\delta_{\varepsilon}+\delta e^{-\delta t})\vspace{3pt}\\
& \!\!\ge\!\! & 0.\eaa$$

As a conclusion, there holds $\overline{w}_t(t,x)-\Delta\overline{w}(t,x)-f(\overline{w}(t,x))\ge 0$ for all $(t,x)\in[0,T_{\varepsilon}]\times\overline\Omega$ with $\overline{w}(t,x)<1$. The comparison principle then yields $\widetilde{w}_R\le \overline{w}$ in $[0,T_{\varepsilon}]\times\overline\Omega$. Consider now any
$$t\in[0,T_\epsilon]\ \hbox{ and }\ x\in\overline{\Omega}\cap\Big(\overline{B(0,L)}\cup\bigcup_{i=1}^m\big\{x\in\overline{\mathcal{H}_i}: x\cdot e_i\le R-R_{\varepsilon}-(c_f+\varepsilon) t\big\}\Big).$$
On the one hand, if  $x\in\overline{B(0,L)}$ or if $x\in\overline{\mathcal{H}_i}$ with $x\cdot e_i\le L$ for some $i\in\{1,\cdots,m\}$, then $\widetilde{w}_R(t,x)\le \overline{w}(t,x)\le \delta+\delta_{\varepsilon}+\delta e^{-\delta t}\le 3\delta$. On the other hand, if $x\in\overline{\mathcal{H}_i}$ with $L<x\cdot e_i\le R-R_{\varepsilon}-(c_f+\varepsilon) t$ for some $i$, one has $\xi_i(t,x)\ge -R+R_{\varepsilon}-\omega+R-C=C+\xi_{\varepsilon}>C$ by~\eqref{defReps2}, hence $\phi(\xi_i(t,x))\le \delta$ and $\Phi(t,x)\le \delta$. Therefore, $\widetilde{w}_R(t,x)\le \overline{w}(t,x)\le \delta+\delta_{\varepsilon}+\delta e^{-\delta t}\le 3\delta$. The proof of Lemma~\ref{lemma5.3+} is thereby complete.
\end{proof}

\begin{remark}\label{remLReps}{\rm As in Remark~$\ref{remLReps1}$, it follows from their proofs that one can choose $L_\epsilon>L$ and $R_\epsilon>0$ in such a way that the conclusions of Lemmas~\ref{lemma3.3},~\ref{lemma5.2} and~\ref{lemma5.3+} hold simultaneously.}
\end{remark}


\subsection{Proof of Theorem \ref{th6}}\label{sec4.2}

Throughout this section, we assume that, for every $i\in\{1,\cdots,m\}$, the time-increasing front-like solution $u_i$ of~\eqref{frontlikei} propagates completely, in the sense that $u_i(t,x)\to1$ as $t\to+\infty$ locally uniformly in~$x\in\overline{\Omega}$. We consider any transition front $u$ connecting $0$ and $1$ for~\eqref{eq1.1}, and associated with some sets $(\Omega^\pm_t)_{t\in\R}$ and $(\Gamma_t)_{t\in\R}$. By comparing $u$ with some front-like solutions $u_i$, we first show that $u$ propagates completely as well, and we derive a result similar to Lemma~\ref{lemma2.6}, namely that the interfaces $\Gamma_t$ are located far away from the origin at very negative and very positive times.

\begin{lemma}\label{lemma5.6}
Under the above assumptions, the front $u$ propagates completely in the sense of~\eqref{complete} and, for every $\rho\ge 0$, there exist some real numbers $T_1<T_2$  such that
$$\Omega\cap B(0,L+\rho)\subset \Omega_t^-\ \ \text{for all $t\le T_1$,}\ \ \text{and}\ \ \Omega\cap B(0,L+\rho)\subset \Omega_t^+\ \ \text{for all $t\ge T_2$}.$$
\end{lemma}

\begin{proof}
First of all, as in the proof of Lemma~\ref{lemma3.2-}, one can assume from Definition~\ref{TF} without loss of generality, even if it means redefining $\Omega^\pm_t$ and $\Gamma_t$, that, for every $t\in\R$ and $i\in\{1,\cdots,m\}$, there is an non-negative integer $n_{i,t}\in\{0,\cdots,n\}$ and some real numbers $L<\xi_{i,t,1}<\cdots<\xi_{i,t,n_{i,t}}$ (if $n_{i,t}\ge1$) such that
\be\label{interfaces}
\Gamma_t\cap\mathcal{H}_i=\bigcup_{k=1}^{n_{i,t}}\big\{x\in\mathcal{H}_i: x\cdot e_i=\xi_{i,t,k}\big\},
\ee
where $n$ is as in~\eqref{eq1.6} and with the convention $\Gamma_t\cap\mathcal{H}_i=\emptyset$ if $n_{i,t}=0$. By~\eqref{eq1.3}, every $\Omega^+_t$ must then contain a half-infinite branch, that is, for every $t\in\R$, there exist $R_t>L$ and $i_t\in\{1,\cdots,m\}$ such that
$$\big\{x\in\mathcal{H}_{i_t}: x\cdot e_{i_t}\ge R_t\big\}\subset \Omega^+_t.$$
In particular, by denoting $i=i_0$, Definition~\ref{TF} implies that $u(0,x)\ge 1-\delta$ for all $x\in\overline{\mathcal{H}_i}$ with $x\cdot e_i\ge R_0+M_{\delta}$, where we recall that $\delta>0$ is given as in~\eqref{d}. Since $u_i(t,x)-\phi(-x\cdot e_i-c_f t)\to0$ as $t\to-\infty$ uniformly in $\overline{\mathcal{H}_i}\cap\overline{\Omega}$ and $u_i(t,x)\to0$ as $t\to-\infty$ uniformly in $\overline{\Omega}\setminus\overline{\mathcal{H}_i}$, there exists~$\tau\in\R$ such that $u_i(\tau,x)\le \delta$ for all $x\in\overline{\mathcal{H}_i}\cap\overline{\Omega}$ with $x\cdot e_i\le R_0+M_{\delta}$ and for all $x\in\overline{\Omega}\setminus\overline{\mathcal{H}_i}$.

Define now the function
$$\underline{u}(t,x)=\max\big(u_i(t+\tau+\omega e^{-\delta t}-\omega,x)-\delta e^{-\delta t},0\big)$$
and let us check that it is a sub-solution for $u$ in $[0,+\infty)\times\overline{\Omega}$, for some constant $\omega>0$ to be chosen. Observe first that $u(0,\cdot)\ge\underline{u}(0,\cdot)$ in $\overline{\Omega}$ and that $\underline{u}_\nu(t,x)=0$ for all $t\ge0$ and $x\in\partial\Omega$ such that $\underline{u}(t,x)>0$. Furthermore, for any fixed $(t,x)\in[0,+\infty)\times\overline{\Omega}$ with $\underline{u}(t,x)>0$, one has
$$\baa{rcl}
\mathcal{L}\underline{u}(t,x) & \!\!=\!\! & \underline{u}_t(t,x)-\Delta\underline{u}(t,x)-f(\underline{u}(t,x))\vspace{3pt}\\
& \!\!=\!\! & -\omega\delta(u_i)_t(t\!+\!\tau\!+\!\omega e^{-\delta t}\!-\!\omega,x)e^{-\delta t}+\delta^2e^{-\delta t}+f(u_i(t\!+\!\tau\!+\!\omega e^{-\delta t}\!-\!\omega,x))-f(\underline{u}(t,x)).\eaa$$
If $u_i(t+\tau+\omega e^{-\delta t}-\omega,x)>1-\delta$, then $1>u_i(t+\tau+\omega e^{-\delta t}-\omega,x)>\underline{u}(t,x)\ge1-2\delta$ and~$f(u_i(t+\tau+\omega e^{-\delta t}-\omega,x))-f(\underline{u}(t,x))\le(f'(1)/2)\delta e^{-\delta t}$, hence $\mathcal{L}\underline{u}(t,x)\le\delta(\delta+f'(1)/2)e^{-\delta t}\le0$ since $(u_i)_t\ge0$ and $\delta\le|f'(1)|/2=-f'(1)/2$. Similarly, if $u_i(t+\tau+\omega e^{-\delta t}-\omega,x)<\delta$, then~$\underline{u}(t,x)<u_i(t+\tau+\omega e^{-\delta t}-\omega,x)\le\delta$ and $f(u_i(t+\tau+\omega e^{-\delta t}-\omega,x))-f(\underline{u}(t,x))\le(f'(0)/2)\delta e^{-\delta t}$, hence $\mathcal{L}\underline{u}(t,x)\le\delta(\delta+f'(0)/2)e^{-\delta t}\le0$. Call now
$$\kappa=\inf_{(s,y)\in\R\times\overline{\Omega},\,\delta\le u_i(s,y)\le1-\delta}(u_i)_t(s,y),$$
which is a positive real number by Lemma~\ref{lemut} applied to $u_i$ (with $I\!=\!\{i\}$ and~$J\!=\!\{1,\cdots,m\}\setminus\{i\}$), and choose $\omega>0$ such that
$$\kappa\,\omega\ge\delta+\max_{[0,1]}|f'|.$$
For that constant $\omega$, if follows that, if $\delta\le u_i(t+\tau+\omega e^{-\delta t}-\omega,x)\le1-\delta$, then~$(u_i)_t(t+\tau+\omega e^{-\delta t}-\omega,x)\ge\kappa$ and $\mathcal{L}\underline{u}(t,x)\le(-\kappa\omega+\delta+\max_{[0,1]}|f'|)\delta e^{-\delta t}\le0$. As a conclusion,~$\mathcal{L}\underline{u}(t,x)\le0$ for all $(t,x)\in[0,+\infty)\times\overline{\Omega}$ such that $\underline{u}(t,x)>0$. One then deduces from the comparison principle that
$$u(t,x)\ge \underline{u}(t,x)\ge u_i(t+\tau+\omega e^{-\delta t}-\omega,x)-\delta e^{-\delta t}\ \text{ for all $t\geq 0$ and $x\in\overline\Omega$}.$$
Since $u_i(t,x)\rightarrow 1$ locally uniformly in $x\in\overline\Omega$ as $t\rightarrow +\infty$ by assumption, one gets that $u(t,x)\to1$ as $t\to+\infty$ locally uniformly in $x\in\overline{\Omega}$, that is, $u$ propagates completely. Furthermore, as in the proof of Lemma~\ref{lemma2.6}, this implies that, for any $\rho\ge0$, there is $T_2\in\R$ such that
$$\Omega\cap B(0,L+\rho)\subset \Omega_t^+\ \text{ for all $t\ge T_2$}.$$

We turn to consider the first conclusion of Lemma~\ref{lemma5.6}. It is sufficient to show the assertion for any $\rho\geq 0$ large enough, and then it will hold automatically for any $\rho\geq 0$. So, consider any~$\epsilon\in(0,c_f]$, let $L_\epsilon\ge L$ and $R_\epsilon>0$ be as in Lemma~\ref{lemma3.3}, and let $\rho$ be any large positive number such that
\be\label{defrho2}
\bigcup_{i=1}^m\big\{y\in\overline{\mathcal{H}_i}: L_{\varepsilon}+R_{\varepsilon}\le y\cdot  e_i\le L_{\varepsilon}+3R_{\varepsilon}+M_{\delta}\big\}\ \subset\ \overline{\Omega}\cap B(0,L+\rho).
\ee
Consider any such $\rho$ and assume now that the first conclusion of Lemma~\ref{lemma5.6} is false for that~$\rho$. Then, from \eqref{eq1.3}, two cases may occur: either there is a sequence $(t_n)_{n\in\mathbb N}\rightarrow -\infty$ such that~$\Omega\cap B(0,L+\rho)\cap \Gamma_{t_n}\neq \emptyset$ for each $n\in\mathbb N$, or there is a sequence $(t_n)_{n\in\mathbb N}\rightarrow -\infty$ such that~$\Omega\cap B(0,L+\rho)\subset \Omega_{t_n}^+$ for each $n\in\mathbb N$.

{\it Case~1: $\Omega\cap B(0,L+\rho)\cap \Gamma_{t_n}\neq \emptyset$ for each $n\in\mathbb N$.} For each $n\in\mathbb N$, pick a point~$x_n\in \Omega\cap B(0,L+\rho)\cap\Gamma_{t_n}$. By \eqref{eq1.4}-\eqref{eq1.5}, for any $R>0$ there is $r>0$ such that, for each~$n\in\mathbb N$, there exists $y_n\in\Omega^+_{t_n}$ with $d_{\Omega}(x_n,y_n)\le r$ and $d_\Omega(y_n,x_n)\ge d_{\Omega}(y_n,\Gamma_{t_n})\ge M_{\delta}+R$. Thus, up to extraction of a subsequence and by taking $R\geq R_{\varepsilon}$ large enough independently of $n$, there is $i\in\{1,\cdots,m\}$ such that, for each $n\in\N$, $y_n\in\overline{\mathcal{H}_i}$, $y_n\cdot e_i\ge R_{\varepsilon}+L_{\varepsilon}$,~$E_n:=\big\{y\in\mathcal{H}_i: |y\cdot e_i-y_n\cdot e_i|\le R_{\varepsilon}\big\}\subset \Omega^+_{t_n}$ and $d_\Omega(E_n,\Gamma_{t_n})\ge M_\delta$ (notice that $i$ can be chosen independently of $n$ up to extraction of a subsequence, since the number $m$ of branches is finite). This implies that $u(t_n,y)\ge 1-\delta$ for all $y\in\overline{\mathcal{H}_i}$ such that~$|y\cdot e_i-y_n\cdot e_i|\le R_{\varepsilon}$. It then follows from the comparison principle and Lemma~\ref{lemma3.3} that, for every $n\in\N$ and $t\ge t_n+(y_n\cdot e_i-R_\epsilon-L_\epsilon)/(c_f-\epsilon)$,
$$u(t,y)\ge v_{i,y_n\cdot e_i,R_{\varepsilon}}(t-t_n,y)\ge 1-3\delta$$
for all $y\in\overline{\mathcal{H}_i}$ such that $L_{\varepsilon}+R_{\varepsilon}\le y\cdot e_i\le y_n\cdot e_i+(c_f-\varepsilon)(t-t_n)$. Since the sequences $(x_n)_{n\in\N}$ and then $(y_n)_{n\in\N}$ are bounded, passing to the limit as $n\to+\infty$ in the previous inequality and using $\lim_{n\to+\infty}t_n=-\infty$ leads to
$$u(t,y)\ge 1-3\delta\ \ \text{for all $t\in\R$ and $y\in\overline{\mathcal{H}_i}$ such that $y\cdot e_i\ge L_{\varepsilon}+R_{\varepsilon}$}.$$

Take now any sequence $(t'_k)_{k\in \mathbb{N}}$ such that $t'_k\rightarrow -\infty$ as $k\rightarrow+\infty$. Then $u(t'_k,y)\ge 1-3\delta$ for all~$y\in\overline{\mathcal{H}_i}$ such that $y\cdot e_i\ge L_{\varepsilon}+R_{\varepsilon}$. By a similar argument to that of the first assertion, we can show the existence of $\tau'\in\R$ and $\omega'>0$ such that, for every $k\in\N$, the function
$$\underline{u}(t,x)=\max\left(u_i(t+\tau'+\omega'e^{-\delta t}-\omega',x)-3\delta e^{-\delta t},0\right)$$
is a sub-solution of the equation satisfied by $u(t'_k+\cdot,\cdot)$ in $[0,+\infty)\times\overline{\Omega}$. By the comparison principle one infers that, for every $k\in\N$,
$$u(t,x)\ge \underline{u}(t-t'_k,x)\ge u_i(t-t'_k+\tau'+\omega'e^{-\delta (t-t'_k)}-\omega',x)-3\delta e^{-\delta (t-t'_k)}\ \text{for all $t\ge t'_k$ and $x\in\overline\Omega$}.$$
Since $t'_k\rightarrow -\infty$ as $k\rightarrow +\infty$ and $u_i(t,x)\rightarrow 1$ as $t\rightarrow +\infty$ locally uniformly in $x\in \overline\Omega$, one gets that $u(t,x)\ge 1$ for every $(t,x)\in\R\times\overline{\Omega}$, which is a contradiction.

{\it Case~2: $\Omega\cap B(0,L+\rho)\subset \Omega_{t_n}^+$ for each $n\in\mathbb N$.} Owing to the property~\eqref{defrho2} satisfied by~$\rho$, it follows that, for each $n\in\N$ and $i\in\{1,\cdots,m\}$, there holds $u(t_n,y)\ge 1-\delta$ for all $y\in\overline{\mathcal{H}_i}$ such that $L_{\varepsilon}+R_{\varepsilon}\le y\cdot  e_i\le L_{\varepsilon}+3R_{\varepsilon}$. Following the same arguments as in case~1, one can reach a contradiction. The proof of Lemma~\ref{lemma5.6} is thereby complete.
\end{proof}
\vskip 0.3cm

We are now ready to carry out the proof of Theorem~\ref{th6}. Roughly speaking, we show that, for any small $\epsilon>0$, the large regions where $u$ is close to $1$ expand with speed at least $c_f-3\epsilon$ in each branch and also from the center of the domain to the other branches, whereas the large regions where $u$ is close to $0$ retract with speed at most $c_f+3\epsilon$ in each branch and also from any branch to the center of the domain.
\vskip 0.3cm

\begin{proof}[Proof of Theorem \ref{th6}]
Our goal is to prove that ${d_\Omega(\Gamma_t,\Gamma_s)}/{|t-s|}\rightarrow c_f$ as $|t-s|\rightarrow +\infty$. We shall first show that
\be\label{5.29}
\liminf_{|t-s|\rightarrow +\infty} \frac{d_\Omega(\Gamma_t,\Gamma_s)}{|t-s|}\ge c_f
\ee
and then
\be\label{5.30}
\limsup_{|t-s|\rightarrow +\infty} \frac{d_\Omega(\Gamma_t,\Gamma_s)}{|t-s|}\le c_f.
\ee
Throughout the proof, $\delta\in(0,1/4)$ is given as in~\eqref{d}.
\vskip 0.3cm

{\it Step 1: some notations.} First of all, as in the beginning of the proof of Lemma~\ref{lemma5.6}, one can assume without loss of generality, even if it means redefining $\Omega^\pm_t$ and $\Gamma_t$, that, for every $t\in\R$ and $i\in\{1,\cdots,m\}$, there are a non-negative integer $n_{i,t}\in\{0,\cdots,n\}$ and some real numbers $\xi_{i,t,1}<\cdots<\xi_{i,t,n_{i,t}}$ (if $n_{i,t}\ge1$) such that~\eqref{interfaces} holds.

Consider now any $\varepsilon\in (0,c_f/3)$. Let $L_{\varepsilon}>L$ and $R_{\varepsilon}>0$ be such that the conclusions of Lemmas~\ref{lemma3.3},~\ref{lemma5.2} and~\ref{lemma5.3+} hold (see Remark~\ref{remLReps}), and let $\overline{R}_{\varepsilon}>0$ be such that
\be\label{eq5.28}
\big\{z\in\overline{\mathcal{H}_i}\cap\overline{\Omega}: |z\cdot e_i-y\cdot e_i|\le R_{\varepsilon}\big\}\subset\big\{z\in\overline{\mathcal{H}_i}\cap\overline{\Omega}: d_{\Omega}(z,y)\le \overline{R}_{\varepsilon}\big\}
\ee
for every $i\in\{1,\cdots,m\}$ and every $y\in\overline{\mathcal{H}_i}\cap\overline{\Omega}$ with $y\cdot e_i\ge R_{\varepsilon}$. Let also $D_{\varepsilon}>L$ be such that
\be\label{eq5.29}
\forall\,i\in\{1,\cdots,m\},\ \forall\,x\in\overline{\mathcal{H}_i}\setminus B(0,D_{\varepsilon}),\ \ x\cdot e_i> L_{\varepsilon}+3R_{\varepsilon}+M_{\delta}+r_{M_{\delta}+ \overline{R}_{\varepsilon}},
\ee
where $r_{M_\delta+\overline{R_\epsilon}}>0$ is given in~\eqref{eq1.5} with $M:=M_\delta+\overline{R_\epsilon}$. Lemmaõ \ref{lemma5.6} then yields the existence of two real numbers $T^1_{\varepsilon}<T^2_{\varepsilon}$ such that
\be\label{eq5.30}
\Omega\cap B(0,D_{\varepsilon})\subset \Omega_t^- \text{ for all $t\le T^1_{\varepsilon}$,\ \ and\ \ }\Omega\cap B(0,D_{\varepsilon})\subset \Omega_t^+ \text{ for all $t\ge T^2_{\varepsilon}$}.
\ee
It then follows from~\eqref{branches} and~\eqref{eq5.29}-\eqref{eq5.30} that, for every $t\ge T_\epsilon^2$,
$$E:=\Omega\cap\Big(B(0,L)\cup\bigcup_{i=1}^m\big\{y\in\mathcal{H}_i: y\cdot e_i\le L_{\varepsilon}+3R_{\varepsilon}+r_{M_{\delta}+\overline{R}_{\varepsilon}}\big\}\Big)\,\subset\,\Omega\cap B(0,D_\epsilon)\,\subset\,\Omega^+_{t}$$
and $d_\Omega(y,\Gamma_t)\ge M_\delta$ for all $y\in\overline{E}$, hence
\be\label{eq5.40-}
u(t,y)\!\ge\!1\!-\!\delta\ \text{for all $t\!\ge\!T^2_{\varepsilon}$ and $y\!\in\!\overline{\Omega}\!\cap\!\Big(\overline{B(0,L)}\cup\bigcup_{i=1}^m\big\{y\!\in\!\overline{\mathcal{H}_i}\!:\! y\cdot e_i\!\le\!L_{\varepsilon}\!+\!3R_{\varepsilon}\!+\!r_{M_\delta\!+\!\overline{R_\epsilon}}\big\}\!\Big)$}.
\ee
In particular, for every $i\in\{1,\cdots,m\}$, we have $u(t,y)\ge 1-\delta$ for all $t\ge T^2_{\varepsilon}$ and $y\in\overline{\mathcal{H}_i}$ such that~$L_{\varepsilon}+R_{\varepsilon}\le y\cdot e_i\le L_{\varepsilon}+3R_{\varepsilon}$. Lemma \ref{lemma3.3} applied with $l=R_\epsilon+L_\epsilon$ then implies that, for all~$\tau\ge R_{\varepsilon}/(c_f-\varepsilon)$ and $t\ge T_\epsilon^2$, $u(t+\tau,y)\ge 1-3\delta$ for all $y\in\overline{\mathcal{H}_i}$ such that $L_{\varepsilon}+R_{\varepsilon}\le y\cdot e_i\le L_{\varepsilon}+2R_{\varepsilon} +(c_f-\varepsilon) \tau$. This together with \eqref{eq5.40-} yields 
\be\label{eq5.41-}\baa{ll}
u(s,y)\ge 1\!-\!3\delta & \displaystyle\text{for all $s$, $t\ge T^2_{\varepsilon}$ with $s-t\ge\frac{R_{\varepsilon}}{c_f-\varepsilon}$ and}\vspace{3pt}\\
& \displaystyle\text{$y\in\overline{\Omega}\cap\Big(\overline{B(0,L)}\cup\bigcup_{i=1}^m\big\{y\in\overline{\mathcal{H}_i}: y\cdot e_i\!\le\!L_{\varepsilon}\!+\!2R_{\varepsilon}\!+\!(c_f\!-\!\epsilon)(s\!-\!t)\big\}\!\Big)$.}\eaa
\ee

Furthermore, it also follows from~\eqref{interfaces} and~\eqref{eq5.29}-\eqref{eq5.30} that
\be\label{xiit1}
\xi_{i,t,1}>L_\epsilon+3R_\epsilon+M_\delta+r_{M_\delta+\overline{R_\epsilon}}\ \hbox{ for all }i\in\{1,\cdots,m\}\hbox{ and }t\in(-\infty,T^1_\epsilon]\cup[T^2_\epsilon,+\infty),
\ee
with the convention $\xi_{i,t,1}=+\infty$ and $\Gamma_t\cap\mathcal{H}_i=\emptyset$ if $n_{i,t}=0$. 
\vskip 0.2cm

{\it Step 2: the lower estimate \eqref{5.29}.} In order to show~\eqref{5.29}, we first claim that
\be\label{5.34}
\liminf_{t<s\le T^1_{\varepsilon},\ |t-s|\rightarrow +\infty} \frac{d_{\Omega}(\Gamma_{t},\Gamma_{s})}{|t-s|}\ge c_f-2\varepsilon.
\ee
Assume by contradiction that there exist two sequences $(t_k)_{k\in\mathbb{N}}$ and $(s_k)_{k\in\mathbb{N}}$ such that~$t_k\!<\!s_k\!\le\!T^1_{\varepsilon}$ for all $k\in\N$, $s_k-t_k\rightarrow +\infty$ as $k\rightarrow +\infty$, and $d_{\Omega}(\Gamma_{t_k},\Gamma_{s_k})<(c_f-2\varepsilon)(s_k-t_k)$ for all $k\in\mathbb{N}$. By definition of the distance, there exist two sequences $(x_k)_{k\in\mathbb N}$ and $(z_k)_{k\in\mathbb N}$ such that $x_k\in\Gamma_{t_k}$, $z_k\in\Gamma_{s_k}$ and $d_{\Omega}(x_k,z_k)<(c_f-2\varepsilon)(s_k-t_k)$ for all $k\in\mathbb{N}$. It follows from~\eqref{branches} and~\eqref{eq5.29}-\eqref{eq5.30} that, for each $k\in\mathbb N$, there is $i_k\in\{1,\cdots,m\}$ such that $x_k\in\Gamma_{t_k}\cap \mathcal{H}_{i_k}$ and
\be\label{eq5.31}
x_k\cdot e_{i_k}>L_{\varepsilon}+3R_{\varepsilon}+M_{\delta}+r_{M_{\delta}+ \overline{R}_{\varepsilon}}.
\ee
Furthermore, by~\eqref{eq1.4}-\eqref{eq1.5}, for each $k\in\mathbb N$, there is $y_k\in\Omega^+_{t_k}$ such that
\be\label{eq5.32}
d_{\Omega}(x_k,y_k)\le r_{M_{\delta}+ \overline{R}_{\varepsilon}}\ \ \text{and}\ \ d_{\Omega}(y_k,\Gamma_{t_k})\ge M_{\delta}+\overline{R}_{\varepsilon}.
\ee
It follows from \eqref{eq5.31}-\eqref{eq5.32} that
\be\label{eq5.32+}
y_k\in\mathcal{H}_{i_k}\ \ \text{and}\ \ y_k\cdot e_{i_k}\ge L_{\varepsilon}+3R_{\varepsilon}+M_{\delta}.
\ee
One then gets from~\eqref{branches},~\eqref{eq5.28} and~\eqref{eq5.32}-\eqref{eq5.32+} that $E_k\!:=\!\big\{y\in\mathcal{H}_{i_k}\!:\! |y\cdot e_{i_k}\!-\!y_k\cdot e_{i_k}|\!\le\!R_{\varepsilon}\big\}\!\subset\!\Omega^+_{t_k}$ and $d_{\Omega}\big(\overline{E_k},\Gamma_{t_k}\big)\ge M_{\delta}$. Thus $u(t_k,y)\ge 1-\delta$ for all $y\in\overline{\mathcal{H}_{i_k}}$ such that $|y\cdot e_{i_k}-y_k\cdot e_{i_k}|\le R_{\varepsilon}$. Moreover, $y_k\cdot e_{i_k}\ge R_{\varepsilon}+L_{\varepsilon}$ by~\eqref{eq5.32+} and the comparison principle implies that $u(t,y)\ge v_{i_k,y_k\cdot e_{i_k},R_{\varepsilon}}(t-t_k,y)$ for all $t\ge t_k$ and $y\in\overline{\Omega}$, with the notation~\eqref{3.2}. Lemma~\ref{lemma3.3} then yields
\be\label{eq5.33}\left\{\baa{ll}
u(t,y)\ge 1-2\delta & \displaystyle\text{for all $0\le t-t_k\le T_{\varepsilon,k}=\frac{y_k\cdot e_{i_k}-R_{\varepsilon}-L_{\varepsilon}}{c_f-\varepsilon}$}\vspace{3pt}\\
& \text{and $y\in\overline{\mathcal{H}_{i_k}}$ with $|y\cdot e_{i_k}-y_k\cdot e_{i_k}|\le (c_f-\varepsilon)(t-t_k)$},\vspace{3pt}\\
u(t,y)\ge 1-3\delta & \text{for all $t-t_k\ge T_{\varepsilon,k}$}\vspace{3pt}\\
& \text{and $y\in\overline{\mathcal{H}_{i_k}}$ with $R_{\varepsilon}+L_{\varepsilon}\le y\cdot e_{i_k}\le y_k\cdot e_{i_k} +(c_f-\varepsilon)(t-t_k)$}\eaa\right.
\ee
(notice that the points $y$ considered in the first line automatically belong to $\overline{\mathcal{H}_{i_k}}\cap\overline{\Omega}$ since $y\cdot e_{i_k}\ge y_k\cdot e_{i_k}-(c_f-\epsilon)T_{\epsilon,k}=R_\epsilon+L_\epsilon>L$).

If $T_{\varepsilon}^1-t_k\ge T_{\varepsilon,k}$ for some $k\in\mathbb N$, it follows from~\eqref{eq5.33} that
\be\label{eq5.35}
u(t_k+T_{\varepsilon,k},y)\ge 1-2\delta\ \ \text{ for all $y\in\overline{\mathcal{H}_{i_k}}$ with $L_{\varepsilon}+R_{\varepsilon}\le y\cdot e_{i_k}\le y_k\cdot e_{i_k}+ (c_f-\varepsilon)T_{\varepsilon,k}$}.
\ee
Since $T_{\varepsilon,k}=(y_k\cdot e_{i_k}-R_{\varepsilon}-L_{\varepsilon})/(c_f-\varepsilon)$, one infers from~\eqref{eq5.32+} and~\eqref{eq5.35} the existence of a point $\tilde y_k\in\mathcal{H}_{i_k}$ such that $L_{\varepsilon}+R_{\varepsilon}\le \tilde y_k\cdot e_{i_k}\leq L_{\varepsilon}+2R_{\varepsilon}$ and $u(t_k+T_{\varepsilon,k},\tilde y_k)\ge 1-2\delta$. It follows  from \eqref{eq5.29}-\eqref{eq5.30} that $\tilde{y}_k\in\Omega^-_{t_k+T_{\epsilon,k}}$ and $d_\Omega(\tilde y_k,\Gamma_{t_k+T_{\varepsilon,k}})\geq M_\delta$, hence $u(t_k+T_{\varepsilon,k},\tilde y_k)\le \delta$. This contradicts $u(t_k+T_{\varepsilon,k},\tilde y_k)\ge 1-2\delta$.

Therefore, there holds $s_k-t_k\le T^1_{\varepsilon}-t_k<T_{\varepsilon,k}$ for every $k\in\mathbb N$. It follows from~\eqref{eq5.33} that
\be\label{eq5.36}
u(s_k,y)\ge 1-2\delta\ \ \text{for all $y\in\overline{\mathcal{H}_{i_k}}$ such that $|y\cdot e_{i_k}-y_k\cdot e_{i_k}|\le (c_f-\varepsilon)(s_k-t_k)$},
\ee
and these points $y$ belong to $\overline{\mathcal{H}_{i_k}}\cap\overline{\Omega}$ since $y\cdot e_{i_k}\ge y_k\cdot e_{i_k}-(c_f-\varepsilon)(s_k-t_k)\ge L_{\varepsilon}+R_{\varepsilon}>L$ due to $s_k-t_k\le T_{\varepsilon,k}=(y_k\cdot e_{i_k}-L_{\varepsilon}-R_{\varepsilon})/(c_f-\varepsilon)$. For each $k\in\N$, since $z_k\in\Gamma_{s_k}$, one infers from~\eqref{eq1.4}-\eqref{eq1.5} the existence of $y'_k\in\Omega_{s_k}^-$ such that $d_{\Omega}(z_k,y'_k)\le r_{M_{\delta}}$ and $d_{\Omega}(y'_k,\Gamma_{s_k})\ge M_{\delta}$, hence
\be\label{eq:5.46}
u(s_k,y'_k)\le \delta.
\ee
Since $d_{\Omega}(x_k,y_k)\le r_{M_{\delta}+\overline{R}_{\varepsilon}}$ and $d_{\Omega}(x_k,z_k)\le (c_f-2\varepsilon)(s_k-t_k)$, one has
$$d_{\Omega}(y_k,y'_k)\le r_{M_{\delta}+\overline{R}_{\varepsilon}}+(c_f-2\varepsilon)(s_k-t_k)+r_{M_{\delta}}\le (c_f-\varepsilon)(s_k-t_k)\ \ \text{for $k$ large enough.}$$
Then \eqref{eq5.36} implies that $y'_k\in\mathcal{H}_{i_k}$ and $u(s_k,y'_k)\ge 1-2\delta$ for $k$ large enough. This contradicts~\eqref{eq:5.46}. Hence the claim~\eqref{5.34} has been proved.
\vskip 0.2cm

Secondly, we claim that
\be\label{5.42}
\liminf_{T^2_{\varepsilon}\le t<s,\ |t-s|\rightarrow +\infty} \frac{d_{\Omega}(\Gamma_{t},\Gamma_{s})}{|t-s|}\ge c_f-2\varepsilon.
\ee
Assume by contradiction that there exist two sequences $(t_k)_{k\in\mathbb{N}}$ and $(s_k)_{k\in\mathbb{N}}$ such that $T^2_{\varepsilon}\le t_k<s_k$, $s_k-t_k\rightarrow +\infty$ as $k\rightarrow +\infty$ and $d_\Omega(\Gamma_{t_k},\Gamma_{s_k})<(c_f-2\varepsilon)(s_k-t_k)$ for all~$k\in\mathbb{N}$. By definition of the distance, there exist two sequences $(x_k)_{k\in\mathbb N}$ and $(z_k)_{k\in\mathbb N}$ such that $x_k\in\Gamma_{t_k}$, $z_k\in\Gamma_{s_k}$ and $d_{\Omega}(x_k,z_k)<(c_f-2\varepsilon)(s_k-t_k)$ for all $k\in\mathbb{N}$. As in the proof of~\eqref{5.34}, properties~\eqref{eq5.29}-\eqref{eq5.30} imply that, for each $k\in\N$, there is $i_k\in\{1,\cdots,m\}$ satisfying $x_k\in\Gamma_{t_k}\cap \mathcal{H}_{i_k}$ and~\eqref{eq5.31}, and then by \eqref{eq1.4}-\eqref{eq1.5} there is $y_k\in\Omega^+_{t_k}$ satisfying~\eqref{eq5.32}. As above,~\eqref{eq5.33} then holds.

On the other hand, for each $k\in\N$, since $z_k\in\Gamma_{s_k}$, it follows from~\eqref{eq1.4}-\eqref{eq1.5} that there exists~$y'_k\in\Omega_{s_k}^-$ such that $d_{\Omega}(z_k,y'_k)\le r_{M_{\delta}}$ and $d_{\Omega}(y'_k,\Gamma_{s_k})\ge M_{\delta}$. This further implies that
\be\label{eq:5.57}
u(s_k,y'_k)\le \delta.
\ee
Moreover, since $x_k$, $y_k\in\mathcal{H}_{i_k}$, since $d_{\Omega}(x_k,y_k)\le r_{M_{\delta}+\overline{R_{\varepsilon}}}$ and since $d_{\Omega}(x_k,z_k)<(c_f-2\varepsilon)(s_k-t_k)$, one infers that
\be\label{eq.5.57+}
d_{\Omega}(y_k,y'_k)<r_{M_{\delta}+\overline{R}_{\varepsilon}}+(c_f-2\varepsilon)(s_k-t_k)+r_{M_{\delta}}\le\Big(c_f-\frac{3\varepsilon}{2}\Big)(s_k-t_k)<(c_f-\epsilon)(s_k-t_k)
\ee
for $k$ large enough. If $s_k-t_k\le T_{\varepsilon,k}=(y_k\cdot e_{i_k}-R_{\varepsilon}-L_{\varepsilon})/(c_f-\varepsilon)$ for a sequence of $k$'s, then
$$y'_k\cdot e_{i_k}>y_k\cdot e_{i_k}-(c_f-3\epsilon/2)(s_k-t_k)\ge R_\epsilon+L_\epsilon>L$$
and $y'_k\in\mathcal{H}_{i_k}$ for these large enough $k$'s, while~\eqref{eq5.33} implies that~$u(s_k,y)\ge1-2\delta$ for all $y\in\overline{\mathcal{H}_{i_k}}$ with $|y\cdot e_{i_k}-y_k\cdot e_{i_k}|\le (c_f-\varepsilon)(s_k-t_k)$, hence
\be\label{usky'k}
u(s_k,y'_k)\ge 1-2\delta.
\ee
This contradicts~\eqref{eq:5.57}. Therefore, $s_k-t_k>T_{\varepsilon,k}$ for all $k$ large enough and it follows from~\eqref{eq5.33} that $u(s_k,y)\ge 1-3\delta$ for all $y\in\overline{\mathcal{H}_{i_k}}$ such that $R_{\varepsilon}+L_{\varepsilon}\le y\cdot e_{i_k}\le y_k\cdot e_{i_k} +(c_f-\varepsilon)(s_k-t_k)$. By remembering that $y_k\in\mathcal{H}_{i_k}$ together with~\eqref{eq5.41-} and~\eqref{eq.5.57+}, one then gets that $u(s_k,y'_k)\ge 1-3\delta$ for $k$ large enough, whether $y'_k$ be in $\overline{\mathcal{H}_{i_k}}$ or in $\overline{B(0,L)}$ or in another $\overline{\mathcal{H}_j}$ with $j\neq i_k$. As a consequence, one has contradicted~\eqref{eq:5.57} and the proof of the claim~\eqref{5.42} is thereby complete.
\vskip 0.2cm

Thirdly, we prove that
\be\label{liminf4}
\liminf_{|t-s|\rightarrow +\infty} \frac{d_\Omega(\Gamma_t,\Gamma_s)}{|t-s|}\ge c_f-3\varepsilon.
\ee
Assume by contradiction that there exist two sequences $(t_k)_{k\in\mathbb{N}}$ and $(s_k)_{k\in\mathbb{N}}$ of real numbers such that $t_k<s_k$, $s_k-t_k\rightarrow +\infty$ as $k\rightarrow +\infty$ and
\be\label{eq:5.60}
d_{\Omega}(\Gamma_{t_k},\Gamma_{s_k})<(c_f-3\varepsilon)(s_k-t_k)\ \  \text{for all $k\in\mathbb{N}$}.
\ee
Then six cases may occur up to extraction of a subsequence, namely, case~1: $t_k<s_k\le T_{\varepsilon}^1$ for all $k\in\N$; case~2: $t_k< T_{\varepsilon}^1< s_k< T_{\varepsilon}^2$ for all $k\in\N$; case~3: $t_k< T_{\varepsilon}^1<T_{\varepsilon}^2\le s_k$ for all $k\in\N$; case~4: $ T_{\varepsilon}^1 \le t_k<s_k\le T_{\varepsilon}^2$ for all $k\in\N$; case~5: $T_{\varepsilon}^1\le t_k\le T_{\varepsilon}^2< s_k$ for all $k\in\N$; case~6: $T_{\varepsilon}^2< t_k<s_k$ for all $k\in\N$. In fact, cases~1 and~6 can not happen for large~$k$ by~\eqref{5.34} and~\eqref{5.42}. Furthermore, since $s_k-t_k\to+\infty$ as $k\to +\infty$, case~4 is ruled out immediately.

Now we consider case~2. In this case, one has $t_k\to-\infty$ as $k\to+\infty$ and it follows from~\eqref{5.34} that $d_{\Omega}(\Gamma_{t_k},\Gamma_{T_{\varepsilon}^1})\ge (c_f-5\varepsilon/2) (T_{\varepsilon}^1-t_k)$ for large $k$. Since $|s_k-T^1_\epsilon|\le T^2_\epsilon-T^1_\epsilon$ for all $k$, property~\eqref{eq+2.19} (which also holds similarly to~\cite[Lemma~3 and Remark~3]{GH} and~\cite[Propositions~1.2 and~4.2]{HR2} in domains with multiple branches) yields the existence of a constant $M>0$ such that $d_{\Omega}(\Gamma_{t_k},\Gamma_{s_k})\ge d_{\Omega}(\Gamma_{t_k},\Gamma_{T_{\varepsilon}^1})-M$ for all $k\in\N$, hence $d_{\Omega}(\Gamma_{t_k},\Gamma_{s_k})\ge(c_f-3\varepsilon) (s_k-t_k)$ for large~$k$, a contradiction with~\eqref{eq:5.60}. Similarly, one can reach a contradiction with~\eqref{eq:5.60} in case~5.

It remains to handle case~3. We assume without loss of generality that $t_k\rightarrow -\infty$ and~$s_k\rightarrow+\infty$ as $k\rightarrow+\infty$ (otherwise, by decreasing $T_{\varepsilon}^1$ and increasing $T_{\varepsilon}^2$, one can reduce case~3 to case~2 or case~5). It follows from \eqref{eq:5.60} that there are two sequences $(x_k)_{k\in\mathbb N}$ and~$(z_k)_{k\in\mathbb N}$ such that
\be\label{xkzk}
x_k\in\Gamma_{t_k},\ \ z_k\in\Gamma_{s_k}\ \hbox{ and }\ d_{\Omega}(x_k,z_k)<(c_f-3\varepsilon)(s_k-t_k)
\ee
for every $k\in\N$. Correspondingly, as in the proofs of~\eqref{5.34} and~\eqref{5.42}, there are two sequences~$(y_k)_{k\in\mathbb N}$ and $(y_k')_{k\in\mathbb N}$ such that $x_k\in\mathcal{H}_{i_k}$ for some $i_k\in\{1,\cdots,m\}$,
\be\label{yky'k}\left\{\baa{l}
y_k\in\Omega_{t_k}^+\cap \mathcal{H}_{i_k},\ \ d_{\Omega}(x_k,y_k)\le r_{M_{\delta}+\overline{R}_{\varepsilon}},\ \ d_{\Omega}(y_k,\Gamma_{t_k})\ge M_{\delta}+\overline{R}_{\varepsilon},\vspace{3pt}\\
y_k'\in\Omega_{s_k}^-,\ \ d_{\Omega}(z_k,y'_k)\le r_{M_{\delta}},\ \ d_{\Omega}(y'_k,\Gamma_{s_k})\ge M_{\delta},\ \ u(s_k,y'_k)\le\delta,\eaa\right.
\ee
and~\eqref{eq5.33} holds. As above, if $s_k-t_k\le T_{\varepsilon,k}=(y_k\cdot e_{i_k}-R_\epsilon-L_\epsilon)/(c_f-\epsilon)$ for a sequence of indices~$k$'s, one infers as in the proof of~\eqref{usky'k} that, for these large $k$'s, $y'_k\in\mathcal{H}_{i_k}$ and~$u(s_k,y'_k)\ge1-2\delta$,  contradicting the inequality $u(s_k,y'_k)\le\delta$ in~\eqref{yky'k}. Therefore,
$$s_k-t_k>T_{\varepsilon,k}=\frac{y_k\cdot e_{i_k}-R_\epsilon-L_\epsilon}{c_f-\epsilon},$$
for all $k$ large enough, and~\eqref{eq5.33} implies that
\be\label{usky}
u(s_k,y)\ge 1-3\delta\text{ for all $y\in\overline{\mathcal{H}_{i_k}}$ with $R_{\varepsilon}+L_{\varepsilon}\le y\cdot e_{i_k}\le y_k\cdot e_{i_k} +(c_f-\varepsilon)(s_k-t_k)$}.
\ee
Now, since $t_k\to-\infty$ and $s_k\to+\infty$ as $k\to+\infty$, Lemma~\ref{lemma5.6} together with~\eqref{xkzk}-\eqref{yky'k} implies that $\lim_{k\to+\infty}|x_k|=\lim_{k\to+\infty}|z_k|=+\infty$ and $\lim_{k\to+\infty}|y_k|=\lim_{k\to+\infty}|y'_k|=+\infty$. In particular, for all $k$ large enough, there is $j_k\in\{1,\cdots,m\}$ such that $y'_k\in\mathcal{H}_{j_k}$. In order to reach a contradiction and complete the proof, we shall consider two cases, up to extraction of a subsequence: either $j_k=i_k$ for all $k\in\N$, or $j_k\neq i_k$ for all $k\in\N$. Since
$$d_{\Omega}(y_k,y'_k)\le d_{\Omega}(y_k,x_k)+d_{\Omega}(x_k,z_k)+d_{\Omega}(z_k,y'_k)<r_{M_\delta+\overline{R_\epsilon}}+(c_f-3\epsilon)(s_k-t_k)+r_{M_\delta},$$
the former case yields
$$R_\epsilon+L_\epsilon\le y'_k\cdot e_{i_k}\le y_k\cdot e_{i_k}+(c_f-\epsilon)(s_k-t_k)$$
for large $k$ and $u(s_k,y'_k)\ge1-3\delta$ by~\eqref{usky}, contradicting the inequality $u(s_k,y'_k)\le\delta$ in~\eqref{yky'k}. Consider now the case where $j_k\neq i_k$ for all $k\in\N$ (that is, roughly spea\-king, the front goes from a branch $\mathcal{H}_{i_k}$ at time $t_k$ to another branch $\mathcal{H}_{j_k}$ at time $s_k$). Since $\lim_{k\to+\infty}|y_k|=\lim_{k\to+\infty}|y'_k|=+\infty$, it then follows from~\eqref{branches} that there is a constant $\Lambda>0$ such that
$$d_\Omega(y_k,y'_k)\ge y_k\cdot e_{i_k}+y'_k\cdot e_{j_k}-\Lambda\ \hbox{ for all }k\in\N.$$
Therefore,
\be\label{y'kejk}
y'_k\cdot e_{j_k}\le d_\Omega(y_k,y'_k)-y_k\cdot e_{i_k}+\Lambda<r_{M_\delta+\overline{R_\epsilon}}+(c_f-3\epsilon)(s_k-t_k)+r_{M_\delta}-y_k\cdot e_{i_k}+\Lambda
\ee
for all $k\in\N$. Let now $(X_i)_{1\le i\le m}$ be some points (independent of the indices~$k$) such that~$X_i\in\mathcal{H}_i$ and $X_i\cdot e_i=R_\epsilon+L_\epsilon$ for each $i\in\{1,\cdots,m\}$. It follows from~\eqref{eq5.29}-\eqref{eq5.30} that $X_i\in\Omega^-_{T^1_\epsilon}$ and $d_\Omega(X_i,\Gamma_{T^1_\epsilon})\ge M_\delta$, hence $u(T^1_\epsilon,X_i)\le\delta$ for each $i\in\{1,\cdots,m\}$. On the other hand, if~$T^1_\epsilon-t_k\ge T_{\epsilon,k}$ for some $k\in\N$, then~\eqref{eq5.33} implies that $u(T^1_\epsilon,X_{i_k})\ge1-3\delta$, leading to a contradiction. Therefore, $T^1_\epsilon-t_k<T_{\epsilon,k}=(y_k\cdot e_{i_k}-R_\epsilon-L_\epsilon)/(c_f-\epsilon)$ for all $k\in\N$. Together with~\eqref{y'kejk}, one infers that
$$\baa{rcl}
y'_k\cdot e_{j_k} & < & r_{M_\delta+\overline{R_\epsilon}}+(c_f-3\epsilon)(s_k-t_k)+r_{M_\delta}-R_\epsilon-L_\epsilon-(c_f-\epsilon)(T^1_\epsilon-t_k)+\Lambda\vspace{3pt}\\
& = & (c_f-3\epsilon)s_k+2\epsilon t_k+r_{M_\delta+\overline{R_\epsilon}}+r_{M_\delta}-R_\epsilon-L_\epsilon-(c_f-\epsilon)T^1_\epsilon+\Lambda\vspace{3pt}\\
& < & L_\epsilon+2R_\epsilon+(c_f-\epsilon)(s_k-T^2_\epsilon)\eaa $$
for all $k$ large enough, since $s_k\to+\infty$ and $t_k\to-\infty$. Finally,~\eqref{eq5.41-} then yields $u(s_k,y'_k)\ge1-3\delta$ for all $k$ large enough, contradicting the inequality $u(s_k,y'_k)\le\delta$ in~\eqref{yky'k}. Consequently, case~3 is ruled out too.

As a conclusion,~\eqref{liminf4} holds for any $\epsilon\in(0,c_f/3)$ and the proof of~\eqref{5.29} is thereby complete.
\vskip 0.3cm

{\it Step 3: the upper estimate \eqref{5.30}.} We first claim that
\be\label{eq:5.68-}
\limsup_{t<s\le T^1_{\varepsilon},\ |t-s|\rightarrow +\infty} \frac{d_\Omega(\Gamma_{t},\Gamma_{s})}{|t-s|}\le c_f+2\varepsilon.
\ee
Assume by contradiction that there exist two sequences $(t_k)_{k\in\mathbb{N}}$ and $(s_k)_{k\in\mathbb{N}}$ such that $t_k<s_k\le T^1_{\varepsilon}$, $s_k-t_k\rightarrow +\infty$ as $k\rightarrow +\infty$ and
\be\label{3.6}
d(\Gamma_{t_k},\Gamma_{s_k})>(c_f+2\varepsilon)(s_k-t_k)\ \ \text{for all $k\in\mathbb{N}$}.
\ee

It follows from~\eqref{branches} and~\eqref{eq5.30} that, for every $i\in\{1,\cdots,m\}$ and $k\in\N$, either $\mathcal{H}_i\cap \Gamma_{t_k}\neq \emptyset$ or~$\mathcal{H}_i\cap\Omega\subset \Omega^-_{t_k}$. Furthermore, for every $k\in\N$, there is at least one integer $i$ such that~$\mathcal{H}_i\cap\Gamma_{t_k}\neq\emptyset$ (otherwise, $\Omega^-_{t_k}$ would be the whole domain $\Omega$). Consider now some integers $k\in\N$ and $i\in\{1,\cdots,m\}$ such that
$$\Gamma_{t_k}\cap\mathcal{H}_i\neq \emptyset,$$
in other words $\xi_{i,t_k,1}\in\R$. Then there is $x_k\in\mathcal{H}_i\cap\Gamma_{t_k}$ such that $x_k\cdot e_i=\xi_{i,t_k,1}$ and by~\eqref{eq1.4}-\eqref{eq1.5} there is $y_k\in\Omega^+_{t_k}$ such that $d_{\Omega}(x_k,y_k)\le r_{M_{\delta}+\overline{R}_{\varepsilon}}$ and $d_{\Omega}(y_k,\Gamma_{t_k})\ge M_{\delta}+\overline{R}_{\varepsilon}$. As in Step~2, it follows that~\eqref{eq5.32+}-\eqref{eq5.33} and~\eqref{eq5.36} hold with $i=i_k$, together with~$s_k-t_k\le T^1_\epsilon-t_k<T_{\epsilon,k}=(y_k\cdot e_i-R_\epsilon-L_\epsilon)/(c_f-\epsilon)$. Since $d_{\Omega}(x_k,y_k)\le r_{M_{\delta}+\overline{R}_{\varepsilon}}$, property~\eqref{eq5.36} especially implies that
$$u(s_k,x)\ge 1-2\delta\ \hbox{ for all }x\in\overline{\mathcal{H}_i}\hbox{ with }\xi_{i,t_k,1}-r_{M_\delta}-M_{\delta}\le x\cdot e_i\le \xi_{i,t_k,1}+M_{\delta}+r_{M_\delta}$$
if $k$ is large enough (so that $(c_f-\epsilon)(s_k\!-\!t_k)\!>\!M_\delta\!+\!r_{M_\delta}\!+\!r_{M_\delta+\overline{R_\epsilon}}$), hence $\{x\in\mathcal{H}_i\!:\! x\cdot e_i\!=\!\xi_{i,t_k,1}\}\!\subset\!\Omega^+_{s_k}$. Using~\eqref{eq5.29}-\eqref{eq5.30}, we then get that
\be\label{eq:5.70-}
\Gamma_{s_k}\cap\big\{x\in\mathcal{H}_i: L_{\varepsilon}+3R_{\varepsilon}+M_\delta+r_{M_{\delta}+\overline{R_\varepsilon}}<x\cdot e_i<\xi_{i,t_k,1}\big\}\neq \emptyset
\ee
if $k$ is large enough. Together with~\eqref{3.6}, one infers that
\be\label{eq:5.70}
\xi_{i,t_k,1}-(c_f+2\varepsilon)(s_k-t_k)>L_{\varepsilon}+3R_{\varepsilon}+M_\delta+r_{M_{\delta}+\overline{R_\varepsilon}}
\ee
for every $i\in\{1,\cdots,m\}$ and $k$ (large enough) with $\mathcal{H}_i\cap\Gamma_{t_k}\neq \emptyset$. Furthermore, if $\mathcal{H}_i\cap\Gamma_{t_k}=\emptyset$ (that is, $n_{i,t_k}=0$), then~\eqref{eq:5.70} holds immediately since $\xi_{i,t_k,1}=+\infty$ in this case by convention. In other words,~\eqref{eq:5.70} holds for all $k$ large enough and for all $i\in\{1,\cdots,m\}$.

On the other hand one has $\Omega\cap\big(B(0,L)\cup\cup_{i=1}^m\big\{x\in\mathcal{H}_i: x\cdot e_i<\xi_{i,t_k,1}\big\}\big)\subset\Omega^-_{t_k}$ for every $k\in\N$ by~\eqref{interfaces} and~\eqref{eq5.30} (remember especially that $\mathcal{H}_i\cap\Omega\subset\Omega^-_{t_k}$ and $\xi_{i,t_k,1}=+\infty$ if $\mathcal{H}_i\cap\Gamma_{t_k}=\emptyset$). Having in mind~\eqref{xiit1}, set for each $k\in\N$
$$\rho_k=\min_{1\le i\le m}\xi_{i,t_k,1}-M_\delta>L_\epsilon+3R_\epsilon+r_{M_\delta+\overline{R_\epsilon}}>R_\epsilon+L$$
(notice that $\rho_k$ is a real number since there is at least one of integer $i$ such that $\mathcal{H}_i\cap\Gamma_{t_k}\neq\emptyset$ and then $\xi_{i,t_k,1}\in\R$). Therefore, $u(t_k,y)\le \delta$ for all $y\in\overline{\Omega}\cap\big(\overline{B(0,L)}\cup\cup_{i=1}^m\big\{y\in\overline{\mathcal{H}_i}: y\cdot e_i\le \rho_k\big\}\big)$. The comparison principle then implies that $u(t_k+t,y)\le \widetilde{w}_{\rho_k}(t,y)$ for all $t\ge 0$ and $y\in\overline\Omega$, with the notation~\eqref{defwtildeR}. Notice that~\eqref{eq:5.70} also yields $s_k-t_k\le (\rho_k-R_{\varepsilon}-L)/(c_f+\varepsilon)$ for all large~$k$. By Lemma~\ref{lemma5.3+}, it then follows that, for all large $k$,
\be\label{eq:5.72}
u(s_k,y)\le3\delta\ \text{for all $y\in\overline{\Omega}\cap\!\Big(\overline{B(0,L)}\cup\bigcup_{i=1}^m\big\{y\in\overline{\mathcal{H}_i}: y\cdot e_i\!\le\!\rho_k\!-\!R_{\varepsilon}\!-\!(c_f\!+\!\varepsilon)(s_k\!-\!t_k)\big\}\!\Big)$}.
\ee

Let, for each $k\in\N$, $i_k\in\{1,\cdots,m\}$ be an integer such that $\rho_k=\xi_{i_k,t_k,1}-M_{\delta}$. By \eqref{eq:5.70-}, for each $k$ large enough, there is $z_k\in\Gamma_{s_k}\cap\mathcal{H}_{i_k}$ such that $L_{\varepsilon}+3R_{\varepsilon}+M_{\delta}+r_{M_\delta+\overline{R_\epsilon}}<z_k\cdot e_{i_k}<\xi_{i_k,t_k,1}$ and, by \eqref{eq1.4}-\eqref{eq1.5}, there is $y'_{k}\in\Omega^+_{s_k}$ such that $d_{\Omega}(z_k,y'_k)\le r_{M_{\delta}}\,(\le r_{M_\delta+\overline{R_\epsilon}})$ and $d_{\Omega}(y'_{k},\Gamma_{s_k})\ge M_{\delta}$. In particular, there holds
\be\label{eq:5.73}
u(s_k,y'_k)\ge 1-\delta.
\ee
On the other hand, by \eqref{3.6}, one has $L+r_{M_\delta}\le L_{\varepsilon}+r_{M_\delta+\overline{R_\epsilon}}<z_k\cdot e_{i_k}\le \xi_{i_k,t_k,1}-(c_f+3\varepsilon/2)(s_k-t_k)$ for all $k$ large enough, while $d_{\Omega}(z_k,y'_k)\le r_{M_{\delta}}$ implies that $y'_k\in\mathcal{H}_{i_k}$ and
$$y'_k\cdot e_i\le z_k\cdot e_{i_k}+r_{M_{\delta}}\le \xi_{i_k,t_k,1}-\Big(c_f+\frac{3\varepsilon}{2}\Big)(s_k-t_k)+r_{M_{\delta}}\le\rho_k-R_\epsilon-(c_f+\varepsilon)(s_k-t_k)$$
for all large $k$. By~\eqref{eq:5.72}, one then gets that $u(s_k,y'_k)\le 3\delta$, which contradicts~\eqref{eq:5.73}. This completes the proof of the claim \eqref{eq:5.68-}.
\vskip 0.2cm

Secondly, we claim that
\be\label{e+2}
\limsup_{T^2_{\varepsilon}\le t<s,\ |t-s|\rightarrow +\infty} \frac{d(\Gamma_{t},\Gamma_{s})}{|t-s|}\le c_f+2\varepsilon.
\ee
Assume by contradiction that there exist two sequences $(t_k)_{k\in\mathbb{N}}$ and $(s_k)_{k\in\mathbb{N}}$ such that $T^2_{\varepsilon}\le t_k<s_k$, $s_k-t_k\rightarrow +\infty$ as $k\rightarrow +\infty$ and
$$d(\Gamma_{t_k},\Gamma_{s_k})>(c_f+2\varepsilon)(s_k-t_k) \text{ for all $k\in\mathbb{N}$}.$$
As above, \eqref{eq5.30} implies that, for each $k\in\N$, there is $i_k\in \{1,\cdots,m\}$ such that $\Gamma_{s_k}\cap \mathcal{H}_{i_k}\neq \emptyset$ (otherwise, $\Omega^+_{s_k}$ would be the whole domain $\Omega$). We then claim that
\be\label{xiiksk1}
\Gamma_{t_k}\cap\big\{x\in\mathcal{H}_{i_k}: L_{\varepsilon}+3R_{\varepsilon}+M_{\delta}+r_{M_{\delta}+ \overline{R}_{\varepsilon}}<x\cdot e_{i_k}<\xi_{{i_k},s_k,1}\big\}\neq \emptyset.
\ee
Assume not. Then $(\Omega\cap B(0,D_{\varepsilon}))\cup\big\{x\in\mathcal{H}_{i_k}: L_{\varepsilon}+3R_{\varepsilon}+M_{\delta}+r_{M_{\delta}+ \overline{R}_{\varepsilon}}<x\cdot e_{i_k}<\xi_{{i_k},s_k,1}\big\}\subset \Omega^+_{t_k}$ by~\eqref{eq5.29}-\eqref{eq5.30}, hence $u(t_k,x)\ge 1-\delta$ for all $x\in\overline{\mathcal{H}_{i_k}}$ with $\xi_{{i_k},s_k,1}-M_{\delta}-2R_{\varepsilon}\le x\cdot e_{i_k}\le \xi_{{i_k},s_k,1}-M_{\delta}$ (notice also that $\xi_{i_k,s_k,1}-M_\delta-2R_\epsilon\ge L_\epsilon>L$ by~\eqref{xiit1}). The comparison principle yields
$$u(t_k+t,x)\ge v_{i_k,\xi_{i_k,s_k,1}-M_{\delta}-R_\epsilon,R_{\varepsilon}}(t,x)\ \text{ for all $t\ge 0$ and $x\in\overline{\Omega}$}.$$
Since $\xi_{i_k,s_k,1}-r_{M_\delta}>L_\epsilon+3R_\epsilon+M_\delta+r_{M_\delta+\overline{R_\epsilon}}-r_{M_\delta}\ge R_\epsilon+L_\epsilon>L$ by~\eqref{xiit1} and since $s_k-t_k\rightarrow +\infty$ as $k\to+\infty$, Lemma \ref{lemma3.3} implies in particular that, for all large $k$,
\be\label{eq:5.75}
u(s_k,x)\ge1-3\delta\ \  \text{for all $x\in\overline{\mathcal{H}_{i_k}}$ such that $\xi_{i_k,s_k,1}-r_{M_{\delta}}\le x\cdot e_{i_k}\le \xi_{{i_k},s_k,1}+r_{M_{\delta}}$}.
\ee
On the other hand, for $z_k\in\Gamma_{s_k}\cap\mathcal{H}_{i_k}$ such that $z_k\cdot e_{i_k}=\xi_{{i_k},s_k,1}$, there exists $y'_k\in\Omega^-_{s_k}$ such that $d_{\Omega}(z_k,y'_k)\le r_{M_{\delta}}$ and  $d(y'_k,\Gamma_{s_k})\ge M_{\delta}$. Thus, $u(s_k,y'_k)\le \delta$, whereas $|y'_k\cdot e_{i_k}-\xi_{{i_k},s_k,1}|\le r_{M_{\delta}}$ and $y'_k\in\mathcal{H}_{i_k}$ since $\xi_{i_k,s_k,1}-r_{M_\delta}>L$. That contradicts~\eqref{eq:5.75}. Therefore,~\eqref{xiiksk1} holds.

Since $d(\Gamma_{t_k},\Gamma_{s_k})>(c_f+2\varepsilon)(s_k-t_k)$, property~\eqref{xiiksk1} implies that
\be\label{xiiksk1bis}
\xi_{{i_k},s_k,1}-(c_f+2\varepsilon)(s_k-t_k)>L_{\varepsilon}+3R_{\varepsilon}+M_{\delta}+r_{M_{\delta}+ \overline{R_{\varepsilon}}}
\ee
and either $E_k:=\big\{x\in\mathcal{H}_{i_k}: |x\cdot e_{i_k}-\xi_{{i_k},s_k,1}|\le (c_f+2\varepsilon)(s_k-t_k)\big\}\subset \Omega^+_{t_k}$ or $E_k\subset\Omega^-_{t_k}$. In the former case, one has $u(t_k,x)\ge1-\delta$ for large $k$ and for all $x\in\overline{\mathcal{H}_{i_k}}$ with $\xi_{{i_k},s_k,1}-M_\delta-2R_\epsilon\le x\cdot e_{i_k}\le\xi_{{i_k},s_k,1}-M_\delta$ and one gets a contradiction as in the previous paragraph. Consider now the latter case, namely $E_k\subset\Omega^-_{t_k}$. Hence $u(t_k,x)\le \delta$ for all $x\in\overline{\mathcal{H}_{i_k}}$ such that $|x\cdot e_{i_k}-\xi_{{i_k},s_k,1}|\le (c_f+2\varepsilon)(s_k-t_k)-M_{\delta}$. Notice that
$$\rho_k:=(c_f+2\epsilon)(s_k-t_k)-M_\delta>R_\epsilon$$
for large $k$, that $\xi_{i_k,s_k,1}\ge\rho_k+L_\epsilon$ by~\eqref{xiiksk1bis}, and that $s_k-t_k\le(\rho_k-R_\epsilon)/(c_f+\epsilon)$ for large $k$. Therefore, the comparison principle and Lemma~\ref{lemma5.2} applied with $R=\rho_k$ and $l=\xi_{i_k,s_k,1}$ imply that, for all $k$ large enough,
\be\label{eq:5.76}\baa{l}
u(s_k,x)\le w_{i_k,\xi_{i_k,s_k,1},\rho_k}(s_k-t_k,x)\le2\delta\vspace{3pt}\\
\qquad\qquad\qquad\qquad\text{for all $x\in\overline{\mathcal{H}_{i_k}}$ with $|x\cdot e_{i_k}-\xi_{{i_k},s_k,1}|\le \varepsilon(s_k-t_k)-M_{\delta}-R_{\varepsilon}$}.\eaa
\ee
However, for $z_k\in\Gamma_{s_k}\cap\mathcal{H}_{i_k}$ such that $z_k\cdot e_{i_k}=\xi_{{i_k},s_k,1}$, there exists $y'_k\in\Omega^+_{s_k}$ such that $d_{\Omega}(z_k,y'_k)\le r_{M_{\delta}}$ and $d(y'_k,\Gamma_{s_k})\ge M_{\delta}$. Hence $u(s_k,y'_k)\ge 1-\delta$, whereas $|y_k'\cdot e_{i_k}-\xi_{{i_k},s_k,1}|\le r_{M_{\delta}}\le \varepsilon(s_k-t_k)-M_{\delta}-R_{\varepsilon}$ for large $k$ and $y'_k\in\mathcal{H}_{i_k}$ since $\xi_{i_k,s_k,1}-r_{M_\delta}>L$. That contradicts \eqref{eq:5.76}. As a consequence, both cases $E_k\subset\Omega^+_{t_k}$ and $E_k\subset\Omega^-_{t_k}$ are ruled out and the proof of the claim \eqref{e+2} is complete.
\vskip 0.2cm

Thirdly, we prove that
\be\label{limsup5}
\limsup_{|t-s|\rightarrow +\infty} \frac{d_{\Omega}(\Gamma_t,\Gamma_s)}{|t-s|}\le c_f+3\varepsilon.
\ee
Assume by contradiction that there exist two sequences $(t_k)_{k\in\mathbb{N}}$ and $(s_k)_{k\in\mathbb{N}}$ such that $t_k<s_k$, $s_k-t_k\rightarrow +\infty$ as $k\rightarrow +\infty$ and
\be\label{eq:5.77}
d_{\Omega}(\Gamma_{t_k},\Gamma_{s_k})>(c_f+3\varepsilon)(s_k-t_k) \text{ for all $k\in\mathbb{N}$}.
\ee
Then, six cases may occur up to extraction of a subsequence, namely, case~1: $t_k<s_k\le T_{\varepsilon}^1$ for all~$k\in\N$; case~2: $t_k< T_{\varepsilon}^1< s_k< T_{\varepsilon}^2$ for all $k\in\N$; case~3: $t_k< T_{\varepsilon}^1<T_{\varepsilon}^2\le s_k$ for all $k\in\N$; case~4: $ T_{\varepsilon}^1 \le t_k<s_k\le T_{\varepsilon}^2$ for all $k\in\N$; case~5: $T_{\varepsilon}^1\le t_k\le T_{\varepsilon}^2< s_k$ for all $k\in\N$; case~6:~$T_{\varepsilon}^2< t_k<s_k$ for all $k\in\N$. In fact, we have already shown case~1 and case~6 are impossible for all~$k$ large enough by \eqref{eq:5.68-} and \eqref{e+2}. Case~4 is clearly not true since $s_k-t_k\to+\infty$ as~$k\to+\infty$.

Now we consider case~2. In this case, one has $t_k\to-\infty$ as $k\to+\infty$ and it follows from~\eqref{eq:5.68-} that $d_{\Omega}(\Gamma_{t_k},\Gamma_{T_{\varepsilon}^1})\le (c_f+5\varepsilon/2) (T_{\varepsilon}^1-t_k)$ for large $k$. Since $|s_k-T^1_\epsilon|\le T^2_\epsilon-T^1_\epsilon$ for all $k$, property~\eqref{eq+2.19} yields the existence of a constant $M>0$ such that $d_{\Omega}(\Gamma_{t_k},\Gamma_{s_k})\le d_{\Omega}(\Gamma_{t_k},\Gamma_{T_{\varepsilon}^1})+M$ for all $k\in\N$, hence $d_{\Omega}(\Gamma_{t_k},\Gamma_{s_k})\le(c_f+3\varepsilon) (s_k-t_k)$ for large~$k$, a contradiction with~\eqref{eq:5.77}. Similarly, one can reach a contradiction with~\eqref{eq:5.77} in case~5.

It remains to handle case~3. We can assume without loss of generality that $t_k\rightarrow -\infty$ and~$s_k\rightarrow+\infty$ as $k\rightarrow+\infty$ (otherwise, by decreasing $T_{\varepsilon}^1$ and increasing $T_{\varepsilon}^2$, case~3 can be reduced to case~2 or case~5). By~\eqref{eq:5.68-} and~\eqref{e+2}, there exist some sequences $(x_k)_{k\in\N}$, $(x'_k)_{k\in\N}$, $(z_k)_{k\in\N}$ and $(z'_k)_{k\in\N}$ such that $x_k\in\Gamma_{t_k}$, $x'_k\in\Gamma_{T^1_\epsilon}$, $z'_k\in\Gamma_{T^2_\epsilon}$, $z_k\in\Gamma_{s_k}$ together with~$d_\Omega(x_k,x'_k)\le(c_f+5\epsilon/2)(T^1_\epsilon-t_k)$ and $d_\Omega(z_k,z'_k)\le(c_f+5\epsilon/2)(s_k-T^2_\epsilon)$. On the other hand, remembering~\eqref{interfaces}, both $\Gamma_{T^1_\epsilon}$ and $\Gamma_{T^2_\epsilon}$ are bounded. Hence, there is a constant $M>0$ such that $d_{\Omega}(x'_k,z'_k)\le M$ for all $k\in\N$. Finally,
$$\baa{rcl}
d_{\Omega}(\Gamma_{t_k},\Gamma_{s_k})\le d_{\Omega}(x_k,z_k) & \le & d_{\Omega}(x_k,x'_k)+d_{\Omega}(x'_k,z'_k)+d_{\Omega}(z'_k,z_k)\vspace{3pt}\\
& \le & \displaystyle\Big(c_f+\frac{5\epsilon}{2}\Big)(T^1_\epsilon-t_k)+M+\Big(c_f+\frac{5\epsilon}{2}\Big)(s_k-T^2_\epsilon)\vspace{3pt}\\
& \le & (c_f+3\epsilon)(s_k-t_k)\eaa$$
for all large $k$, contradicting~\eqref{eq:5.77}.

As a conclusion,~\eqref{limsup5} holds for every $\epsilon>0$ small enough, implying the desired upper estimate~\eqref{5.30}. Together with~\eqref{5.29}, the proof of Theorem~\ref{th6} is thereby complete.
\end{proof}


\subsection{Proof of Corollary \ref{th5}}

Before going further, we need a lemma ensuring the complete propagation of any front-like solution $u_i$ coming from the branch $\mathcal{H}_i$ under the assumptions on $\Omega$ in Corollary~\ref{th5}.

\begin{lemma}\label{lemma5.3}
There exist $R>0$ and $\mu\in(0,1)$ such that the following holds. For any smooth domain $\Omega$ satisfying the conditions of Corollary~$\ref{th5}$ and for any solution $v$ of the Cauchy problem
\be\label{Cauchy}\left\{\baa{rcll}
v_t & = & \Delta v+f(v), &\text{ for $t>0$, $x\in\overline{\Omega}$},\vspace{3pt}\\
v_\nu & = & 0, &\text{ for $t>0$, $x\in\partial\Omega$},\vspace{3pt}\\
v(0,x) & = & v_0(x)\in[0,1], &\text{ for $x\in\overline{\Omega}$},\eaa\right.
\ee
if there are $i\neq j\in\{1,\cdots,m\}$ and $s\in\R$ such that $v_0\ge1-\mu$ in $B(P_{i,j}(s),R)$, then $v(t,\cdot)\to1$ as $t\to+\infty$ locally uniformly in $\overline{\Omega}$.
\end{lemma}

\begin{proof}
Let $\theta_2\in(0,1)$ be defined as in~\eqref{deftheta12}. From the proof of Lemma~\ref{lemmap}, there exist $R>0$ and a radially decreasing solution $\psi\in C^2(\overline{B(0,R)})$ of~\eqref{eqphi}. Denote
$$\mu=1-\max_{\overline{B(0,R)}}\psi=1-\psi(0)\in(0,1-\theta_2)\subset(0,1)$$
and let us show that Lemma~\ref{lemma5.3} holds with these values of $R>0$ and $\mu\in(0,1)$. Let also~$\overline{\psi}\in C(\R^N)$ be the function defined by $\overline{\psi}=\psi$ in $\overline{B(0,R)}$ and $\overline{\psi}=0$ in $\R^N\setminus\overline{B(0,R)}$.

Consider then any smooth domain $\Omega$ satisfying the conditions of Corollary~\ref{th5} and let $v$ be given as in Lemma~\ref{lemma5.3}. Let $w$ be the solution of the Cauchy problem~\eqref{Cauchy} with initial condition $w(0,x)=\overline{\psi}(x-P_{i,j}(s))$ for $x\in\overline{\Omega}$. Remember also that $B(P_{i,j}(s),R)\subset\Omega$ from the assumptions made on $\Omega$. As in the proof of Lemma~\ref{lemmap}, the maximum principle then implies that $w$ is increasing with respect to $t$ and, from standard parabolic estimates, there is a solution~$w_\infty\in C^2(\overline{\Omega})$ of
$$\left\{\baa{rcl}
\Delta w_{\infty}+f(w_{\infty}) & = & 0\ \text{ in $\overline{\Omega}$},\vspace{3pt}\\
w_\infty & > & 0\ \text{ in $\overline{\Omega}$},\vspace{3pt}\\
(w_{\infty})_\nu & = & 0\ \text{ on $\partial\Omega$,}\eaa\right.$$
such that $w(t,\cdot)\to w_\infty$ in $C^2_{loc}(\overline{\Omega})$ as $t\to+\infty$. Furthermore, $1\ge v(0,\cdot)\ge w(0,\cdot)$ in $\overline{\Omega}$ owing to the definition of $w(0,\cdot)$ and since $1-\mu\ge\psi$ in $\overline{B(0,R)}$. Hence, $0<w(t,x)\le v(t,x)\le1$ for all~$t>0$ and $x\in\overline{\Omega}$. In order to complete the proof of Lemma~\ref{lemma5.3}, it is therefore sufficient to show that $w_\infty\equiv 1$ in $\overline{\Omega}$. 

To show that $w_\infty\equiv1$ in $\overline{\Omega}$, notice first that $w_\infty>\psi(\cdot-P_{i,j}(s))$ in $\overline{B(P_{i,j}(s),R)}\,(\subset\overline{\Omega})$ since~$w$ is increasing in $t$ and both $w_\infty$ and $\psi(\cdot-P_{i,j}(s))$ are continuous in $\overline{B(P_{i,j}(s),R)}$. As in the proof of Lemma~\ref{lemmap}, the sliding method and the strong maximum principle then imply that
$$w_\infty>\psi(\cdot-P_{i,j}(s'))\ \hbox{ in }\overline{B(P_{i,j}(s'),R)}\,(\subset\overline{\Omega})$$
for all $s'\in\R$. From the assumptions made on the paths $P_{i',j'}$ before Corollary~\ref{th5}, it follows from the same arguments that $w_\infty>\psi(\cdot-P_{i,j'}(s'))$ in $\overline{B(P_{i,j'}(s'),R)}$ for all $j'\in\{1,\cdots,m\}\setminus\{i\}$ and for all $s'\in\R$ and finally $w_\infty>\psi(\cdot-P_{i',j'}(s'))$ in $\overline{B(P_{i',j'}(s'),R)}$ for all $i'\neq j'\in\{1,\cdots,m\}$ and for all $s'\in\R$. In particular, one has $w_\infty>\psi$ in $\overline{B(0,R)}$ since by assumption $0\in\cup_{1\le i'\neq j'\le m}P_{i',j'}(\R)$. 

By using now that $\Omega$ is star-shaped with respect to $0$, one then infers from the sliding method exactly as in the proof of~\cite[Theorem~1.12]{BBC} that $w_\infty>\max_{e\in\mathbb{S}^{N-1}}\overline{\psi}(\cdot-he)$ in $\overline{\Omega}$ for all $h\ge0$. As a consequence, $1\ge w_\infty>\psi(0)>\theta_2$ in $\overline{\Omega}$ and one concludes that $w_\infty\equiv 1$ in $\overline{\Omega}$ with the same arguments as for $p_\infty$ at the end of the proof of Lemma~\ref{lemmap}.

The proof of Lemma~\ref{lemma5.3} is thereby complete.
\end{proof}
\vskip 0.3cm

\begin{proof}[Proof of Corollary~\ref{th5}]
Let $R>0$ and $\mu\in(0,1)$ be as in Lemma~\ref{lemma5.3} and let $\Omega$ be any smooth domain with $m\,(\ge2)$ cylindrical branches satisfying the conditions of Corollary~\ref{th5}. Let $i$ be any integer in $\{1,\cdots,m\}$ and let $u_i:\R\times\overline{\Omega}\to(0,1)$ be the time-increasing front-like solution of~\eqref{eq1.1} emanating from the branch $\mathcal{H}_i$, that is, $u_i$ satisfies~\eqref{frontlikei}. In particular, since $\phi(-\infty)=1$, there are $t_0\in\R$ and $s\in\R$ such that
$$u_i(t_0,\cdot)\ge1-\mu\ \hbox{ in }\overline{B(P_{i,j}(s),R)}$$
for all $j\in\{1,\cdots,m\}\setminus\{i\}$. Lemma~\ref{lemma5.3} then implies that $u_i(t,\cdot)\to1$ as $t\to+\infty$ locally uniformly in $\overline{\Omega}$, that is, $u_i$ propagates completely. Corollary~\ref{th5} then follows from Theorems~\ref{th10} and~\ref{th6}.
\end{proof}


\subsection{Proof of Corollary~\ref{cor4}}

Let $\Omega$ be a smooth domain with $m\,(\ge2)$ cylindrical branches in the sense of~\eqref{branches}. Notice first that, from standard parabolic estimates, for each $i\in\{1,\cdots,m\}$, the time-increasing front-like solution $u_i:\R\times\overline{\Omega}\to(0,1)$ of~\eqref{eq1.1} emanating from the branch $\mathcal{H}_i$ (that is, $u_i$ satisfies~\eqref{frontlikei}) converges to a $C^2(\overline{\Omega})$ solution $p_i:\overline{\Omega}\to(0,1]$ of
\be\label{eqpi}\left\{\baa{rcl}
\Delta p_i+f(p_i) & = & 0\ \hbox{ in }\overline{\Omega},\vspace{3pt}\\
(p_i)_\nu & = & 0\ \hbox{ on }\partial\Omega,\vspace{3pt}\\
p_i(x) & \to & 1\ \hbox{ as }|x|\to+\infty\hbox{ with }x\in\overline{\mathcal{H}_i}.\eaa\right.
\ee
From Theorems~\ref{th10} and~\ref{th6}, the proof of Corollary~\ref{cor4} will be complete once one shows that there is $R_0>0$ such that, for any $R\ge R_0$, for any $x_0\in\R^N$ and for any $i\in\{1,\cdots,m\}$, any $C^2(\overline{R\Omega+x_0})$ solution $p_i:\overline{R\Omega+x_0}\to(0,1]$ of~\eqref{eqpi} (with $R\Omega+x_0$ instead of $\Omega$ and $R\mathcal{H}_i+x_0$ instead of $\mathcal{H}_i$) satisfies $p_i\equiv1$ in $\overline{R\Omega+x_0}$.

To do so, assume by way of contradiction that the conclusion does not hold. Then there are an integer $i\in\{1,\cdots,m\}$, a sequence $(R_k)_{k\in\N}$ of positive real numbers converging to $+\infty$, a sequence $(x_k)_{k\in\N}$ in $\R^N$ and a sequence $(p_{i,k})_{k\in\N}$ of classical solutions $p_{i,k}:\overline{R_k\Omega+x_k}\to(0,1]$ of
$$\left\{\baa{rcl}
\Delta p_{i,k}+f(p_{i,k}) & = & 0\ \hbox{ in }\overline{R_k\Omega+x_k},\vspace{3pt}\\
(p_{i,k})_\nu & = & 0\ \hbox{ on }\partial(R_k\Omega+x_k),\vspace{3pt}\\
p_{i,k}(x) & \to & 1\ \hbox{ as }|x|\to+\infty\hbox{ with }x\in\overline{R_k\mathcal{H}_i+x_k},\eaa\right.$$
such that $p_{i,k}\not\equiv1$ in $\overline{R_k\Omega+x_k}$ for each $k\in\N$.

Now, from~\eqref{branches} and the smoothness of $\Omega$, there is $r>0$ such that, for any $y\in\overline{\Omega}$ and for any $r'\in(0,r]$, there is a continuous path $\gamma:[0,+\infty)\to\Omega$ such that
$$y\in\overline{B(\gamma(0),r')},\ B(\gamma(s),r')\subset\Omega\hbox{ for all }s\ge0,\ \lim_{s\to+\infty}|\gamma(s)|=+\infty\hbox{ and }\gamma(s)\in\mathcal{H}_i\text{ for $s$ large}.$$
Furthermore, if $B(y,r)\subset\Omega$, one can take $\gamma(0)=y$. Consider also $R>0$ and a solution $\psi$ of~\eqref{eqphi}, as given in the proof of Lemma~\ref{lemmap}. Observe that $r$ and $R$ are independent of $k$.

Take any $k\in\N$ large enough so that $r\,R_k\ge R$ and consider any point $y\in\overline{R_k\Omega+x_k}$ such that $B(y,R)\subset R_k\Omega+x_k$. There is then a continuous path $\gamma:[0,+\infty)\to\,R_k\Omega+x_k$ such that~$\gamma(0)=y$,~$B(\gamma(s),R)\subset R_k\Omega+x_k$ for all $s\ge0$, $\lim_{s\to+\infty}|\gamma(s)|=+\infty$ and $\gamma(s)\in R_k\mathcal{H}_i+x_k$ for all $s$ large enough. Since $\lim_{|x|\to+\infty,\,x\in\overline{R_k\mathcal{H}_i+x_k}}p_{i,k}(x)=1$ and $\max_{\overline{B(0,R)}}\psi=\psi(0)<1$, there holds~$p_{i,k}>\psi(\cdot-\gamma(s))$ in $\overline{B(\gamma(s),R)}$ for all $s$ large enough. Since $p_{i,k}>0$ in $\overline{\Omega}$ and~$\psi=0$ on $\partial B(0,R)$, the same type of sliding method as in the proof of Lemma~\ref{lemmap} then implies that~$p_{i,k}>\psi(\cdot-\gamma(s))$ in $\overline{B(\gamma(s),R))}$ for all $s\ge0$. As a consequence,
\be\label{pik}
p_{i,k}(y)\ge\psi(y-\gamma(0))=\psi(0)>\theta_2\ \hbox{ for all }y\in\overline{R_k\Omega+x_k}\hbox{ such that }B(y,R)\subset R_k\Omega+x_k.
\ee

On the other hand, since $f>0$ on $(\theta_2,1)$ and since $p_{i,k}\not\equiv1$ in $\overline{R_k\Omega+x_k}$, the strong maximum principle and the Hopf lemma imply that $\inf_{\overline{R_k\Omega+x_k}}p_{i,k}\le\theta_2$. Therefore, for each $k\in\N$, there is a point $y_k\in\overline{R_k\Omega+x_k}$ such that $p_{i,k}(y_k)<(\theta_2+\psi(0))/2$. Together with~\eqref{pik}, one infers that, for each $k\in\N$, there is $z_k\in\partial(R_k\Omega+x_k)$ such that~$|z_k-y_k|<R$. One is then led to a contradiction as in the last two paragraphs of the proof of Corollary~\ref{cor3} (up to replacing the inequalities $p_\infty(\zeta,0,\cdots,0)=p_\infty(X(\zeta))\le\theta_2<\psi(0)$ with~$p_\infty(\zeta,0,\cdots,0)=p_\infty(X(\zeta))\le(\theta_2+\psi(0))/2<\psi(0)$). As a conclusion, the existence of the sequences $(R_k)_{k\in\N}$, $(x_k)_{k\in\N}$ and $(p_{i,k})_{k\in\N}$ together with the integer $i$ is thus ruled out and the proof of Corollary~\ref{cor4} is complete.\hfill$\Box$


\end{document}